\documentclass[a4paper,12pt]{article}
\usepackage{amssymb,enumerate}
\usepackage{hyperref}
\usepackage{authblk}
\usepackage{amsthm}
\usepackage{indentfirst}
\usepackage{paralist}
\usepackage{amsmath}
\usepackage{txfonts}
\usepackage{mathrsfs}
\usepackage{mathdots}
\usepackage{appendix}
\usepackage[hang,flushmargin]{footmisc}
\usepackage{geometry}
\newcommand{\F}{{\mathbb{F}}}

\newcommand{\bG} {\mathbf G}

\newcommand{\bL} {\mathbf L}

\newcommand{\bc}{\mathbf{c}}
\newcommand{\cE} {\mathcal E}
\newcommand{\cO} {\mathcal O}
\newcommand{\cW} {\mathcal W}
\newcommand{\cF} {\mathcal F}
\newcommand{\cC} {\mathcal C}
\newcommand{\fS} {\mathfrak S}
\newcommand{\sC}{\mathscr C}
\newcommand{\fb}{\mathfrak{b}}
\newcommand{\cR}{\mathcal{R}}

\newcommand{\Aut}{{\operatorname{Aut}}}

\newcommand{\Irr}{{\operatorname{Irr}}}
\newcommand{\IBr}{{\operatorname{IBr}}}
\newcommand{\Res}{{\operatorname{Res}}}
\newcommand{\Ind}{{\operatorname{Ind}}}
\newcommand{\Rad}{{\operatorname{Rad}}}
\newcommand{\GL}{\operatorname{GL}}
\newcommand{\GU}{\operatorname{GU}}
\newcommand{\SL}{\operatorname{SL}}
\newcommand{\PSL}{\operatorname{PSL}}

\newcommand{\Bl}{\operatorname{Bl}}
\newcommand{\PSU}{\operatorname{PSU}}
\newcommand{\Sp}{\operatorname{Sp}}
\newcommand{\CSp}{\operatorname{CSp}}

\newcommand{\SO}{\operatorname{SO}}
\newcommand{\CSO}{\operatorname{CSO}}
\newcommand{\GO}{\operatorname{GO}}
\newcommand{\CO}{\operatorname{CO}}

\newcommand{\bl}{\operatorname{bl}}

\newcommand{\tD}{\tilde D}
\newcommand{\tG}{{\tilde G}}

\let\eps=\epsilon
\let\vare=\varepsilon
\let\la=\lambda
\let\ti=\times

\theoremstyle{theorem}
\newtheorem{thm}{Theorem}[section]
\newtheorem{lem}[thm]{Lemma}
\newtheorem{prop}[thm]{Proposition}
\newtheorem{cor}[thm]{Corollary}
\newtheorem{conj}[thm]{Conjecture}

\theoremstyle{definition}
\newtheorem{defn}[thm]{Definition}
\newtheorem{rmk}[thm]{Remark}

\numberwithin{equation}{section}

\title{On the inductive blockwise Alperin weight condition for classical groups}

\author[a,b]{Zhicheng Feng}
\author[c]{Zhenye Li}
\author[d]{Jiping Zhang}

\affil[a]{School of Mathematics and Physics, University of Science and Technology Beijing, Beijing 100083, China}

\affil[b]{FB Mathematik, TU Kaiserslautern, Postfach 3049, 67653 Kaiserslautern,	Germany}

\affil[c]{Institute of Mathematics, Academy of Mathematics and Systems Science, Chinese Academy of Sciences, Beijing 100190, China}

\affil[d]{Beijing International Center for Mathematical Research Lmam, The School of
	Mathematical Sciences, Peking University, Beijing 100871,  China}

\providecommand{\subjclass}{\textbf{\textit{2010 Mathematics Subject Classification   }}}

\begin{document}
	
\date{}

\maketitle

\let\thefootnote\relax\footnotetext{Supported by SFB TRR 195 and NSFC (No. 11631001).}

\let\thefootnote\relax\footnotetext{
	\emph{Email addresses}: \leftline{zfeng@pku.edu.cn~(Z. Feng) , lizhenye@pku.edu.cn~(Z. Li), jzhang@pku.edu.cn~(J. Zhang).}}

\begin{abstract}
Recently, there has been substantial progress on the Alperin weight conjecture.
As a step to establish the Alperin weight conjecture for all finite groups, we prove the inductive blockwise Alperin weight condition for simple groups of  classical type under some additional assumption.
\end{abstract}



\section{Introduction}

On the 1986 Arcata conference on representations of finite groups, J. L. Alperin put forward his famous conjecture, which is now called
the \emph{Alperin weight conjecture}. To state it, let $G$ be a finite group and $\ell$ a prime, $B$ an $\ell$-block of $G$. As usual, we
denote by ${\rm Irr}(B)$ and  ${\rm IBr}_\ell(B)$ the sets of ordinary irreducible characters and
irreducible $\ell$-Brauer characters of $B$ respectively.
For an $\ell$-subgroup $R$ of $G$ and $\varphi\in\Irr(N_G(R))$, the pair $(R,\varphi)$ is called an \emph{$\ell$-weight}
if $R\subseteq\ker\varphi$ is of $\ell$-defect zero viewed as a character of $N_G(R)/R$.
Note that $R$ is necessarily an $\ell$-radical subgroup of $G$ for any $\ell$-weight $(R,\varphi)$.
An $\ell$-weight $(R,\varphi)$ is called a \emph{$B$-weight} if $\bl_\ell(\varphi)^G=B$, where $\bl_\ell(\varphi)$ is the $\ell$-block of $N_G(R)$ containing $\varphi$.
We denote  by   $\cW_\ell(B)$   the set of all $G$-conjugacy classes of $B$-weights so that
the Alperin weight conjecture can be stated as follows.

\begin{conj}[Alperin, \cite{Al87}]
	\label{weiconj}
Let $G$ be a finite group, $\ell$ a prime.
If $B$ is an $\ell$-block of $G$, then
$$|\cW_\ell(B)|=|\IBr_\ell(B)| .$$
\end{conj}

So far substantial progress has been achieved for the Alperin weight conjecture. 
Specifically,
it was shown to hold for finite groups of Lie type in defining characteristic by Cabanes \cite{Ca88},
for symmetric groups and general linear groups by Alperin and Fong \cite{AF90},
and for certain groups of classical type by An \cite{An93} and \cite{An94}.

Even though the Alperin weight conjecture was subsequently checked for several further families of finite groups, 
it has not been possible so far to find a general proof for arbitrary finite groups.
A reduction theorem for the blockfree version of the Alperin weight conjecture was obtained by Navarro and Tiep \cite{NT11} in 2011.
Soon afterwards, Sp\"ath \cite{Sp13} refined this result to achieve a reduction theorem for the blockwise version of Alperin's weight conjecture;
if all finite (quasi-)simple groups satisfy the so-called \emph{inductive blockwise Alperin weight} (iBAW) condition, then the Alperin weight conjecture \ref{weiconj} holds for any finite group.

To the present, the (iBAW) condition has been verified for some cases,
such as  simple alternating groups,
many of the sporadic groups,
simple groups of Lie type in the defining characteristic,
Suzuki groups and Ree groups,
simple groups of type $G_2$ and $^3D_4$,
some cases of blocks of groups of type $A$;
see for instance \cite{Br}, \cite{BS19}, \cite{CS13}, \cite{Fe18}, \cite{LZ18},  \cite{Ma14}, \cite{Sc16} and \cite{Sp13}.
Unfortunately, it still seems a far way to deal with the case of simple groups of Lie type in non-defining characteristic in general.

In this paper, we consider the classical groups.
As a first step to verify the (iBAW) condition, we need to establish a blockwise equivariant bijection between $\ell$-Brauer characters and $\ell$-weights.
In fact, An \cite{An94} has essentially given such a bijection.
In this paper, we first consider the groups of type $B$ and prove the bijection given in \cite{An94} for $\SO_{2n+1}(q)$ is equivariant under the field automorphism~(see Theorem \ref{equiv-bij-so}).
From this, we obtain a blockwise $\Aut(S)$-equivariant bijection between $\IBr_\ell(S)$ and $\cW_\ell(S)$ under some assumption~(see Theorem \ref{bi-simple-b}), where $S=\Omega_{2n+1}(q)$.
Our first main result about non-faithful blocks of groups of type $B$ is the following.

\begin{thm}\label{ibawc-dominate}
	Let $X=\mathrm{Spin}_{2n+1}(q)$ with $q=p^f$ odd  and $n\ge2$,  $\ell\ne p$ an odd prime and $B$ an $\ell$-block of $X$.	
	Assume further $f$ is odd, $\ell$ is linear and
	$B$ dominates some $\ell$-block of $\Omega_{2n+1}(q)$.
	Then the inductive blockwise Alperin weight (iBAW)  condition~(cf. Definition \ref{induc})~ holds for $B$.
\end{thm}

Recall that an odd prime $\ell$ not dividing $q$ is called \emph{linear} (for $q$) if the multiplicative order of $q$ modulo $\ell$ is odd.
In this paper, the assumption that $\ell$ is linear is always to ensure the unitriangular shape of decomposition matrices, which is due to Gruber and Hiss  \cite{GH97}.
It is an open problem to show that decomposition matrices of finite groups of Lie type in non-defining characteristic have unitriangular shape~(see for example \cite[Problem 4.8]{Ma15}).
If this is true, then the assumption that $\ell$ is linear can be removed from the main results of this paper.

For groups of type $C$, we first verify the Alperin weight conjecture \ref{weiconj} for every $\ell$-block of $\Sp_{2n}(q)$ when both $\ell$ and $q$ are odd, and then we prove the (iBAW) condition for the simple group of symplectic type and a linear prime if the outer automorphism group is cyclic,  which can be stated as follows.

\begin{thm}\label{baw-sp}
Let $q$ be a power of an odd prime, $n\ge 2$, $\ell\ne p$ an odd prime.
Then the Alperin weight conjecture \ref{weiconj} holds for every $\ell$-block of $\Sp_{2n}(q)$.
\end{thm}

\begin{thm}\label{equiv-bij-c}
	Let $q=p^f$ be a power of an odd prime $p$,
	$\ell\ne p$ an odd prime and $n\ge 2$. Assume that $f$ is odd and $\ell$ is linear.
	Then the inductive blockwise Alperin weight  (iBAW) condition (cf. Definition \ref{ibaw-simple}) holds for the simple group $\mathrm{PSp}_{2n}(q)$ and the prime $\ell$.
\end{thm}

In order  to prove Theorem \ref{equiv-bij-c}, we need a parametrization of $\ell$-blocks of $\Sp_{2n}(q)$,
which may be of independent interest; see Theorem \ref{blocks-sp}.
In order to do this, we make use of both the parametrization of $\ell$-blocks of $\CSp_{2n}(q)$
from Fong--Srinivasan \cite{FS89} , and
the label of $\ell$-blocks of an arbitrary finite group of Lie type
from Cabanes--Enguehard \cite{CE99} for $\ell\ge 7$ and Kessar--Malle \cite{KM15} for the largest possible generality.
From this, we obtain a blockwise bijection between the irreducible $\ell$-Brauer characters and $\ell$-weights of $\Sp_{2n}(q)$ which is equivariant under the action of automorphisms~(see Theorem \ref{bij-sp}).

We should mention that Conghui Li has proved independently Theorem \ref{bij-sp}  in \cite{Li19} with different methods.

In addition, we also determined a similar parametrization of $\ell$-blocks for $\SO^\pm_{2n}(q)$, which is described in Appendix \ref{Appendix-blocks-so}.
From this, if a hypothesis for the action of  $\GO^\pm_{2n}(q)$ on the characters of  $\SO^\pm_{2n}(q)$ is true, then the Alperin weight conjecture \ref{weiconj} holds for every $\ell$-block of $\SO^\pm_{2n}(q)$ when both $\ell$ and $q$ are odd; see Theorem \ref{alp-conj-SO2n}.

Is there an analogue of Jordan decomposition for weights of finite groups of Lie type?
Malle proposed this problem in \cite[Problem~4.9]{Ma15}.
Furthermore, following Kessar--Malle \cite[\S4]{KM18}, we hope for a Bonnaf\'e--Rouquier type reduction~(cf. \cite{BR03}, see also \cite[\S3.3]{KM18}) to a few special situations, \emph{i.e.}, quasi-isolated blocks.
In this sense, the unipotent blocks would play a fundamental and important role when considering the (iBAW) condition for finite quasi-simple groups of Lie type.
In \cite{Fe18}, the author verified the (iBAW) condition for unipotent blocks of groups of type $A$, untwisted or twisted, under some additional assumption on the prime involved.
Considering classical type, the following is our main result for unipotent blocks.

\begin{thm}\label{ibawc-uni}
Assume that both $\ell$ and $q$ are odd such that $\ell\nmid q$. Suppose that one of the following holds.
\begin{enumerate}[(i)]
	\item $X\in\{\mathrm{Spin}_{2n+1}(q), \Sp_{2n}(q)\}$ with $n\ge2$.
	\item $X=\mathrm{Spin}^-_{2n}(q)$ with $n\ge4$.
	\item $X=\mathrm{Spin}^+_{2n}(q)$ with $n>4$ and $\ell$ is linear.
\end{enumerate}	
Then the inductive blockwise Alperin weight  (iBAW) condition~(cf. Definition \ref{induc})~ holds for every unipotent $\ell$-block of $X$.
\end{thm}

This paper is built up as follows. In section \S\ref{sec:Preliminaries}, we introduce the general notation and state the (iBAW) condition.
In section \S\ref{wei-spec-case}, the action of automorphisms on the weights of classical groups for a special case is considered.
Then we prove Theorem \ref{equiv-bij-so} and prove Theorem \ref{ibawc-uni} for type $B$ in section \S\ref{typeB}.
In section \S \ref{typeC}, we give a classification for blocks of symplectic groups and then prove Theorem \ref{equiv-bij-c} and prove Theorem \ref{ibawc-uni} for type $C$.
Finally, the (iBAW) condition for unipotent blocks of classical groups of type $D$ and $^2D$ are verified in  section \S\ref{typeD}.
In addition, we consider the Alperin weight conjecture for special orthogonal groups in even-dimension in Appendices.

\section{Preliminaries}   \label{sec:Preliminaries}

\subsection{General results}

Let $G$ be a finite group. Concerning the block and character theory of $G$ we mainly
follow the notation of \cite{Na98}, where for sets of $\ell$-Brauer characters or $\ell$-blocks we add a subscript to indicate the corresponding prime $\ell$~(\emph{e.g.} $\IBr_\ell(G)$, $\Bl_\ell(G)$).
We denote the restriction of $\chi\in\Irr(G)\cup\IBr_\ell(G)$ to some subgroup $H\le G$ by $\Res^G_H\chi$, while $\Ind^G_H\psi$ denotes the character induced from $\psi\in\Irr(H)\cup\IBr_\ell(H)$ to $G$.
For $N\unlhd G$ we sometimes identify the characters of $G/N$ with the characters of $G$ whose kernel contains $N$.

The cardinality of a set, or the order of a finite group, $X$, is denoted by $|X|$.
If a group $A$ acts on a finite set $X$, we denote by $A_x$ the stabilizer of $x\in X$ in $A$, analogously we denote by $A_{X'}$ the setwise stabilizer of $X'\subseteq X$.

Let $\ell$ be a prime.
If $A$ acts on a finite group $G$ by automorphisms, then there is a natural action of $A$ on $\Irr(G)\cup\IBr_\ell(G)$ given by ${}^{a^{-1}}\chi(g)=\chi^a(g)=\chi(g^{a^{-1}})$ for every $g\in G$, $a\in A$ and $\chi\in\Irr(G)\cup\IBr_\ell(G)$.
For $P\le G$ and $\chi\in\Irr(G)\cup\IBr_\ell(G)$, we denote by $A_{P,\chi}$ the stabilizer of $\chi$ in $A_P$.

Let $\chi\in \Irr(G)$, we denote by $\chi^\circ$ the restriction of $\chi$ to the set of all $\ell'$-elements of $G$.
Let $Y\subseteq \IBr_\ell(G)$.
A subset $X \subseteq \Irr(G)$ is called a \emph{basic set} of $Y$ if
$\{\chi^\circ\mid \chi\in X\}$ is a $\mathbb Z$-basis of $\mathbb ZY$.
If $Y=\IBr_\ell(B)$ for some $\ell$-block $B$ of $G$, then we also say $X$ a  basic set of $B$.

Let $\mathcal O$ denote the ring of algebraic integers in $\mathbb C$.	
Following \cite[\S2]{Na98} we fix a maximal ideal $M$ of $\mathcal O$ containing the ideal $\ell\mathcal O$. Then by \cite[Lemma 2.1]{Na98} the field $\F := \mathcal O /M$ is an algebraic closure of its prime field $\F_\ell$ of characteristic $\ell$, and we denote by   	
${}^*:\mathcal O\to \F$ the natural epimorphism.
Let $\chi\in \Irr(G)$.
Then the central character associated to $\chi$ is the algebra homomorphism $\omega_\chi: Z(\mathbb C G)\to\mathbb C$, $C\mapsto \omega_\chi(\hat C)=\frac{|C|\chi(x)}{\chi(1)}$, where $C$ is a conjugacy class of $G$, $\hat C=\sum_{x\in C}x$ and $x\in C$.
This yields an algebra homomorphism $\lambda_\chi:Z(\F G)\to\F$ such that $\lambda_\chi(\hat C)=\omega_\chi(\hat C)^*$ for a conjugacy class  $C$ of $G$.
Then for $\chi,\psi\in\Irr(G)$, they are in the same $\ell$-block of $G$ if and only if $\lambda_\chi=\lambda_\psi$.
Let $B$ be an $\ell$-block of $G$, then we define $\lambda_B=\lambda_\chi$ for $\chi\in\Irr(B)$.
For a prime $\ell$ and $\varphi\in\Irr(G)\cup\IBr_\ell(G)$, we denote by
$\mathrm{bl}_\ell(\varphi)$ the $\ell$-block of $G$ containing $\varphi$.

A subgroup $R \le G$ is \emph{$\ell$-radical} if $R=\cO_\ell(N_G(R))$. We also say that $R$ is an $\ell$-radical subgroup of $G$.
We denote by $\Rad_\ell(G)$ the set of $\ell$-radical subgroups	of $G$.
Furthermore, $\Rad_\ell(G)/\sim_G$  denotes a $G$-transversal of radical $\ell$-subgroup of $G$.

We denote  the set of all $G$-conjugacy classes of $\ell$-weights of $G$  by $\cW_\ell(G)$ while
$\cW_\ell(B)$ denotes  the set of all $G$-conjugacy classes of $B$-weights for an $\ell$-block $B$ of $G$.

\vspace{2ex}

The following lemma is elementary.
\begin{lem}\label{gen-quo-wei}
	Let $G$ be a finite group, $Z$ a central subgroup of $G$ and
	$\pi:G\to \bar G=G/Z$ be the canonical homomorphism.
	Suppose that $\bar B$ is an $\ell$-block of $\bar G$ which is dominated by the $\ell$-block
	$B$ of $G$.
	Let $(\bar R,\bar\varphi)$   be a $\bar B$-weight and let $R=\mathcal O_\ell(\pi^{-1}(\bar R))$ and $\varphi$ the inflation of $\bar\varphi$ from $N_{\bar G}(\bar R)=\overline{N_G(R)}$ to $N_G(R)$.
	Then $(R,\varphi)$  is a $ B$-weight.
\end{lem}
\begin{lem}\label{rad-quotient}
Let $G$ be a finite group, $Z$ a central $\ell'$-subgroup of $G$ and $\bar G=G/Z$.
Then there is a bijection  $\Theta:\Rad_\ell(G)\to \Rad_\ell(\bar G)$ given by $R\mapsto \bar R$
 with inverse given by $Q/Z\mapsto \cO_\ell(Q)$.

 Moreover, $\Theta$ induces a bijection between $\Rad_\ell(G)/\sim_G$ and $\Rad_\ell(\bar G)/\sim_{\bar G}$.
\end{lem}
\begin{proof}
This follows by \cite[Lem.~2.3~(c)]{NT11}.
\end{proof}

Note that, we have $N_{\bar G}(\bar R)=\overline{N_G(R)}$ in Lemma \ref{rad-quotient}.

\begin{lem}\label{wei-quotient}
Keep the hypothesis and notation of Lemma \ref{rad-quotient}.
\begin{enumerate}[(i)]
	\item If $(\bar R,\bar{\varphi})$ is an $\ell$-weight of $\bar G$, then $(R,\varphi)$ is an $\ell$-weight of $G$, where $R=\Theta^{-1}(\bar R)$ and $\varphi$ is the inflation of $\bar{\varphi}$ to $N_G(R)$.
	\item If $(R,\varphi)$ is an $\ell$-weight of $G$ such that $Z\le\ker \varphi$ and $\bar R=\Theta(R)$, then $(\bar R,\bar{\varphi})$ is an $\ell$-weight of $\bar G$, where $\bar{\varphi}$ is the character of $N_{\bar G}(\bar R)$ whose inflation is $\varphi$.
	\item Let $\bar B$ be an $\ell$-block of $\bar G$ and $B$ an $\ell$-block of $G$ dominating $\bar B$.
	Then the map $\cW_\ell(B)\to\cW_\ell(\bar B) $ given by $(R,\varphi)\mapsto (\bar R,\bar{\varphi})$ is a bijection.
\end{enumerate}
\end{lem}
\begin{proof}
Both (i) and (ii) are obvious.
For (iii), by (i), (ii) and Lemma \ref{rad-quotient}, it suffices to show that $Z\le\ker \varphi$ holds for any $(R,\varphi)\in\cW_\ell(B)$.
Let $b=\bl_\ell(\varphi)$, then $b^G=B$.
Thus $\la_b(z)=\la_B(z)$ for all $z\in Z$.
By \cite[Thm.~(9.9)(c)]{Na98}, $Z\le \ker \chi$ for all $\chi\in \Irr(B)$.
Hence $\omega_\chi(z)=\chi(z)/\chi(1)=1$ and then $\la_B(z)=1$ for all $z\in Z$.
Thus $\la_b(z)=1$, i.e. $\omega_\varphi(z)^*=1$  for all $z\in Z$.
However, $\omega_\varphi(z)$ is an $\ell'$-root of unity and then by \cite[Lem.~(2.1)]{Na98}, $1=\omega_\varphi(z)=\varphi(z)/\varphi(1)$ for all $z\in Z$.
Hence $Z\le \ker\varphi$, as desired.
\end{proof}

\begin{lem}\label{act-radical}
Keep the hypothesis and notation of Lemma \ref{rad-quotient}.
Let $\sigma\in\Aut(G)$ such that 
$\sigma$ stabilizes $Z$,
and $\bar{\sigma}$ the automorphism of $\bar G$ induced by $\sigma$.
Let $\bar R_i\in\Rad_\ell(\bar G)$ and $R_i=\Theta^{-1}(\bar R_i)$ for $i=1,2$.
If $\bar{\sigma}(\bar R_1)=\bar R_2$, then $\sigma(R_1)=R_2$.
\end{lem}
\begin{proof}
By the assumption, $\sigma(R_1)Z=R_2Z$.	
Then $\sigma(R_1)=R_2$ since it is the unique Sylow $\ell$-subgroup of $\sigma(R_1)Z=R_2Z$, as stated.
\end{proof}

\begin{lem}\label{ind-quot}
Keep the hypothesis and notation of Lemma \ref{rad-quotient} and Lemma \ref{wei-quotient}~(iii). 
Let $A$ be a subgroup of $\Aut(G)$ such that $A$ stabilizes $Z$,
and $\bar A$ the subgroup of $\Aut(\bar G)$ induced by $A$.
If there is an $\bar A$-equivariant bijection between $\IBr_\ell(\bar B)$ and $\cW_\ell(\bar B)$, then there is an $A$-equivariant bijection between $\IBr_\ell(B)$ and $\cW_\ell(B)$.
\end{lem}

\begin{proof}
This follows by Lemma \ref{wei-quotient}~(iii) and Lemma \ref{act-radical} and the fact that $\IBr_\ell(B)=\IBr_\ell(\bar B)$~(see for example \cite[Thm.~(9.9)]{Na98}).
\end{proof}

\vspace{2ex}

Let $G$ be a finite group, $\chi\in\Irr(G)$, $\ell$ a prime,
and $\theta$ a linear character of $G$. Then $\theta\chi$ is an irreducible character of $G$ and the map $\chi\mapsto \theta\chi$ is a permutation on $\Irr(G)$.
Moreover, this permutation respects $\ell$-blocks.
Now we let $B$ be an $\ell$-block of $G$ and assume that $\theta$ is of $\ell'$-order.
	Then by \cite[Lem.~2.4]{Fe18}, there is an $\ell$-block $\theta\otimes B$ of $G$, \label{def-otimes-block}
	such that $\Irr(\theta\otimes B)=\{\theta\chi\mid \chi\in\Irr(B)\}$.
	Moreover, $\IBr_\ell(\theta\otimes B)=\{\theta^\circ\phi\mid\phi\in\IBr_\ell(B)\}$.	

\begin{lem}\label{ext-Br}
Let $N$ be a normal subgroup of a finite group $G$ such that $G/N$ is cyclic and $G/NZ(G)$ is an $\ell'$-group and
let $b$ be an $\ell$-block of $N$.
Suppose that there are $m$ $\ell$-blocks of $G$ covering $b$, where $m=|G/N|_{\ell'}$.
Then the following statements hold.
\begin{enumerate}[(i)]
\item $\Res^{G}_{N}:\IBr_\ell(B)\to \IBr_\ell(b)$ is bijective for any $\ell$-block $B$ of $G$ covering $b$.
\item Let $A$ be a subgroup of $\Aut(G)$ stabilizing $N$.
Suppose that $B$ is an $\ell$-block of $G$ covering $b$ such that $B$ is $A$-invariant.
If $\phi\in\IBr_\ell(b)$ which is $A$-invariant, then there is an extension $\tilde\phi\in\IBr_\ell(B)$ of $\phi$ such that $\tilde\phi$ is $A$-invariant.
\end{enumerate}
\end{lem}

\begin{proof}
Let $\phi\in\IBr_{\ell}(b)$ and $\tilde\phi\in\IBr_\ell(G\mid\phi)$.
Since $G/N$ is cyclic, by Clifford theory, each irreducible $\ell$-Brauer character covering $\phi$ has the form $\tilde\phi\tau$ with $\tau\in\IBr_\ell(G/N)$.
Then $|\IBr_\ell(G\mid\phi)|\le m$.
Now there are $m$ $\ell$-blocks of $G$ covering $b$, so
$|\IBr_\ell(G\mid\phi)|= m$.
By Clifford theory,
$G/G_\phi$ is an $\ell$-group.
Since $G/NZ(G)$ is an $\ell'$-group, we have $G=G_\phi$, and then $\tilde\phi$ is an extension of $\phi$.
Thus (i) follows easily.

For (ii), let $\tilde\phi\in\IBr_\ell(B)$ be the extension of $\phi$.
Since $\phi$ is $A$-invariant and $B$ is $A$-invariant too,
we get that $\tilde\phi^{a}\in\IBr_\ell(B)$ is also an extension of $\phi$ for any $a\in A$.
By the uniqueness of $\tilde\phi$, we have $\tilde\phi^a=\tilde\phi$.
Then $\tilde\phi$ is $A$-invariant.
\end{proof}

\begin{lem}
	\label{ind-block}
	Let $K$ be a subgroup of a finite group $G$ and $b$ an $\ell$-block of $K$, and
	 $\theta$ a linear character of $G$ of $\ell'$-order.
	Assume that both $b^G$ and $((\Res^G_K\theta)\otimes b)^G$ are defined. Then  $((\Res^G_K\theta)\otimes b)^G=\theta\otimes b^G$.
\end{lem}
\begin{proof}
	Let $B=b^G$.
Then
$\lambda_{B}(\hat C)=\lambda_{ b}(\widehat{C\cap K})$ and
$\lambda_{((\Res^G_K\theta)\otimes b)^G}(\hat C)=\lambda_{(\Res^G_K\theta)\otimes b}(\widehat{C\cap K})$, for any conjugacy class $C$ of $G$.
	It is easy to check that
	$\lambda_{\theta\otimes B}(\hat C)=\theta(x)^*\lambda_{ b}(\widehat{C\cap K})$ and
$\lambda_{(\Res^G_K\theta)\otimes b}(\widehat{C\cap K})=\theta(x)^*\lambda_{ b}(\widehat{C\cap K})$ for $x\in C$.
Thus $\lambda_{\theta\otimes B}=\lambda_{((\Res^G_K\theta)\otimes b)^G}$ and then
$((\Res^G_K\theta)\otimes b)^G=\theta\otimes B$.
\end{proof}

By Lemma \ref{ind-block} we have the following result immediately.

\begin{cor}\label{lin-wei}
Let $G$ be a finite group, $B$ an $\ell$-block and $(R,\varphi)$ a $B$-weight.
Suppose that $\theta$ is a linear character of $G$ with $\ell'$-order.
Then $(R,(\Res^{G}_{N_G(R)}\theta)\varphi)$ is a $(\theta\otimes B)$-weight.
\end{cor}

We will make use of the following result.

\begin{lem}\label{ex-unitriibr}
Let $A$ be a finite group, $G$ a normal subgroup of $A$ and $B$ an $\ell$-block of $G$.
Suppose that there exists a basic set $X\subseteq \Irr(B)$ of $B$ such that the corresponding decomposition matrix is unitriangular with respect to a suitable order.
If every $\chi\in X$ extends to $A_\chi$,
then every $\phi\in\IBr_\ell(B)$ extends to $A_\phi$.
\end{lem}

\begin{proof}
By \cite[Lem.~7.5]{CS13}, there exists an $A$-equivariant bijection  $\mathscr D: X\to \IBr_\ell(B)$ such that $\chi^\circ=\mathscr D(\chi)+\sum\limits_{\phi\in\IBr_\ell(B)\setminus \{\mathscr D(\chi)  \}}d_\phi\phi$ with $d_\phi\in\mathbb Z_{\ge 0}$.
In particular, $A_\chi=A_{\mathscr D(\chi)}$.
Now let $\tilde\chi$ be an extension of $\chi$ to $A_\chi$ and let $\tilde\phi$ be an irreducible constituent of $\tilde\chi^\circ$ such that $\tilde\phi\in\IBr_\ell(A_\chi\mid\mathscr D(\chi))$.
Since $\mathscr D(\chi)$ is $A_\chi$-invariant, by Clifford theory, we know that $\tilde\phi$ is an extension of $\mathscr D(\chi)$ to $A_{\mathscr D(\chi)}$.
\end{proof}

\subsection{The inductive blockwise Alperin weight condition}

	For a finite group $H$ and a prime $\ell$, we denote by
$\mathrm{dz}_\ell(H)$ the set of $\ell$-defect zero irreducible characters of $H$.	
	If $Q$ is a radical $\ell$-subgroup of $H$ and $B$ an $\ell$-block of $H$, then we define the set
	$$\mathrm{dz}_\ell(N_H(Q)/Q,B):=\{ \chi\in\mathrm{dz}_\ell(N_H(Q)/Q)\ | \ \mathrm{bl}_\ell(\chi)^H=B  \},$$
	where we regard $\chi$ as an irreducible character of $N_G (Q)$ containing $Q$ in its kernel when considering the induced $\ell$-block $\mathrm{bl}_\ell(\chi)^H$.

There are several versions of the (iBAW) condition.
Apart from the original version given in \cite[Def. 4.1]{Sp13}, there is also a version treating only blocks with defect groups involved in certain sets of $\ell$-groups \cite[Def. 5.17]{Sp13}, or a version handling single blocks \cite[Def. 3.2]{KS16}.
We shall consider the inductive condition for a single block here.

\begin{defn}[{\cite[Def. 3.2]{KS16}}]
	\label{induc}
	Let $\ell$ be a prime, $S$ a finite non-abelian simple group and $X$ the universal $\ell'$-covering group of $S$. Let $B$ be an $\ell$-block of $X$. We say the \emph{inductive blockwise Alperin weight (iBAW) condition} holds for $B$ if the following statements hold.
	\begin{enumerate}
		\item[(i)] There exist subsets $\IBr_\ell(B\ |\ Q)\subseteq \IBr_\ell(B)$ for $Q\in \Rad_\ell(X)$ with the following properties.
		\begin{enumerate}
			\item[(1)] $\IBr_\ell(B\ |\ Q)^a=\IBr_\ell(B\ |\ Q^a)$ for every $Q\in \Rad_\ell(X)$, $a\in \Aut(X)_B$,
			
			\item[(2)] $\IBr_\ell(B)=\dot{\bigcup}_{Q\in \Rad_\ell(X)/\thicksim_X}\IBr_\ell(B\ |\ Q)$.
		\end{enumerate}
		\item[(ii)] For every $Q\in \Rad_\ell(X)$ there exists a bijection $$\Omega^X_Q:\IBr_\ell(B\ |\ Q)\to \mathrm{dz}_\ell(N_X(Q)/Q,B)$$such that $\Omega^X_Q(\phi)^a=\Omega^X_{Q^a}(\phi^a)$ for every $\phi\in \IBr_\ell(B\ |\ Q)$ and $a\in \Aut(X)_B$.
		
		\item[(iii)] For every $Q\in \Rad_\ell(X)$ and every $\phi\in \IBr_\ell(B\ |\ Q)$ there exist a finite group $A:=A(\phi, Q)$ and $\tilde{\phi}\in \IBr_\ell(A)$ and $\tilde{\phi'}\in \IBr_\ell(N_A(\overline{Q}))$, where we use the notation $$\overline{Q}:=QZ/Z\ \text{and}\ Z:=Z(X)\cap \ker(\phi),$$ with the following properties.
		\begin{enumerate}[(1)]
			\item for $\overline{X}:=X/Z$ the group $A$ satisfies $\overline{X}\unlhd A$, $A/C_A(\overline{X})\cong \Aut(X)_\phi$, $C_A(\overline{X})=Z(A)$ and $\ell \nmid |Z(A)|$,
			
			\item $\tilde{\phi}\in \IBr_\ell(A)$ is an extension of the $\ell$-Brauer character of $\overline{X}$ associated with $\phi$,
			
			\item $\tilde{\phi'}\in \IBr_\ell(N_A(\overline{Q}))$ is an extension of the $\ell$-Brauer character of $N_{\overline{X}}(\overline{Q})$ associated with the inflation of $\Omega^X_Q(\phi)^\circ\in \IBr_\ell(N_X(Q)/Q)$ to $N_X(Q)$,
			
			\item $\mathrm{bl}_\ell(\mathrm{Res}^A_J(\tilde{\phi}))=\mathrm{bl}_\ell(\Res^{N_A(\overline{Q})}_{N_J(\overline{Q})}(\tilde{\phi'}))^J$ for every subgroup $J$ satisfying $\overline{X}\le J \le A$.
		\end{enumerate}
\end{enumerate}
\end{defn}

\begin{defn}\label{ibaw-simple}
Let $\ell$ be a prime, $S$ a finite non-abelian simple group and $X$ the universal $\ell'$-covering group of $S$.
We say that \emph{the inductive blockwise Alperin weight (iBAW) condition holds for $S$ and the prime $\ell$} if the (iBAW) condition
holds for every $\ell$-block of $X$.
\end{defn}

\begin{lem}\label{first-2-con}
Let $\ell$ be a prime, $S$ a finite non-abelian simple group and $X$ the universal $\ell'$-covering group of $S$.
Let $B$ be an $\ell$-block of $X$.
If there is an $\Aut(X)_B$-equivariant bijection between $\IBr_\ell(B)$ and $\cW_\ell(B)$, then there are natural defined sets $\IBr_\ell(B\mid Q)$ and bijections $\Omega_Q(X)$ such that (i) and (ii) of Definition \ref{induc} holds for $B$.
\end{lem}
\begin{proof}
This is \cite[Lem.~2.10]{Sc16}.
\end{proof}

\begin{cor}\label{cyclic-outer-ind}
Let $\ell$ be a prime, $S$ a finite non-abelian simple group such that  $\Aut(S)/S$ is cyclic and $X$ the universal $\ell'$-covering group of $S$. Let $B$ be an $\ell$-block of $X$.
If there is an $\Aut(X)_B$-equivariant bijection between $\IBr_\ell(B)$ and $\cW_\ell(B)$, then the (iBAW) condition holds for $B$.
\end{cor}
\begin{proof}
By the proof of \cite[Lem.~6.1]{Sp13}, it suffices to prove (i) and (ii) of Definition \ref{induc}~(for details, see \cite[Rmk.~2.7]{Sc16}), which follows by Lemma \ref{first-2-con}.
\end{proof}

By \cite[Thm.~C]{Sp13}, the (iBAW) condition holds for every finite non-abelian simple group of Lie type and its defining characteristic.
From this, we only consider non-defining characteristic for classical groups in this paper.

\subsection{Background of the representations of finite groups of Lie type}

We will need to view some finite classical groups as the groups of fixed
points under some Frobenius endomorphisms of certain connected reductive algebraic
groups.
Let $q$ be a power of a prime $p$ and let
$\F_q$ be the field of $q$ elements.
Also let $\overline\F_q$ be the algebraic closure of the field $\F_q$.

Algebraic groups are usually denoted by boldface letters.
Suppose that $\mathbf G$ is a connected reductive algebraic group over $\overline\F_q$ and $F:\mathbf G\to\mathbf G$ a Frobenius endomorphism endowing $\mathbf G$ with an $\mathbb F_q$-structure.
The group of rational points $\mathbf G^F$ is finite.
Let $\mathbf G^*$ be dual to $\mathbf G$ with corresponding Frobenius endomorphism also denoted $F$.

Let $\ell$ be a prime number different from $p$.
For a semisimple $\ell'$-element $s$ of ${\mathbf G^*}^F$,
we denote by $\mathcal E_{\ell}(\mathbf G^F,s)$ the union of the Lusztig series $\mathcal E(\mathbf G^F,st)$, where $t$ runs through semisimple $\ell$-elements of ${\mathbf G^*}^F$ commuting with $s$.
By \cite{BM89}, the set $\mathcal E_{\ell}(\mathbf G^F,s)$ is a union of $\ell$-blocks of $\mathbf G^F$.

Also, we denote by $\mathcal E(\mathbf G^F,\ell')$ the set of irreducible characters of $\mathbf G^F$ lying in a Lusztig series $\mathcal E(\mathbf G^F,s)$, where $s\in{\mathbf G^*}^F$ is a semisimple $\ell'$-element.
Considering the elements of $\mathcal E(\mathbf G^F,\ell')$ as a basic set is the main argument of \cite{GH91} with the assumption that $\ell$ is good and $Z(\bG)$ is connected.
It was generalized in \cite[Thm.~A]{Ge93}, which can be stated as follows.

\begin{thm}\label{basicset}
	Let $\ell$ be a prime good for $\bG$ and not dividing the defining characteristic of $\bG$.
	Assume that $\ell$ does not divide $(Z(\bG)/Z^\circ (\bG))_F$~(the largest quotient of $Z(\bG)$ on which $F$ acts trivially).
	Let $s\in{\bG^*}^F$ be a semisimple $\ell'$-element.
	Then $\mathcal E(\mathbf G^F,s)$ form a basic set of $\mathcal E_\ell(\mathbf G^F,s)$.
\end{thm}

In this paper, any algebraic group  $\bG$ involved  is of classical type and the prime $\ell$ is always odd.
Thus the	hypothesis of Theorem \ref{basicset} is always satisfied.

Let $d$ be a positive integer. We will make use of the terminology of Sylow $d$-theory~(see for instance \cite{BM92} and \cite{BMM93}).
For an $F$-stable maximal torus $\mathbf T$ of $\mathbf G$, denote by $\mathbf T_d$ its Sylow $d$-torus.
An $F$-stable Levi subgroup $\mathbf L$ of $\mathbf G$ is called \emph{$d$-split} if $\mathbf L=C_{\mathbf G}(Z^\circ (\mathbf L)_d)$, and $\zeta\in\Irr (\mathbf L^F)$ is called \emph{$d$-cuspidal} if ${}^*R^{\mathbf L}_{\mathbf M\subseteq \mathbf P}(\zeta)=0$ for all proper $d$-split Levi subgroups $\mathbf M < \mathbf L$ and any parabolic subgroup $\mathbf P$ of $\mathbf L$ containing $\mathbf M$ as Levi complement.

Let $s\in {\mathbf G^*}^F$ be semisimple.
Following \cite[Def.~2.1]{KM15},
we say $\chi\in\mathcal E(\mathbf G^F,s)$ is \emph{$d$-Jordan-cuspidal} if
\begin{itemize}
	\item $Z^\circ(C^\circ_{\mathbf G^*}(s))_d=Z^\circ(\mathbf G^*)_d$, and
	\item $\chi$ corresponds under Jordan decomposition~(cf. \cite[Prop. 5.1]{Lu88}) to the $C_{\mathbf G^*}(s)^F$-orbit of a $d$-cuspidal unipotent character of $C^\circ_{\mathbf G^*}(s)^F$.
\end{itemize}
If $\mathbf L$ is a $d$-split Levi subgroup of $\mathbf G$ and $\zeta\in\Irr(\mathbf L^F)$
is $d$-Jordan-cuspidal, then $(\mathbf L,\zeta)$ is called a \emph{$d$-Jordan-cuspidal pair} of $\mathbf G$.

Now we define an integer $e_0=e_0(q,\ell)$ for odd prime $\ell$, which is denoted by ``$e$'' in \cite{KM15}~(in this paper, we will use ``$e$'' for another integer, see \S\ref{notations-and-conventions}, page \pageref{def-int-e}):
\begin{equation}\label{definitionofe0}
	e_0=e_0(q,\ell)= \text{multiplicative order of}\ q \ \text{modulo}\ \ell.
\end{equation}

The paper \cite{CE99} gave a label for arbitrary $\ell$-blocks of finite groups of Lie type for $\ell\ge 7$ and it was generalised in \cite{KM15} to its largest possible generality.
Under the conditions of \cite[Thm. A (e)]{KM15}, the set of $\mathbf G^F$-conjugacy classes of $e_0$-Jordan-cuspidal pairs $(\mathbf L,\zeta)$ of $\mathbf G$ such that $\zeta\in\mathcal E(\mathbf L^F,\ell')$, is a labeling set of the $\ell$-blocks of $\mathbf G^F$.

By \cite[Thm.]{BM11}, the Mackey formula holds for groups of classical type, hence
the Lusztig induction $R_{\bL\subseteq \mathbf P}^{\bG}$ is independent of the ambient parabolic subgroup $\mathbf P$ in this paper.
So throughout this paper we always omit the parabolic subgroups when considering Lusztig inductions.

\subsection{Some notation and conventions for classical groups} \label{notations-and-conventions}

From now on we always assume that $p$ is an odd prime, $q=p^f$ with a positive integer $f$, and $\ell$ is an odd prime number different from $p$.
We will consider some classical groups over $\F_q$, such as (conformal) symplectic or orthogonal groups. For the definitions of the classical groups appearing in this paper, we refer to \cite[\S 2]{KL90}.
One should note that the notation in \cite{KL90} differs from the notation used in this paper in the way that in \cite{KL90} the conformal orthogonal groups are denoted by $\GO_n^{\pm}(q)$ and the general orthogonal groups are written as $\mathrm{O}_n^{\pm}(q)$,
while in this paper (follows the notation of \cite{MT11}) the conformal orthogonal groups are denoted by $\CO_n^{\pm}(q)$ and the general orthogonal groups are written as $\mathrm{GO}_n^{\pm}(q)$.

We follow mainly the notation from \cite{FS89} and \cite{An94}.
Let $V$ be a finite-dimensional symplectic or orthogonal space over the field $\F_q$.
We denote by $I(V)$ the group of isometries of $V$,
$I_0(V)$ the subgroups of $I(V)$ of determinant $1$, and
$\eta(V)=\pm1$ the type of $V$ if $V$ is orthogonal.
For simplicity, we set $\eta(V) = 1$ if $V$ is symplectic.
Furthermore, we identify $1$, $-1$ with $+$, $-$ respectively when considering the type of spaces and groups.
Obviously, $I(V)=I_0(V)=\Sp(V)$ if  $V$ is a symplectic space and $I(V)=\GO(V)$, $I_0(V)=\SO(V)$ if $V$ is an orthogonal space.

We recall that there exists a set $\cF$ of polynomials serving as elementary divisors for all semisimple elements of each of these groups.
We denote by  $\Irr(\F_{q}[x])$ the set of all  monic irreducible polynomials over the field $\F_{q}$.
For each $\Delta$ in $\Irr(\F_{q}[x])\setminus \{x\}$, we define $\Delta^*$ be the polynomial in
$\Irr(\F_{q}[x])$ whose roots are the inverses of the roots of $\Delta$.
Now, we denote by
\begin{align*}
\cF_{0}&=\left\{ x-1,x+1 ~\right\},\\
\cF_{1}&=\left\{ \Delta\in\Irr(\F_{q}[x])\mid \Delta\notin \cF_0, \Delta\neq x,\Delta=\Delta^* ~\right\},\\
\cF_{2}&=\left\{~ \Delta\Delta^* ~|~ \Delta\in\Irr(\F_{q}[x])\setminus \cF_0, \Delta\neq x,\Delta\ne\Delta^* ~\right\}.
\end{align*}
Let $\cF=\cF_0\cup\cF_1\cup\cF_2$.
Given $\Gamma\in\cF$, denote by $d_\Gamma$ its degree and by $\delta_\Gamma$ its \emph{reduced degree} defined by
$$\delta_\Gamma=
\left\{ \begin{array}{ll} d_\Gamma & \text{if}\ \Gamma\in\cF_0; \\
\frac{1}{2}d_\Gamma & \text{if}\ \Gamma\in\cF_1\cup \cF_2 . \end{array} \right.$$
Since the polynomials in $\cF_1\cup \cF_2$ have even degree, $\delta_\Gamma$ is an integer.
In addition, we mention a sign $\varepsilon_\Gamma$ for $\Gamma\in\cF_1\cup \cF_2$ defined by
$$\varepsilon_\Gamma=
\left\{ \begin{array}{ll} -1 & \text{if}\ \Gamma\in\cF_1; \\
1 & \text{if}\ \Gamma\in\cF_2 . \end{array} \right.$$

Given a semisimple element $s\in I(V)$, there exists a unique orthogonal decomposition
\begin{equation}\label{def-pri-dec}
V=\sum\limits_\Gamma V_\Gamma(s), \quad s=\prod_{\Gamma}s(\Gamma),
\end{equation}
where the $V_\Gamma(s)$ are non-degenerate subspaces of $V$, $s(\Gamma)\in I(V_\Gamma(s))$, and $s(\Gamma)$ has minimal polynomial $\Gamma$.
The decomposition (\ref{def-pri-dec}) is called the \emph{primary decomposition} of $s$ in $I(V)$.
Let $m_\Gamma(s)$ be the multiplicity of $\Gamma$ in $s(\Gamma)$.
If $m_\Gamma(s)\ne 0$, then we say $\Gamma$ is an \emph{elementary divisor} of $s$.
Then the centralizer of $s$ in  $I(V)$ has a decomposition $C_{I(V)}(s)=\prod_{\Gamma}C_\Gamma(s)$, where $C_\Gamma(s)=C_{I(V_\Gamma(s))}(s(\Gamma))$.
Moreover, by \cite[(1.13)]{FS89},
\begin{equation*}
C_\Gamma(s)=
\left\{ \begin{array}{ll} I(V_\Gamma(s)) & \text{if}\ \Gamma\in\cF_0; \\
\GL_{m_\Gamma(s)}(\varepsilon_\Gamma q^{\delta_\Gamma}) & \text{if}\ \Gamma\in\cF_1\cup\cF_2. \end{array} \right.
\end{equation*}
Here, $\GL_m(-q)$ means $\GU_m(q)$.
Note that $C_\Gamma(s)\le I_0(V_\Gamma(s))$ for $\Gamma\in\cF_1\cup\cF_2$.

Let $\eta_\Gamma(s)$ be the type of $V_\Gamma(s)$.
Here $\eta_\Gamma(s)=1$ for all $\Gamma\in\cF$ if $V$ is symplectic.
By \cite[(1.12)]{FS89},
the multiplicity and type functions $\Gamma\mapsto m_\Gamma(s)$, $\Gamma\mapsto \eta_\Gamma(s)$ satisfy the following relations
\begin{equation}\label{function-mult-type}
\begin{aligned}
&\mathrm{dim}V=\sum\limits_\Gamma d_\Gamma m_\Gamma(s),\\
&\eta(V)=(-1)^{\frac{q-1}{2}m_{x-1}(s)m_{x+1}(s)}\prod_{\Gamma}\eta_\Gamma(s),\\
&\eta_\Gamma(s)=\vare_\Gamma^{m_\Gamma(s)}\  \text{for} \ \Gamma\in\cF_1\cup\cF_2.
\end{aligned}
\end{equation}
Conversely, if $\Gamma\mapsto m_\Gamma(s)$, $\Gamma\mapsto \eta_\Gamma$ are functions from $\cF$ to $\mathbb N$, $\{\pm 1\}$ respectively satisfying (\ref{function-mult-type}),
then there exists a semisimple element $s$ of $I(V)$ with these functions as multiplicity and type functions.
Moreover, two semisimple elements $s$ and $s'$ of $I(V)$ are $I(V)$-conjugate if and only if $m_\Gamma(s) = m_\Gamma(s')$ and $\eta_\Gamma(s)=\eta_\Gamma(s')$ for all $\Gamma\in\cF$.

Now assume that $V$ is orthogonal.
A semisimple element $s$ lies in $I_0(V)$ if and only if $m_{x+1}(s)$ is even.
If $s\in I_0(V)$, then
$$|C_{I_0(V)}(s):\prod_{\Gamma}C_{I_0(V_\Gamma(s))}(s(\Gamma))|=1\ \text{or}\ 2,$$ and index $2$ occurs if and only if $m_{x-1}(s)$ and $m_{x+1}(s)$ are both non-zero.
For more details, see \cite[\S1]{FS89}.

\vspace{2ex}

For a semisimple element $s\in I_0(V)$, we define $\Psi_\Gamma(s)$ to be the set of partitions of $m_\Gamma(s)$ if $\Gamma\in\cF_1\cup\cF_2$.
If $\Gamma\in\cF_0$, then  $\Psi_\Gamma(s)$ is defined to be the set of symbols of rank $[\frac{m_\Gamma(s)}{2}]$ such that
\begin{itemize}
\item If  $V_\Gamma(s)$ is symplectic or orthogonal of odd-dimension, then the symbols have odd defect.
\item If  $V_\Gamma(s)$ is orthogonal of even-dimension and type $+$, then the symbols have defect divisible by $4$.
Moreover, degenerate symbols are counted twice.
\item If $V_\Gamma(s)$ is orthogonal of even-dimension and type $-$, then the symbols have defect congruent to $2$ modulo $4$.
\end{itemize}
Let
\begin{equation}\label{def-par-sym}
\Psi(s)=\prod_\Gamma \Psi_\Gamma(s).
\end{equation}
Following \cite[p.~132]{FS89},
we define an operate $'$ on the sets $\Psi(s)$ and $\Psi_\Gamma(s)$ as follows.
Let $\mu_\Gamma\in \Psi_\Gamma(s)$.
Then define $(\mu_\Gamma)'=\mu_\Gamma$ if $\mu_\Gamma$ is a partition or a non-degenerate symbol, and
define $(\mu_\Gamma)'$ to be the other copy of $\mu_\Gamma$ in $\Psi_\Gamma(s)$ if $\mu_\Gamma$ is a degenerate symbol.
If $\mu=\prod_\Gamma \mu_\Gamma\in \Psi(s)$, then we define $\mu'=\prod_\Gamma (\mu_\Gamma)'$.

In this paper, we let $e$ be the multiplicative order of $q^2$ modulo $\ell$. \label{def-int-e}
Then  $e=e_0/\mathrm{gcd}(2,e_0)$,
where $e_0$ is defined as in (\ref{definitionofe0}).
We say the prime $\ell$ is \emph{linear} if $e_0$ is odd while $\ell$ is \emph{unitary} if $e_0$ is even.

Let $\cF'$ be the subset of $\cF$ consisting of polynomials whose roots are of $\ell'$-order.
For $\Gamma\in\cF'$, we define $e_\Gamma$ to be  the multiplicative order of $q^{2}$ or $\vare_\Gamma q^{\delta_\Gamma}$ modulo $\ell$ according as $\Gamma\in\cF_0$ or $\Gamma\in\cF_1\cup\cF_2$. Then $e_\Gamma=e$ for $\Gamma\in\cF_0$.
Let $s$ be a semisimple $\ell'$-element of $I_0(V)$
and $\mu=\prod_\Gamma \mu_\Gamma\in \Psi(s)$.
Now we define the $e_\Gamma$-core of $\mu_\Gamma$ for every $\Gamma\in \cF'$.
If $\Gamma\in\cF_1\cup\cF_2$, the $e_\Gamma$-core of $\mu_\Gamma$ is defined in the usual way for partitions~(see for example \cite[\S3]{Ol93}).
For $\Gamma\in\cF_0$, the $e_\Gamma$-core of $\mu_\Gamma$ is defined in \cite[p.~159]{FS89} which we state as follows.
Let $\kappa_\Gamma$ be the symbols which is gotten by actually removing $w_\Gamma$ $e_\Gamma$-hooks (or $e_\Gamma$-cohooks, resp.) from $\mu_\Gamma$ and
there is no $e_\Gamma$-hooks (or $e_\Gamma$-cohooks, resp.) in $\kappa_\Gamma$ if $\ell$ is linear (or unitary, resp.).
If $\kappa_\Gamma$ is degenerate and $w_\Gamma>0$, then both copies of $\kappa_\Gamma$ are considered as the $e_\Gamma$-core of $\mu_\Gamma$~(\emph{i.e.}, the $e_\Gamma$-core of $\mu_\Gamma$ is defined to be the doubleton $\{ \kappa_\Gamma,\kappa_\Gamma'\}$).
If $\kappa_\Gamma$ is degenerate and $w_\Gamma=0$~(\emph{i.e.}, $\kappa_\Gamma=\mu_\Gamma$),
then only $\kappa_\Gamma$, but not its copy, is the $e_\Gamma$-core of $\mu_\Gamma$.
If $\kappa_\Gamma$ is non-degenerate, then the $e_\Gamma$-core of $\mu_\Gamma$ is
$\kappa_\Gamma$.

Note that the definition of $e_\Gamma$-core of a symbol here (as in \cite{FS89}) is the same with the definition in \cite[p.~307]{FS86}, and is slightly different from those used in \cite{BMM93}, \cite{Ol86} and \cite{Ol93}.
We follow \cite{FS86} and \cite{FS89} and say $e_\Gamma$-core for both $e_\Gamma$-core (when $\ell$ is linear) and $e_\Gamma$-cocore (when $\ell$ is unitary) in \cite{BMM93}, \cite{Ol86} and \cite{Ol93}.

Let $s$ be a semisimple $\ell'$-element.
For $\Gamma\in\cF$, we define $\mathcal C_\Gamma(s)$ the set of  $\kappa_\Gamma$ such that there exists $\mu_\Gamma\in\Psi_\Gamma(s)$ satisfying that $\kappa_\Gamma$ is an $e_\Gamma$-core of $\mu_\Gamma$.
Denote
\begin{equation}\label{def-core}
\cC(s)=\prod_{\Gamma} \cC_\Gamma(s).
\end{equation}
In particular, for each $\kappa\in\cC(s)$, the cardinality of the set $\kappa$ is $1$, $2$ or $4$. 
We also define an operate $'$ on the sets $\cC(s)$ and $\cC_\Gamma(s)$ as follows.  \label{def-c-s}
Define $\kappa_\Gamma'=\kappa_\Gamma$ if $\kappa_\Gamma$ is a doubleton and define $\kappa_\Gamma'=\kappa_\Gamma$ otherwise.
For $\kappa=\prod_\Gamma \mu_\Gamma\in \cC(s)$,  we define $\kappa'=\prod_\Gamma (\kappa_\Gamma)'$.

For $\kappa\in\cC(s)$, we let
\begin{equation}\label{def-par-sym-core}
\Psi(s,\kappa)=\prod_\Gamma\Psi_\Gamma(s,\kappa),
\end{equation}
where $\Psi_\Gamma(s,\kappa):=\{\mu\in \Psi_\Gamma(s)\mid \kappa_\Gamma \ \text{is the}\ e_\Gamma\text{-core} \ \text{of}\ \mu_\Gamma\ \text{for every}\ \Gamma\in\cF\}$.

For integers $d\ge 1$, $w\ge 0$,
we let
\begin{equation}\label{def-set-lab}
\mathcal P(d,w)=\{(\mu_1,\ldots,\mu_{d})\mid \sum_{i=1}^{d}|\mu_i|=w \},
\end{equation}
where $\mu_1,\ldots, \mu_{d}$ are partitions.

Now let
\begin{equation}\label{par-P0}
\mathcal P_0(2e,w)=\{ (\mu_1,\ldots,\mu_{2e})\in\mathcal P(2e,w)\mid \mu_i=\mu_{i+e}\ \text{for}\ 1\le i\le e \}
\end{equation}
and $\mathcal P_1(2e,w)=\mathcal P(2e,w)\setminus \mathcal P_0(2e,w)$.
Then $\mathcal P_0(2e,w)$ is not empty if and only if $w$ is even.
We define an equivalence relation on $\mathcal P(2e,w)$:
for $\mu^{(k)}=(\mu^{(k)}_1,\ldots,\mu^{(k)}_{2e})$ with $k=1,2$,
\begin{equation}\label{equ-par}
\mu^{(1)}\sim\mu^{(2)}\ \text{if and only if}\ \mu^{(1)}_i=\mu^{(2)}_{e+i}\ \text{and} \ \mu^{(2)}_i=\mu^{(1)}_{e+i}\ \text{for every}\ 1\le i\le e.
\end{equation}
Let $\mathcal P'_1(2e,w)=\mathcal P_1(2e,w)/\sim$.
Then we define
\begin{equation}\label{def-p'}
\mathcal P'(2e,w):=\mathcal P_0(2e,w)\cup\mathcal P'_1(2e,w),
\end{equation}
where the elements of $\mathcal P_0(2e,w)$ are counted twice.
In particular, if $w$ is odd, then $\mathcal P_0(2e,w)$ is empty and $\mathcal P'(2e,w)=\mathcal P'_1(2e,w)$.

Let $\beta_\Gamma=2$ or $1$ according as $\Gamma\in\cF_0$ or $\Gamma\in\cF_1\cup\cF_2$.
Then by \cite[Prop.~(3.7)]{Ol93} and the proofs of \cite[Prop.~14 and 15]{Ol86},
$\Psi_\Gamma(s,\kappa)$ is in bijection with
\begin{equation}\label{bijection-sym-skew-par}
\begin{split}
 &\mathcal P(\beta_\Gamma e_\Gamma,w_\Gamma) \ \textrm{if $\Gamma\in \cF_0$ and $\kappa_\Gamma$ is non-degenerate or  $\Gamma\in \cF_1\cup\cF_2$, and}\\
 &\mathcal P'(2e,w_\Gamma)\ \textrm{if $\Gamma\in \cF_0$ and $\kappa_\Gamma$ is degenerate}.
\end{split}
\end{equation}

\vspace{2ex}

Again following the notation of \cite{FS89} and \cite{An94}, we denote by $V$ and $V^*$ finite-dimensional symplectic or orthogonal spaces over $\F_q$ related as follows:
\begin{equation*}
\begin{matrix}
V & \mathrm{dim V} & V^* &  \mathrm{dim V^*} \\
\text{symplectic} & 2n & \text{orthogonal} & 2n+1\\
\text{orthogonal} & 2n+1 & \text{symplectic} & 2n\\
\text{orthogonal} & 2n & \text{orthogonal} & 2n
\end{matrix}
\end{equation*}
where $\eta(V)=\eta(V^*)=1$ in the first two cases and $\eta(V)=\eta(V^*)$ in the third case.
Here $\eta(V)=1$ for a symplectic space as before.

\vspace{2ex}

For a positive integer $d$, we denote by $I_{(d)}$ the identity matrix of degree $d$
and by $I_{d}$ the identity matrix of degree $\ell^d$.
We end the section with the following lemmas.

\begin{lem}\label{central}
Let $n=md$ and $\varepsilon\in\{\pm 1\}$.
Let $\iota$ be the natural embedding of $\GL_m(\varepsilon q^d)$  into $\SO_{2n}^\varepsilon (q)$.
Then $\iota(\GL_m(\varepsilon q^d))\nsubseteq \Omega_{2n}^\varepsilon (q)=[\SO_{2n}^\varepsilon (q),\SO_{2n}^\varepsilon (q)]$.
\end{lem}

\begin{proof}
First assume that $\varepsilon=1$.
Now we give the structure of the embeddings.
First note that
$\SO_{2n}^+ (q)=\{ A\in \SL_{2n}(q)\mid A^{tr}K_{2n} A=K_{2n} \}$, where $K_{2n}=
\begin{pmatrix}
 & &1 \\
 & \iddots & \\
 1&&
\end{pmatrix}$.
Let $\xi$ be a generator of $\F_{q^d}^\times$.
Let $\iota_1$ be the embedding of $\GL_m(q^d)$ into $\GL_{n}(q)$ 
induced by $\F_{q^d}^\times\hookrightarrow \GL_d(q)$, $\xi\mapsto (\Lambda_\xi)$, where $\Lambda_\xi$ is the minimal polynomial of $\xi$ over $\F_{q}$.
Let $A\in \GL_n(q)$, then we take $\iota_2(A)=\mathrm{diag}(A,K_n (A^{-1})^{tr} K_n)$.
Then $\iota=\iota_2\circ\iota_1$.
Let $\theta: \SO_{2n}^+ (q)\to \F_q^\times/(\F_q^\times)^2$ be the spinor norm~(see, for example, \cite[\S 9]{Gr02} or \cite[\S2]{KL90}).
Then $\Omega_{2n}^+ (q)$ is the kernel of $\theta$.
Thus it suffices to show that $\theta(\iota(\GL_m(\varepsilon q^d)))=\F_q^\times/(\F_q^\times)^2$.

Let $V$ be the orthogonal space with $\mathrm{dim}(V)=2n$ with a non-degenerate
symmetric bilinear form $B$ and basis $\nu_1,\ldots,\nu_{2n}$ such that $B(\nu_i,\nu_j)=1$ if $i+j=2n+1$ and $B(\nu_i,\nu_j)=0$ otherwise.
For non-singular $v\in V$, let $\sigma_v$ be the reflection along $v$, \emph{i.e.}, $\sigma_\nu(u)=u-\frac{2B(u,v)}{B(v,v)}v$ for any $u\in V$.
For $\zeta\in \F_{q}^\times$ and $1\le i\le n$
the matrix corresponding to $\sigma_{\nu_i+\zeta\nu_{2n-i+1}}$ with respect to the basis $\nu_1,\ldots,\nu_{2n}$ is
$$\begin{pmatrix}
I_{(i-1)}& & & &  \\
& & & -\zeta & \\
& & I_{(2n-2i)} & & \\
&-\zeta^{-1} & & &\\
& & & & I_{(i-1)}
\end{pmatrix}.$$

Now let $\xi$ be a generator of $\F_{q^d}^\times$.
Over the field $\F_{q^d}^\times$, $\iota(\mathrm{diag}(\xi,I_{(m-1)}))$ is conjugate to
$$\mathrm{diag}(\xi,\xi^q,\ldots,\xi^{q^{d-1}},I_{(2(m-1)d)},\xi^{q^{-d+1}},\ldots, \xi^{-q},\xi^{-1}),$$
which is $\prod\limits_{i=0}^{d-1} \sigma_{\nu_i+\xi^{q^i}\nu_{2n-i+1}}\sigma_{\nu_i+\nu_{2n-i+1}}$.
Let $Q(v)=B(v,v)$ for $v\in V$ be the associated quadratic form.
Then $\prod\limits_{i=0}^{d-1}Q(\nu_i+\xi^{q^i}\nu_{2n-i+1})Q(\nu_i+\nu_{2n-i+1})=\prod\limits_{i=0}^{d-1}2^2\xi^{q^i}=2^{2d}\xi^{\frac{q^d-1}{q-1}}\in\F_q^\times$.
By Remark 3 after \cite[Cor.~9.9]{Gr02},
the spinor norm can be determined from the Clifford groups, and then is independent of the fields.
Then $\theta(\iota(\mathrm{diag}(a,I_{(m-1)})))=\xi^{\frac{q^d-1}{q-1}}(\F_q^\times)^2$.
Note that $\xi^{\frac{q^d-1}{q-1}}$ is a generator of $\F_q^\times$, we have
$\theta(\iota(\GL_m(\varepsilon q^d)))=\F_q^\times/(\F_q^\times)^2$.

The proof for $\varepsilon=-1$ is completely analogous.
\end{proof}

\begin{lem}\label{central2}
Let $V_0$ and $V_1$ be orthogonal spaces such that $\dim(V_0)$ is even and $\dim(V_1)=\ell^d$.
Suppose that $C$ is a subgroup of $I_0(V_0)$ satisfies that $C\nsubseteq \Omega(V_0)$.
Let $V=V_0\otimes V_1$ be the orthogonal space defined as in \cite[p.~127]{KL90}.
Then $C\otimes I_d\nsubseteq \Omega(V)$.
\end{lem}

\begin{proof}
Let $\theta_0: I_0(V_0)\to \F_q^\times/(\F_q^\times)^2$ and $\theta: I_0(V)\to \F_q^\times/(\F_q^\times)^2$
be the spinor norms.
For a non-singular vector $v\in V_0$, we claim that $\theta(\sigma_v\otimes I_d)=\theta_0(\sigma_v)^{\ell^d}$.
In fact, by \cite[Lem.~4.4.13]{KL90},
$\sigma_v\otimes I_d=\sigma_{v\otimes w_1}\cdots \sigma_{v\otimes w_{\ell^d}}$, where $w_1,\ldots,w_{\ell^d}$ is a basis of $V_1$ as in \cite[Prop.~2.5.12]{KL90}.
Thus $\theta(\sigma_v\otimes I_d)=\theta_0(\sigma_v)^{\ell^d}\delta$ by direct calculation, where $\delta\in\F_q^\times$ is determined by the basis $w_1,\ldots,w_{\ell^d}$.
For $g\in C$, we write $g=\sigma_{v_1}\cdots\sigma_{v_t}$. 
Then $g\otimes I_d=(\sigma_{v_1}\otimes I_d)\cdots(\sigma_{v_t}\otimes I_d)$,
and then $\theta(g\otimes I_d)=\theta_0(g)^{\ell^d}\delta^t$ which implies $\theta(g\otimes I_d)=\theta_0(g)^{\ell^d}$ since $t$ is even.
From this, $C\otimes I_d\nsubseteq \Omega(V)$ since $\ell$ is odd.
\end{proof}

\section{A basic case of weights for classical groups}\label{wei-spec-case}

\subsection{Radical subgroups of $I(V)$}\label{section:radical}

We first give some more notation and conventions as in \cite{An94}.
Let $e$ be defined as in \S\ref{notations-and-conventions}.
First, we define the integer $a$, and sign $\vare=\pm 1$ as follows:
let $\ell^a$ be the exact power of $\ell$ dividing $q^{2e}-1$ and
let $\vare$ be the sign chosen so that $\ell^a$ divides $q^e-\vare$.

Let $\alpha,\gamma$ be non-negative integers, $Z_\alpha$ be the cyclic group of order $\ell^{a+\alpha}$ and $E_\gamma$ be an extraspecial $\ell$-group of order $\ell^{2\gamma+1}$.
We may assume the exponent of $E_\gamma$ is $\ell$ by \cite[(1B)]{An94}.
Denote by $Z_\alpha E_\gamma$ the central product of $Z_\alpha$ and $E_\gamma$ over $\Omega_1(Z_\alpha)=Z(E_\gamma)$.
Let $V_{\alpha,\gamma}$ be a symplectic or orthogonal space over $\F_q$ of dimension $2e\ell^{\alpha+\gamma}$ and $\eta(V_{\alpha,\gamma})=\varepsilon$ if
$V_{\alpha,\gamma}$ is orthogonal.
By \cite[(1A)]{An94}, the group $Z_\alpha E_\gamma$ can be embedded into $\GL_{\ell^\gamma}(\varepsilon q^{e\ell^\alpha})$ uniquely up to conjugacy in the sense that $Z_\alpha$ is identified with $\cO_\ell(Z(\GL_{\ell^\gamma}(\varepsilon q^{e\ell^\alpha})))$.
We denote by $R_{\alpha,\gamma}$ the image of $Z_\alpha E_\gamma$ under the composition
$Z_\alpha E_\gamma\hookrightarrow \GL_{\ell^\gamma}(\varepsilon q^{e\ell^\alpha}) \hookrightarrow I(V_{\alpha,\gamma})$.
Then by \cite[(1C)]{An94}, $R_{\alpha,\gamma}$ is uniquely determined by $Z_\alpha E_\gamma$ up to conjugacy.

For an integer $m\ge 1$, let $V_{m,\alpha,\gamma}=V_{\alpha,\gamma}\perp \cdots \perp V_{\alpha,\gamma}$ with $m$ terms and
let $R_{m,\alpha,\gamma}=R_{\alpha,\gamma}\otimes I_{(m)}$.
For each positive integer $c$, let $A_c$ denote the elementary abelian group of order $\ell^c$.
For a sequence of positive integers $\bc=(c_1,\ldots,c_t)$ with $t\geqslant0$,
we denote by $A_\bc=A_{c_1}\wr\cdots\wr A_{c_t}$ and $|\bc|=c_1+\cdots+c_t$.
Then $A_\bc$ can be regarded as an $\ell$-subgroup of the symmetric group $\mathfrak S(\ell^{|\bc|})$.
Let $V_{m,\alpha,\gamma,\bc}=V_{m,\alpha,\gamma}\perp \cdots \perp V_{m,\alpha,\gamma}$ with $\ell^{|\bc|}$ terms.
Groups of the form $R_{m,\alpha,\gamma,\bc}=R_{m,\alpha,\gamma}\wr A_\bc$ are called the basic subgroups of $I(V_{m,\alpha,\gamma,\bc})$.
Then $R_{m,\alpha,\gamma,\bc}$ is determined up to conjugacy in $I(V_{m,\alpha,\gamma,\bc})$ and $\eta(V_{m,\alpha,\gamma,\bc})=\varepsilon^m$ if $V_{m,\alpha,\gamma,\bc}$ is orthogonal.
By \cite[(2D)]{An94}, any $\ell$-radical subgroup $R$ of $I(V)$ is conjugate to $R_0 \times R_1 \times \cdots \times R_u$, where $R_0$ is a trivial group and $R_i$ ($i\geqslant1$) is a basic subgroup.
Moreover, by the construction in \cite{An94},
\begin{equation}
R_{m,\alpha,\gamma,\bc}C_{I(V_{m,\alpha,\gamma,\bc})}(R_{m,\alpha,\gamma,\bc})\le I_0(V_{m,\alpha,\gamma,\bc})
\end{equation}  and
\begin{equation}\label{nor-bas}
N_{I(V_{m,\alpha,\gamma,\bc})}(R_{m,\alpha,\gamma,\bc})\nleq I_0(V_{m,\alpha,\gamma,\bc}).
\end{equation}

By \cite[Lem.2.2(ii)]{Fe18}, the map
$\Rad_\ell(I(V))\to\Rad_\ell(I_0(V))$ given by $R\mapsto R\cap I_0(V)$ is surjective.
Since $\ell$ is odd, we have that $\Rad_\ell(I(V))=\Rad_\ell(I_0(V))$.
Now assume that $V$ is orthogonal and let $\Omega(V)=[I_0(V),I_0(V)]$. Then $|I_0(V):\Omega(V)|=2$.
Similarly, $\Rad_\ell(\Omega(V))=\Rad_\ell(I_0(V))$.

\begin{lem}\label{detofC}
Assume that $V$ is orthogonal.
	Let $R$ be an $\ell$-radical subgroup of $I_0(V)$, then
	$N_{I(V)}(R)\nleq I_0(V)$ and
	 $N_{I_0(V)}(R)\nsubseteq \Omega(V)$.
\end{lem}

\begin{proof}
	$N_{I(V)}(R)\nleq I_0(V)$ follows by (\ref{nor-bas}).
For the second assertion, it suffices to show that $\theta(C_{I_0(V)}(R))=\F_q^\times/(\F_q^\times)^2$, where $\theta$ is the spinor norm.
If $R=R_0$ is trivial, then $C_{I_0(V)}(R)=I_0(V)$.
Now we assume that $R\ne R_0$.
And then it suffices to show that $\theta(C_{I_0(V_{m,\alpha,\gamma,\bc})}(R_{m,\alpha,\gamma,\bc}))=\F_q^\times/(\F_q^\times)^2$ for any $m,\alpha,\gamma,\bc$.
Note that $C_{I_0(V_{m,\alpha,\gamma,\bc})}(R_{m,\alpha,\gamma,\bc})=C_{I(V_{m,\alpha,\gamma,\bc})}(R_{m,\alpha,\gamma,\bc})=
C_{m,\alpha}\otimes I_\gamma\otimes I_\bc$ 
where $C_{m,\alpha}\cong \GL_m(\varepsilon q^{e\ell^\alpha})$ by \cite[p.12-13]{An94}.
Here $C_{m,\alpha}$ is the image of an embedding of $\GL_m(\varepsilon q^{e\ell^\alpha})$ into  $I_0(V_{m,\alpha})$.
Thus the assertion follows by Lemma \ref{central} and \ref{central2}.
\end{proof}

The following lemma follows from Lemma \ref{detofC} immediately.
\begin{lem}\label{gexingz}
	Let $V$ be orthogonal and
	$R$ an $\ell$-radical subgroup of $I_0(V)$. Then
$N_{I(V)}(R)/N_{I_0(V)}(R)\cong I(V)/I_0(V)$	
	and
$N_{I_0(V)}(R)/N_{\Omega(V)}(R)\cong I_0(V)/\Omega(V)$.
\end{lem}

Let $R$ be an $\ell$-radical subgroup of $I(V)$, then
by Lemma \ref{detofC}, $I(V)=I_0(V)N_{I(V)}(R)$ and
$I_0(V)=\Omega(V)N_{I_0(V)}(R)$.
So if two $\ell$-radical subgroups of $I(V)$ are $I(V)$-conjugate, then they are $I_0(V)$-conjugate and $\Omega(V)$-conjugate.
Thus we have:

\begin{cor}\label{relaofradicalconj}
	Assume that $V$ is orthogonal. Then
	$\Rad_\ell(I(V))/\thicksim_{I(V)}=\Rad_\ell( I_0(V))/\thicksim_{I_0(V)}=\Rad_\ell(\Omega(V))/\thicksim_{\Omega(V)}$.
\end{cor}

\begin{rmk}\label{action-on-rad}
By the uniqueness of $R_{m,\alpha,\gamma,\bc}$ proved in \cite[(1C)]{An94}, we know that $\Aut(I(V))$ acts trivially on $\Rad_\ell(I(V))/\thicksim_{I(V)}$.

Let $V$ be a symplectic or orthogonal space, $\tilde G=I(V)$, and $R$ an $\ell$-radical subgroup of $\tilde G$. Then there exists a corresponding decomposition $V=V_0\perp V_1\perp\cdots\perp V_t$, $R=R_0\times R_1\times\cdots\times R_t$
such that $R_0$ is the trivial subgroup of $I(V_0)$ and $R_i$ is a basic subgroup of $I(V_i)$ for $i > 1$.
Let $\sigma$ be an automorphism of $\tilde G$.
Then there is an automorphism $\sigma'$ of $\tilde G$, which is
a composition of $\sigma$ by some suitable inner automorphism,
such that $\sigma'$ stabilizes $V_i$ and $R_i$ for $0\le i\le t$.
\end{rmk}

\subsection{Action of automorphisms on weights of $I(V)$}\label{act-wei-IV}

Given $\Gamma\in\cF$, let $e_\Gamma$ and $\beta_\Gamma$ be defined as in \S\ref{notations-and-conventions}, and
let $\alpha_\Gamma$ and $m_\Gamma$  be the following:
$\ell^{\alpha_\Gamma}$ is the exact power of $\ell$ dividing $d_\Gamma$, and
$m_\Gamma$ satisfies $m_\Gamma e\ell^{\alpha_\Gamma}= e_\Gamma \delta_\Gamma$.
Recall that there is no direct connection between $m_\Gamma$ and $m_\Gamma(s)$.

In this section, we let $\sigma:=F_p$ be the field automorphism of $G=I_0(V)$ which sends $(a_{ij})$ to $(a^p_{ij})$ and let
$\sigma^*$ be the automorphism of   $G^*$ such that $\sigma$ is dual to  $\sigma^*$ as in \cite[\S5.3]{Ta18}.
Then $\sigma^*$ is also the field automorphism which sends $(a_{ij})$ to $(a^p_{ij})$.

Recall that $\cF'$ is defined to be the subset of $\cF$ consisting of polynomials whose roots are of $\ell'$-order.
Given $\Gamma\in\cF'$, we define $G_\Gamma$, $R_\Gamma$, $C_\Gamma$, $\theta_\Gamma$ and $s_\Gamma$ as follows:
let $V_\Gamma$ denote a symplectic or orthogonal space of dimension $2e_\Gamma \delta_\Gamma$ over $\F_q$ and of type $\vare$ or $\vare_\Gamma^{e_\Gamma}$ according as $\Gamma\in\cF_0$ or $\Gamma\in\cF_1\cup\cF_2$ if $V_\Gamma$ is orthogonal.
Let $\tilde G_\Gamma=I(V_\Gamma)$ and $G_\Gamma=I_0(V_\Gamma)$.
Thus $\tilde G_\Gamma$ has a primary element $s_\Gamma^*$ with a unique elementary divisor
$\Gamma$ of multiplicity  $\beta_\Gamma e_\Gamma$ and $\tilde G_\Gamma$ has a basic subgroup $R_\Gamma$ of form $R_{m_\Gamma,\alpha_\Gamma,0}$ by \cite[(1.12) and (5.2)]{FS89}.
Let $\tilde C_\Gamma=C_{\tilde G_\Gamma}(R_\Gamma)$ and $\tilde N_\Gamma=N_{\tilde G_\Gamma}(R_\Gamma)$.
Then $s_\Gamma^*\in G_\Gamma$, $\tilde C_\Gamma\le G_\Gamma$ and $\tilde C_\Gamma\cong \GL_{m_\Gamma}(\vare q^{e\ell^{\alpha_\Gamma}})$,
so that a Coxeter torus $\tilde T_\Gamma$ of $\tilde C_\Gamma$ has order $q^{m_\Gamma e\ell^{\alpha_\Gamma}}-\vare^{m_\Gamma}$.
The dual $\tilde T_\Gamma^*$ is embedded as a regular subgroup of $\tilde C_\Gamma^*$,
and $\tilde C_\Gamma^*$ is embedded as a regular subgroup of ${G_\Gamma}^*$.
By \cite[p.~22]{An94},
there exists an element $s_\Gamma$ in $\tilde T_\Gamma^*$ such that $C_{\tilde C_\Gamma^*}(s_\Gamma)=\tilde T_\Gamma^*$ and as an element of ${G_\Gamma}^*$, $s_\Gamma$ and $s_\Gamma^*$ are dual to each other in the sense of \cite[(3E)]{An94}.
Here, $s_\Gamma$ has unique elementary divisor $\Gamma$ and $s_\Gamma$ is uniquely determined by $\Gamma$ up to $I(V_\Gamma^*)$-conjugacy.
We denote by $\widehat {s_\Gamma}$ the character of $\tilde T_\Gamma$ corresponding to $s_\Gamma$ and let $\tilde \theta_\Gamma=\pm R_{\tilde T_\Gamma}^{\tilde C_\Gamma}(\widehat {s_\Gamma})$ where the sign is chosen so $\tilde \theta_\Gamma$ is an irreducible character of $\tilde C_\Gamma$.
The block $\tilde \fb_\Gamma$ of $\tilde C_\Gamma$ containing $\tilde \theta_\Gamma$ then has defect group $R_\Gamma$ by \cite[(4C)]{FS82} or \cite[(3.2)]{Br86}.

Let $\bc$, $\gamma$ be a tuple and an integer as in the previous sections, and $\delta=|\bc|+\gamma$.
Let $V_{\Gamma,\delta}=V_\Gamma\bot \cdots \bot V_\Gamma$, where there are $\ell^\delta$ terms $V_\Gamma$ on the right-hand side.
Then if $V_\Gamma$ is orthogonal,
$V_{\Gamma,\delta}$ has type $\vare^{\ell^\delta}=\vare$ or $\vare_\Gamma^{e_\Gamma\ell^\delta}=\vare_\Gamma^{e_\Gamma}$  according as $\Gamma\in \cF_0$ or $\Gamma\in\cF_1\cup\cF_2$.
Let $\tilde G_{\Gamma,\gamma,\bc}=I(V_{\Gamma,\delta})$,
$G_{\Gamma,\gamma,\bc}=I_0(V_{\Gamma,\delta})$, $R_{\Gamma,\gamma,\bc}=R_{m_\Gamma,\alpha_\Gamma,\gamma,\bc}$,
$\tilde N_{\Gamma,\gamma,\bc}=N_{\tilde G_{\Gamma,\gamma,\bc}}(R_{\Gamma,\gamma,\bc})$ and
$\tilde C_{\Gamma,\gamma,\bc}=C_{\tilde G_{\Gamma,\gamma,\bc}}(R_{\Gamma,\gamma,\bc})$.
Then $\tilde C_{\Gamma,\gamma,\bc}\cong \tilde C_\Gamma\otimes I_\delta$.
Let $\tilde \theta_{\Gamma,\gamma,\bc}=\tilde \theta_\Gamma\otimes I_\delta$,
then $\tilde \theta_{\Gamma,\gamma,\bc}$ can be viewed as a canonical character of $\tilde C_{\Gamma,\gamma,\bc}R_{\Gamma,\gamma,\bc}$ with $R_{\Gamma,\gamma,\bc}$ in the kernel and all canonical characters are of this form.
Let $B_{\Gamma,\delta,i}=\bl_\ell(\tilde\theta_{\Gamma,\delta,i})^{G_{\Gamma,\delta,i}}$.
Then $B_{\Gamma,\delta,i}\subseteq \mathcal E_\ell(G_{\Gamma,\delta,i},x_\Gamma)$, where $x_\Gamma=s_\Gamma\otimes I_\delta$ by the proof of \cite[(4A)]{An94}.

Let $\cR_{\Gamma,\delta}$ be the set of all the basic subgroups of the form $R_{\Gamma,\gamma,\bc}$ with $\gamma+|\bc|=\delta$.
Label the basic subgroups in $\cR_{\Gamma,\delta}$ as $R_{\Gamma,\delta,1}$, $R_{\Gamma,\delta,2}$, $\ldots$ and we denote the canonical character associated to $R_{\Gamma,\delta,i}$ by $\theta_{\Gamma,\delta,i}$.
It is possible that there exists $\Gamma'\in\cF'$ such that $m_{\Gamma'}=m_\Gamma=:m$ and $\alpha_{\Gamma'}=\alpha_\Gamma=:\alpha$.
In this case, $\cR_{\Gamma,\delta}=\cR_{\Gamma',\delta}$ and naturally we may choose the labeling of $\cR_{\Gamma,\delta}$ and $\cR_{\Gamma',\delta}$ such that $R_{\Gamma,\delta,i}=R_{\Gamma',\delta,i}$ for $i=1,2,\ldots$.
By convention,
we denote $R_{m,\alpha,\gamma,\bc}$ as $R_{\Gamma,\delta,i}$ or $R_{\Gamma',\delta,i}$ depending on whether
the related canonical character of $\tilde C_{m,\alpha}R_{m,\alpha}=\tilde C_{m,\alpha}$ considered is $\theta_\Gamma$ or $\theta_{\Gamma'}$.

Let $\sC_{\Gamma,\delta}$ be the set of characters of $(\tilde N_{\Gamma,\delta,i})_{\tilde \theta_{\Gamma,\delta,i}}$ lying over $\tilde \theta_{\Gamma,\delta,i}$ and of defect zero as characters of
$(\tilde N_{\Gamma,\delta,i})_{\tilde \theta_{\Gamma,\delta,i}}/R_{\Gamma,\delta,i}$ for all $i$.
By Clifford theory, this set is in bijection with the set of characters of $\tilde N_{\Gamma,\delta,i}$ lying over $\tilde \theta_{\Gamma,\delta,i}$ and of defect zero as characters of $\tilde N_{\Gamma,\delta,i}/R_{\Gamma,\delta,i}$ for all $i$.
We assume $\sC_{\Gamma,\delta}=\{\tilde \psi_{\Gamma,\delta,i,j}\}$ with $\tilde \psi_{\Gamma,\delta,i,j}$ a character of $(\tilde N_{\Gamma,\delta,i})_{\tilde \theta_{\Gamma,\delta,i}}$.
Then $|\sC_{\Gamma,\delta}|=\beta_\Gamma e_\Gamma \ell^d$ by the proof of \cite[(4A)]{An94}.

We define ${}^{\sigma^*}\Gamma$ to be the unique elementary divisor of $\sigma^*(s_\Gamma)$.
Obviously $m_{^{\sigma^*}\Gamma}=m_\Gamma$, $\alpha_{^{\sigma^*}\Gamma}=\alpha_\Gamma$ and $R_{\Gamma,\delta,i}=R_{{^{\sigma^*}\Gamma},\delta,i}$.
By Remark \ref{action-on-rad}, we may assume $R_{\Gamma,\delta,i}^\sigma=R_{\Gamma,\delta,i}$ up to a composition of $\sigma$ by a suitable inner automorphism.
Then we may assume $B_{\Gamma,\delta,i}^\sigma= B_{^{{\sigma^*}^{-1}}\Gamma,\delta,i}$ since $B_{\Gamma,\delta,i}\subseteq \mathcal E_\ell(G_{\Gamma,\delta,i},x_\Gamma)$, $B_{\Gamma,\delta,i}^\sigma\subseteq \mathcal E_\ell(G_{\Gamma,\delta,i},{\sigma^*}^{-1}(x_\Gamma))$ and ${\sigma^*}^{-1}(x_\Gamma)=x_{{}^{{\sigma^*}^{-1}}\Gamma}$~(see for instance \cite[Prop.~7.2]{Ta18}).
Since $\tilde \theta_{\Gamma,\delta,i}$ is the canonical character of a root block of $B_{\Gamma,\delta,i}$,
we have $\tilde \theta_{\Gamma,\delta,i}^\sigma=\tilde \theta_{^{{\sigma^*}^{-1}}\Gamma,\delta,i}$ up to a composition of $\sigma$ by an inner automorphism.
Then we may denote $R_{\Gamma,\delta,i}^\sigma=R_{{^{{\sigma^*}^{-1}}\Gamma},\delta,i}$, $\tilde N_{\Gamma,\delta,i}^\sigma=\tilde N_{{^{{\sigma^*}^{-1}}\Gamma},\delta,i}$ and
$\tilde C_{\Gamma,\delta,i}^\sigma=\tilde C_{{^{{\sigma^*}^{-1}}\Gamma},\delta,i}$ although the corresponding terms indexed by $\Gamma$ and $^{{\sigma^*}^{-1}}\Gamma$ are actually the same.
Also $((\tilde N_{\Gamma,\delta,i})_{\tilde \theta_{\Gamma,\delta,i}})^\sigma=
(\tilde N_{\Gamma,\delta,i})_{\tilde \theta_{\Gamma,\delta,i}^\sigma}=
(\tilde N_{{^{{\sigma^*}^{-1}}\Gamma},\delta,i})_{\tilde\theta_{{^{{\sigma^*}^{-1}}\Gamma},\delta,i}}$.
We may choose the labeling of $\sC_{\Gamma,\delta}$ and $\sC_{^{{\sigma^*}^{-1}}\Gamma,\delta}$ such that
\begin{equation}\label{weight:convention-act}
\tilde \psi_{\Gamma,\delta,i,j}^\sigma=\tilde \psi_{^{{\sigma^*}^{-1}}\Gamma,\delta,i,j}.
\end{equation}

\begin{rmk}\label{act-field-wei}
We can assume (\ref{weight:convention-act}) because $\tilde \psi_{\Gamma,\delta,i,j}$
is invariant under the action of $\sigma$ if $^{{\sigma^*}^{-1}}\Gamma=\Gamma$.
We prove this as follows.
First note that $^{{\sigma^*}^{-1}}\Gamma=\Gamma$ if and only if $\tilde\theta_{\Gamma,\delta,i}$ is invariant under the action of $\sigma$.
Let $R_{\Gamma,\delta,i}=R_{m_\Gamma,\alpha_\Gamma,\gamma,\bc}$ with $\bc=(c_1,\ldots,c_t)$.
We also abbreviate $R=R_{\Gamma,\delta,i}$, $\tilde N=\tilde N_{\Gamma,\delta,i}$,
$\tilde C=\tilde C_{\Gamma,\delta,i}$ and $\tilde\theta=\tilde\theta_{\Gamma,\delta,i}$.
By \cite[(2E)]{An94},
$\tilde N/R=\tilde N_{m_\Gamma,\alpha_\Gamma,\gamma}/R_{m_\Gamma,\alpha_\Gamma,\gamma}\times \prod_{i=1}^t \GL_{c_i}(\ell)$  and then we may assume that $|\bc|=0$, \emph{i.e.},
$R=R_{m_\Gamma,\alpha_\Gamma,\gamma}$.
Let $\tilde N^0=\{ g\in\tilde N\mid [g,Z(R)]=1\}$, then by the remark after \cite[(3I)]{An94}, $\tilde N^0\le \tilde N_{\tilde\theta}$ and
Also, $\tilde N^0=CRL$, where $L\le \tilde N^0$ satisfies
$[L,C]=1$, $L\cap C=Z(L)=Z(C)$
and $L/Z(L)R\cong \Sp_{2\gamma}(\ell)$.
Then there is exactly one character $\vartheta$ of $\tilde N^0$ which lies over $\tilde\theta$ and of $\ell$-defect zero when viewed as a character of $\tilde N^0/R$ and we may write $\vartheta=\tilde \theta\times \zeta$, where $\zeta$ is the Steinberg character of $\Sp_{2\gamma}(\ell)$.
Hence $\vartheta$ is invariant under the action of $\sigma$.
On the other hand, we have $\tilde N_{\tilde\theta}/\tilde N^0\cong \tilde N_\Gamma/\tilde C_\Gamma$.
From this we may assume further that $\gamma=0$ and then $\delta=0$, which means
$R=R_{\Gamma}$, $\tilde N=\tilde N_{\Gamma}$,
$\tilde C=\tilde C_{\Gamma}$ and $\tilde\theta=\tilde\theta_{\Gamma}$.
Now $\tilde C\cong\GL_{m_\Gamma}(\vare q^{e\ell^{\alpha_\Gamma}})$.
By \cite[(1.14)]{FS89}, $\tilde N=\langle \tilde C, D \rangle$, where $D$ is generated by one or two element and every element of $D$ acts on $\tilde C$ as a field or graph automorphism.
Then $\tilde\theta$ extends to $\langle \tilde N,\sigma \rangle$ by a result of Bonnaf\'e \cite[Thm. 4.3.1 and Lem. 4.3.2]{Bo99}~(see \cite[Prop.~4.17]{Fe18} for details).
Thus $\sigma$ acts trivially on $\Irr(\tilde N\mid\tilde\theta)$.
\end{rmk}

Now let $V$ be a symplectic or even-dimensional orthogonal space and
let $\tilde B$ be an $\ell$-block of $\tilde G=I(V)$ with defect group $D$ and root block $\mathbf b$ such that $V=[V,D]$ and $\mathbf b^{G}\subseteq \mathcal E_\ell(G,s)$ for some semisimple $\ell'$-element $s\in G^*$.
Let $s^*$ be a dual of $s$ in $G$ in the sense of \cite[(3E)]{An94}.
Then $m_\Gamma(s^*)=w_\Gamma\beta_\Gamma e_\Gamma$ for some positive integer $w_\Gamma$.
Similar with \cite[p.145]{LZ18} for groups of type $A$, now we define
$i\cW_\ell(\tilde B)$ to
be the set of elements
$K=K_\Gamma$, where
$K_\Gamma:\bigcup_\delta\sC_{\Gamma,\delta}\to\{~\ell\textrm{-cores}~\}$ such that
$\sum_{\delta,i,j}\ell^\delta |K_\Gamma(\psi_{\Gamma,\delta,i,j})|=w_\Gamma.$
Here, an $\ell$-core means $\ell$-core of some partition.

A bijection between $\cW_\ell(\tilde B)$ and $i\cW_\ell(\tilde B)$ has been constructed implicitly in \cite[(4E)]{An94} and can be described as follows.
Let $(R,\tilde\varphi)$ be an $\ell$-weight of $\tilde G$.
Set $\tilde C=C_{\tilde G}(R)$ and $\tilde N=N_{\tilde G}(R)$.
Thus there exists an $\ell$-block $\tilde b$ of $\tilde CR$ with $R$ a defect group such that $\tilde \varphi=\Ind_{\tilde N(\theta)}^{\tilde N}\tilde \psi$ where $\tilde \theta$ is the canonical character of $\tilde b$ and $\tilde \psi$ is a character of $\tilde N(\tilde \theta)$ lying over $\tilde \theta$ and of $\ell$-defect zero as a character of $\tilde N(\tilde \theta)/R$.

We may suppose that $Z(D)\le Z(R)\le R\le D$ so that $V=[V,R]$.
Assume we have the following decomposition $\tilde  \theta_+=\prod_{\Gamma,\delta,i}\tilde \theta_{\Gamma,\delta,i}^{t_{\Gamma,\delta,i}}$, $R_+=\prod_{\Gamma,\delta,i}R_{\Gamma,\delta,i}^{t_{\Gamma,\delta,i}}.$
Note that $\theta_\Gamma$ determines a semisimple $\ell'$-element of $G_\Gamma$ with a unique elementary divisor $\Gamma$ of multiplicity $\beta_\Gamma e_\Gamma$.
Thus $m_{\Gamma}(s^*)=\prod_{\delta,i} t_{\Gamma,\delta,i}\beta_\Gamma e_\Gamma\ell^\delta$ for each $\Gamma$.

Now we have
$$\tilde N(\tilde \theta)=\prod_{\Gamma,\delta,i}
\tilde N_{\Gamma,\delta,i}(\tilde \theta_{\Gamma,\delta,i})\wr\fS(t_{\Gamma,\delta,i}),
\quad \tilde \psi=\prod_{\Gamma,\delta,i} \tilde \psi_{\Gamma,\delta,i}$$
with $\tilde \psi_{\Gamma,\delta,i}$ a character of $\tilde N_{\Gamma,\delta,i}(\tilde \theta_{\Gamma,\delta,i})\wr\fS(t_{\Gamma,\delta,i})$ covering $\tilde \theta_{\Gamma,\delta,i}^{t_{\Gamma,\delta,i}}$ and of $\ell$-defect zero as a character of $\left(\tilde N_{\Gamma,\delta,i}(\tilde \theta_{\Gamma,\delta,i})\wr\fS(t_{\Gamma,\delta,i})\right)/
R_{\Gamma,\delta,i}^{t_{\Gamma,\delta,i}}$.
By Clifford theory, $\tilde \psi_{\Gamma,\delta,i}$ is of the form
\begin{equation}\label{weights:psi}
\Ind_{\tilde N_{\Gamma,\delta,i}(\theta_{\Gamma,\delta,i})\wr
	\prod_j\fS(t_{\Gamma,\delta,i,j})}
^{\tilde N_{\Gamma,\delta,i}(\theta_{\Gamma,\delta,i})\wr\fS(t_{\Gamma,\delta,i})}
\overline{\prod_j\tilde \psi_{\Gamma,\delta,i,j}^{t_{\Gamma,\delta,i,j}}}\cdot
\prod_j\phi_{\kappa_{\Gamma,\delta,i,j}}
\end{equation}
where $t_{\Gamma,\delta,i}=\sum_j t_{\Gamma,\delta,i,j}$, $\overline{\prod_j\tilde \psi_{\Gamma,\delta,i,j}^{t_{\Gamma,\delta,i,j}}}$ is an extension of
$\prod_j\tilde \psi_{\Gamma,\delta,i,j}^{t_{\Gamma,\delta,i,j}}$ from $\tilde N_{\Gamma,\delta,i}(\tilde \theta_{\Gamma,\delta,i})^{t_{\Gamma,\delta,i}}$ to $\tilde N_{\Gamma,\delta,i}(\tilde \theta_{\Gamma,\delta,i})\wr
\prod_j\fS(t_{\Gamma,\delta,i,j})$, $\kappa_{\Gamma,\delta,i,j} \vdash t_{\Gamma,\delta,i,j}$ is an $e_\Gamma$-core and $\phi_{\kappa_{\Gamma,\delta,i,j}}$ a character of $\fS(t_{\Gamma,\delta,i,j})$ corresponding to $\kappa_{\Gamma,\delta,i,j}$.
Now, define $K_\Gamma:\bigcup_\delta\sC_{\Gamma,\delta} \to \{~\ell\textrm{-cores}~\},
\psi_{\Gamma,\delta,i,j} \mapsto \kappa_{\Gamma,\delta,i,j}$.
Then $K=\prod_\Gamma K_\Gamma$ is the labeling of the $\ell$-weight $(R,\tilde\varphi)$.

We can define the action of $\sigma^*$ on $K$ by $(\sigma^* K)_{{}^{\sigma^*}\Gamma}=K_\Gamma$.

\begin{lem}\label{action-basic}
With the notation above, if $(R,\tilde \varphi)$ is a $\tilde B$-weight with label $K$,
$(R,\tilde \varphi)^\sigma$ is a $\tilde B^\sigma$-weight with label ${\sigma^*}^{-1} K$.
\end{lem}

\begin{proof}
Let $K'$ be the label of $(R,\tilde \varphi)^\sigma$.
	First note that $R^\sigma=R$,  $\tilde C^\sigma=\tilde C$, $\tilde N^\sigma=\tilde N$,
	and $\sigma$ stabilizes every $\tilde C_{\Gamma,\delta,i}$ up to conjugacy.
	
	By the argument above,
	we may denote $R_{\Gamma,\delta,i}^\sigma=R_{{^{{\sigma^*}^{-1}}\Gamma},\delta,i}$, $\tilde N_{\Gamma,\delta,i}^\sigma=\tilde N_{{^{{\sigma^*}^{-1}}\Gamma},\delta,i}$ and
	$\tilde C_{\Gamma,\delta,i}^\sigma=\tilde C_{{^{{\sigma^*}^{-1}}\Gamma},\delta,i}$ although the corresponding terms indexed by $\Gamma$ and $^{{\sigma^*}^{-1}}\Gamma$ are actually the same.
	To determine $K'$, we note that $\tilde \psi^\sigma=\prod_{\Gamma,\delta,i} \tilde \psi_{\Gamma,\delta,i}^\sigma$.
	By (\ref{weights:psi}), $\tilde \psi_{\Gamma,\delta,i}^\sigma$ is
	$$\Ind
	_{\tilde N_{\Gamma,\delta,i}(\tilde \theta_{\Gamma,\delta,i})^\sigma\wr
		\prod_j\fS(t_{\Gamma,\delta,i,j})}
	^{\tilde N_{\Gamma,\delta,i}(\tilde \theta_{\Gamma,\delta,i})^\sigma\wr
		\fS(t_{\Gamma,\delta,i})}
	\left(\overline{\prod_j\tilde \psi_{\Gamma,\delta,i,j}^{t_{\Gamma,\delta,i,j}}}\right)^\sigma\cdot
	\prod_j\phi_{\kappa_{\Gamma,\delta,i,j}}.$$
	Here, we note that $\sigma$ acts trivially on $\fS(t_{\Gamma,\delta,i})$ and $\fS(t_{\Gamma,\delta,i,j})$ by Remark \ref{action-on-rad}.
	Since $\tilde \theta_{\Gamma,\delta,i}^\sigma= \tilde \theta_{{^{{\sigma^*}^{-1}}\Gamma},\delta,i}$, we have $\tilde N_{\Gamma,\delta,i}(\tilde \theta_{\Gamma,\delta,i})^\sigma=
	\tilde N_{{^{{\sigma^*}^{-1}}\Gamma},\delta,i}(\tilde \theta_{{^{{\sigma^*}^{-1}}\Gamma},\delta,i})$.
	We can fix the way to extend $\prod_j\tilde \psi_{\Gamma,\delta,i,j}^{t_{\Gamma,\delta,i,j}}$ as in \cite[Lem.~25.5]{Hu98} , then we have that
	$$\left(\overline{\prod_j\tilde \psi_{\Gamma,\delta,i,j}^{t_{\Gamma,\delta,i,j}}}\right)^\sigma
	=\overline{\prod_j\left(\tilde \psi_{\Gamma,\delta,i,j}^\sigma\right)^{t_{\Gamma,\delta,i,j}}}.$$
	Since $\tilde \psi_{\Gamma,\delta,i,j}^\sigma=\psi_{^{{\sigma^*}^{-1}}\Gamma,\delta,i,j}$ by (\ref{weight:convention-act}), $\tilde \psi_{\Gamma,\delta,i}^\sigma$ is
	$$\Ind
	_{\tilde N_{{^{{\sigma^*}^{-1}}\Gamma},\delta,i}(\tilde \theta_{{^{{\sigma^*}^{-1}}\Gamma},\delta,i})\wr
		\prod_j\fS(t_{\Gamma,\delta,i,j})}
	^{\tilde N_{{^{{\sigma^*}^{-1}}\Gamma},\delta,i}(\tilde \theta_{{^{{\sigma^*}^{-1}}\Gamma},\delta,i})\wr
		\fS(t_{\Gamma,\delta,i})}
	\overline{\prod_j\tilde \psi_{{^{{\sigma^*}^{-1}}\Gamma},\delta,i,j}^{t_{\Gamma,\delta,i,j}}}\cdot
	\prod_j\phi_{\kappa_{\Gamma,\delta,i,j}}.$$
	Then $K'_{^{{\sigma^*}^{-1}}\Gamma}=K_\Gamma$ which is just $K'={{\sigma^*}^{-1} K}$.
\end{proof}

Thus, by a similar proof as in \cite[(1A)]{AF90}, there is a canonical bijection between $\mathcal P(\beta_\Gamma e_\Gamma,w_\Gamma)$ (defined as in (\ref{def-set-lab})) and $K_\Gamma$ for every $\Gamma\in\cF$ by \cite[Prop.(3.7)]{Ol93}.
Let $\mathcal P(\tilde B)=\prod_\Gamma \mathcal P(\beta_\Gamma e_\Gamma,w_\Gamma)$.
Then by the argument above, we have a bijection between
$\mathcal P(\tilde B)$ and $\cW_\ell(\tilde B)$.
We also define $\sigma^*\mu=\prod_\Gamma(\sigma^*\mu)_\Gamma$
with $(\sigma^*\mu)_{\sigma^*\Gamma}=\mu_\Gamma$.
By Lemma \ref{action-basic}, we have

\begin{cor}\label{action-basic-par}
	With the notation above, if $(R,\tilde \varphi)$ is a $\tilde B$-weight with label $\mu\in \mathcal P(\tilde B)$,
	then $(R,\tilde \varphi)^\sigma$ is a $\tilde B^\sigma$-weight with label ${\sigma^*}^{-1} \mu$.
\end{cor}

\vspace{2ex}

Now we consider the action of diagonal automorphisms on the weights of $\tilde B$.
The following two lemmas will be useful.

\begin{lem}\label{res-char}
Let $H$ be an arbitrary finite group, $L,K\unlhd H$ and $M=L\cap K$ such that $|H/K|=2$ and $H/L$ is cyclic.
Suppose that $\varphi\in\Irr(L)$ such that $\theta=\Res^{L}_{M}\varphi\in\Irr(M)$.
Assume further $\theta$ is $H$-invariant.
Then $|H:H_\varphi|\le 2$. Moreover, exactly one of the following statements hold.
\begin{enumerate}[(i)]
	\item $|H:H_\varphi|=1$ and $\Res^H_K\chi$ is irreducible for every $\chi\in\Irr(H_\theta)$.
	\item $|H:H_\varphi|=2$ and $\Res^H_K\chi$ is not irreducible for every $\chi\in\Irr(H_\theta)$.
\end{enumerate}
\end{lem}

\begin{proof}
Let $\chi\in \Irr(H\mid \varphi)$.
First by Clifford theory, we may write $\chi=\Ind^H_{H_\varphi}\psi$ for some $\psi\in\Irr(H_\varphi\mid\varphi)$ and then Mackey formula implies that $\Res^H_K \chi=\Ind^K_{K_\varphi} (\Res_{K_\varphi}^{H_\varphi} \psi)$.
Note that $\psi$ is an extension of $\theta$.
So $\Res_{K_\varphi}^{H_\varphi} \psi$ is irreducible.
Thus $\Ind^K_{K_\varphi} (\Res_{K_\varphi}^{H_\varphi} \psi)$ is a sum of $|H:H_\varphi|$ irreducible constituents.
Hence $|H:H_\varphi|\le 2$ since $|H/K|=2$ and the assertion follows easily.
\end{proof}

\begin{lem}\label{extension-central}
Let $H$ be an arbitrary finite group and $M\unlhd H$ such that $H=C_H(M)M$.
Suppose that $\theta\in\Irr(M)$ and $\eta\in \Irr(C_H(M))$ such that $\eta(1)=1$ and $\Irr(Z(M)\mid\theta)=\Irr(Z(M)\mid\eta)$.
Then there exists a unique extension $\varphi$ of $\theta$ to $H$
with $\Irr(C_H(M)\mid \varphi)=\{\eta\}$.
In particular, $\varphi(c)\ne 0$ for every $c\in C_H(M)$.
\end{lem}

\begin{proof}
In fact, this follows from \cite[2.1]{Sp09}. But for convenience, we still give the details here.
Let $\mathcal D:M\to \GL_{\theta(1)}(\mathbb C)$ be a $\mathbb C$-representation of $M$ affording $\theta$.
Then define $\mathcal D':H\to \GL_{\theta(1)}(\mathbb C)$, $cg\mapsto \eta(c)\mathcal D(g)$ for $c\in C_H(M)$ and $g\in M$.
It is easy to check that $\mathcal D'$ is well-defined and is a $\mathbb C$-representation of $H$.
Let $\varphi$ be the character afforded by $\mathcal D'$.
Then $\varphi$ is the unique extension of $\theta$ to $H$
with $\Irr(C_H(M)\mid \varphi)=\{\eta\}$.
Let $c\in C_H(M)$, then $\varphi(c)=\eta(c)\theta(1)$.
Now $\eta$ is a linear character, then $\eta(c)\ne 0$ for  every $c\in C_H(M)$.
Thus $\varphi(c)\ne 0$ for every $c\in C_H(M)$.
\end{proof}

Following the notation of \cite{FS89},  we denote by $J(V)$ the group of all conformal endomorphisms of $V$ when $\dim(V)$ is even.
Then $J(V)=\CSp(V)$ or $\CO(V)$ according as $V$ is symplectic or orthogonal.
We also let $J_0(V)=\CSp(V)$ or $\CSO(V)$ according as $V$ is symplectic or orthogonal.
Then $|J_0(V)/I_0(V)Z(J_0(V))|=2$.
Let $J_0(V)=\langle I_0(V),\tau \rangle$ where $\tau\in J_0(V)$.
Obviously, $J(V)=\langle I(V),\tau \rangle$.
Then for a basic subgroup $R_{\Gamma,\delta,i}$, up to a composition of $\tau$ by some suitable inner automorphism, which is denoted by $\tau'$, we have that
$N_{J(V_{\Gamma,\delta})}(R_{\Gamma,\delta,i})=\langle \tilde N_{\Gamma,\delta,i},  \tau'\rangle$,
$C_{J(V_{\Gamma,\delta})}(R_{\Gamma,\delta,i})=\langle \tilde C_{\Gamma,\delta,i},  \tau'\rangle$ and $\tau'$ commutes with $R_{\Gamma,\delta,i}\tilde C_{\Gamma,\delta,i}$ by \cite[\S5]{FS89}.

\begin{lem} \label{act-diag-wei}
	Keep the hypothesis and setup above.
\begin{enumerate}[(i)]
	\item If $\Gamma\ne x+1$, then every element of $\sC_{\Gamma,\delta}$ is invariant under $\tau'$.
	\item If $\Gamma=x+1$, then no element of $\sC_{\Gamma,\delta}$ is invariant under $\tau'$.
\end{enumerate}
\end{lem}

\begin{proof}
	Similar with the argument of Remark \ref{act-field-wei}, we may assume that $\delta=0$.
	Then $R_{\Gamma,\delta,i}=R_\Gamma$,
	$\tilde N_{\Gamma,\delta,i}=\tilde N_\Gamma$,
	$\tilde C_{\Gamma,\delta,i}=\tilde C_\Gamma$ and
	$\tilde \theta_{\Gamma,\delta,i}=\tilde \theta_\Gamma$.
	In this way $\sC_{\Gamma,\delta}$ is the set of extensions of $\tilde\theta_\Gamma$ to $\tilde N_\Gamma$.
	Recall that $\tilde\theta_\Gamma=\pm R_{\tilde T_\Gamma}^{\tilde C_\Gamma}(\widehat{s_\Gamma})$.
We abbreviate
$R=R_{\Gamma}$, $\tilde N=\tilde N_{\Gamma}$,
$\tilde C=\tilde C_{\Gamma}$, $\tilde \theta=\tilde \theta_{\Gamma}$,
$\tilde N'=N_{J(V_{\Gamma})}(R_{\Gamma})$
and $\tilde C'=C_{J(V_{\Gamma})}(R_{\Gamma})$.
Also, there is an extension $\tilde \theta'$ of $\tilde \theta$ to $\tilde C'$.
Claim that $\tilde\theta'(\tau')\ne 0$.
In order to do this, we choose a canonical $\tilde \theta'$.
First note that $C_{\tilde C'}(\tilde C)=\langle Z(\tilde C),\tau' \rangle$ is abelian.
Let $\{\eta\}=\Irr(Z(\tilde C)\mid \tilde\theta)$ and take $\eta'$ to be an extension of $\eta$ to $C_{\tilde C'}(\tilde C)$.
By Lemma \ref{extension-central}, there is a unqiue extension $\tilde \theta'$ of $\tilde \theta$ to $\tilde C'$ with $\{\eta'\}=\Irr(C_{\tilde C'}(\tilde C)\mid \tilde \theta')$.
In particular, $\tilde\theta'(\tau')\ne 0$, as claimed.
Hence for $g\in \tilde N$, $(\tilde\theta')^g=\tilde\theta'$ if and only if $\tilde\theta'(^{g}\tau')=\tilde\theta'(\tau')$ by \cite[Cor.~1.22]{Sc15}.
Also by \cite[(5A)]{FS89}, $[\tau',\tilde N]\le Z(\tilde C)$.
So $\tilde\theta'(^{g}\tau')=\tilde\theta([g,\tau'])\tilde\theta'(\tau')$ and thus
$(\tilde\theta')^g=\tilde\theta'$ if and only if $[g,\tau']\in \mathrm{ker}(\tilde\theta)$.
Now we calculate the stabilizer of $\tilde\theta'$ in $\tilde N'$.

If $\Gamma=x-1$, then $\tilde\theta'$ is the trivial character.
Thus $\tilde N'_{\tilde\theta'}=\tilde N'$.
If $\Gamma\in\cF_1\cup \cF_2$, then by \cite[(6A)(2)]{FS89}, we also have $\tilde N'_{\tilde\theta'}=\tilde N'$.
If $\Gamma=x+1$, then by \cite[(6A)(3)]{FS89}, we have $\tilde N'_{\tilde\theta'}\ne \tilde N'$.
Hence the assertion holds by Lemma \ref{res-char}.
\end{proof}

For $\Gamma=x+1$, we recall that $|\sC_{\Gamma,\delta}|=2e\ell^\delta$.
Thus by Lemma \ref{act-diag-wei}, we may rewrite $\sC_{x+1,\delta}=\{\tilde \psi_{x+1,\delta,i,j}\mid 1\le i\le 2e, 1\le j\le \ell^\delta\}$ such that
$\tilde \psi_{x+1,\delta,i,j}^{\tau'}=\tilde \psi_{x+1,\delta,e+i,j}$ for every $1\le i\le e$.
By Lemma \ref{act-diag-wei} again, we have the following result by a similar argument with Lemma \ref{action-basic} and \ref{action-basic-par}.

\begin{prop}\label{action-par-diag-wei}
With the notation of Corollary \ref{action-basic-par},  we let $(R,\tilde \varphi)$ be a $\tilde B$-weight with label $\mu\in \mathcal P(\tilde B)$ and write
$\mu=\prod_\Gamma\mu_\Gamma$, where $\mu_\Gamma=(\mu_\Gamma^{(1)},\ldots,\mu_\Gamma^{(\beta_\Gamma e_\Gamma)})$.
Then the image of $(R,\tilde \varphi)$ under the non-trivial action of $J(V)/Z(J(V))I(V)$ is a $\tilde B$-weight with label $\mu^\dag\in \mathcal P(\tilde B)$,
where $\mu^\dag=\prod_\Gamma\mu^\dag_\Gamma$ with $\mu^\dag_\Gamma=({\mu^\dag_\Gamma}^{(1)},\ldots,{\mu^\dag_\Gamma}^{(\beta_\Gamma e_\Gamma)})$ such that $\mu^\dag_\Gamma=\mu_\Gamma$ if $\Gamma\ne x+1$ and
${\mu^\dag_{x+1}}^{(i)}=\mu_{x+1}^{(e+i)}$,
${\mu^\dag_{x+1}}^{(e+i)}=\mu_{x+1}^{(i)}$ for every $1\le i\le e$.
\end{prop}

\section{Type $B$}\label{typeB}

In this section, we let $G=\SO_{2n+1}(q)$ and $\bG=\SO_{2n+1}(\overline\F_q)$~(a connected reductive algebraic group).
As usual, we always
denote by $F_p$ the field automorphism which sends $(a_{ij})$ to $(a_{ij}^p)$ and we write $E=\langle F_p\rangle$.
Let $F:=F_p^f$ be the standard Frobenius endomorphism over $\bG$.
We write $\bG^F$ for the group of fixed points,
then $G=\bG^F$.
As before, we denote by $\bG^*=\Sp_{2n}(\overline\F_q)$ the dual of $\bG$ and $G^*={\bG^*}^F=\Sp_{2n}(q)$.

\subsection{Characters and $\ell$-Brauer characters of $\SO_{2n+1}(q)$ and $\Omega_{2n+1}(q)$}

Let $s\in {\bG^*}^F$ be a semisimple element and
let $\Psi(s)=\prod_{\Gamma}\Psi_\Gamma(s)$ be defined as in (\ref{def-par-sym}).
Then the unipotent characters of $C_{\mathbf G^*} (s)^F$ are in bijection with  $\Psi(s)$.
For $\mu\in\Psi(s)$, we denote $\psi_\mu$ the unipotent character of $C_{\mathbf G^*} (s)^F$ corresponding to $\mu$.
Now we define  $i\Irr(G)$ to be the set of $G^*$-conjugacy classes of pairs
$(s,\mu)$ where $s$ is a semisimple element of ${\bG^*}^F$ and $\mu\in \Psi(s)$.
Here two pairs $(s_1,\mu_1)$ and $(s_2,\mu_2)$ are said to be $G^*$-conjugate if there exists $g\in G^*$ such that $s_1={}^gs_2$ and $\mu_1=\mu_2$.
The irreducible characters of $G$ have been classified by Lusztig \cite{Lu77}.
By the Jordan decomposition of characters of $G$, there is a bijection between
$\cE(C_{{\bG^*}^F}(s),1)$ and $\cE(\bG^F,s)$ for every semisimple element $s$ of $G^*$.
For $\mu\in\Psi(s)$, we denote $\chi_{s,\mu}$ the character in  $\cE(\bG^F,s)$ corresponding to $\psi_\mu$.
So $i\Irr(G)$ is a labeling set of the characters of $G$.

In this section,
we assume that $\sigma=F_p$ is the field automorphism, then $\sigma^*$ is also the field automorphism $F_p$ of $G^*$.
For $\mu\in\Psi(s)$, we define $\sigma\mu=\prod_\Gamma(\sigma\mu)_\Gamma$, with $(\sigma\mu)_{{}^{\sigma^*}\Gamma}=\mu_\Gamma$.

\begin{prop}\label{act-on-irr}
	With the above definitions, we have
	$\chi^\sigma_{s,\mu}=\chi_{{\sigma^*}^{-1}(s),{\sigma^*}^{-1}\mu}$ in the sense that the pair $(s,\mu)$ in the subscript means in fact a $G^*$-conjugacy class.
\end{prop}

\begin{proof}
	This follows by \cite[Thm.~3.1]{CS13} (or \cite[Prop.~1.3.1(iv)]{En08}) and the fact that the unipotent characters of symplectic groups with odd defining characteristic and general linear and unitary groups are invariant under the automorphism groups~(see, for example, \cite[Thm.~2.5]{Ma08}).
\end{proof}

Let $s$ be a semisimple $\ell'$-element and let $\cC(s)=\prod_{\Gamma} \cC_\Gamma(s)$ as (\ref{def-core}).
We define $i\Bl_\ell(G)$ to be the set of $G^*$-conjugacy classes of pairs $(s,\kappa)$ where $s$ is a semisimple $\ell'$-element of $G^*$ and $\kappa\in\cC(s)$.
Then by \cite[(10B)]{FS89},
there is a bijection $(s,\kappa)\mapsto B(s,\kappa)$ from $i\Bl_\ell(G)$ to $\Bl_\ell(G)$.
Also by \cite[(11A)]{FS89}, $\Irr(B(s,\kappa))\cap\cE(G,\ell')=\{ \chi_{s,\mu}\mid \mu\in\Psi(s,\kappa)\}$, where $\Psi(s,\kappa)$ is as defined in (\ref{def-par-sym-core}).

We define the action of $\sigma$ on $i\Bl_\ell(G)$ similarly as the action on $i\Irr(G)$.
The following result can be deduced directly from Proposition \ref{act-on-irr}.

\begin{prop}\label{act-block}
	$B(s,\kappa)^\sigma=B({\sigma^*}^{-1}(s),{\sigma^*}^{-1}\kappa)$.
\end{prop}

We define $i\IBr_\ell(G):=\{ (s,\mu)\in i\Irr(G)\mid s \ \text{is of} \ \ell'\text{-order}  \}$.
Then for $(t,\mu)\in i\IBr_\ell(G)$,
$\chi_{t,\mu}$ lies in the $\ell$-block $B(s,\kappa)$ if and only if $t$ is $G^*$-conjugate to $s$ and
$\mu\in\Psi(s,\kappa)$.

\begin{prop}\label{action-ibr}
	Assume that $\ell$ is linear.
	There is a bijection $(s,\mu)\mapsto \phi_{s,\mu}$ from $i\IBr_\ell(G)$ to $\IBr_\ell(G)$ such that $\phi_{s,\mu}^\sigma=\phi_{{\sigma^*}^{-1}(s),{\sigma^*}^{-1}\mu}$.
\end{prop}
\begin{proof}
	By Theorem \ref{basicset}, $\cE(\bG^F,\ell')$ is a basic set of $\IBr_\ell(G)$.
	By \cite{GH97}, the decomposition matrix with respect to  $\cE(\bG^F,\ell')$ is unitriangular since $\ell$ is linear.
	Then by \cite[Lem.~7.5]{CS13}, there is a $\sigma$-equivariant bijection from  $\cE(\bG^F,\ell')$ to $\IBr_\ell(G)$ which preserves blocks.
	Thus the assertion follows from Proposition \ref{act-on-irr}.
\end{proof}

\begin{rmk}\label{brau-dec}
In fact, by the construction in \cite[Lem.~7.5]{CS13}, the proof of Proposition \ref{action-ibr}
gives a bijection $\mathscr D:\mathcal E(\bG^F,\ell')\to \IBr_\ell(G)$ such that
 $\phi_{s,\mu}=\mathscr D(\chi_{s,\mu})$ for all $(s,\mu)\in i\IBr_\ell(G)$.
 In addition, there is a partial order relation $\le$ on $\IBr_\ell(G)$,
 such that $\chi_{s,\mu}^\circ=\phi_{s,\mu}+\sum\limits_{\varphi\in\IBr_\ell(G),\varphi\lneq \varphi} d_\varphi \varphi$ with $d_\varphi\in\mathbb Z$.
\end{rmk}

Note that $Z(G^*)=\langle z \rangle$,
where $z=-I_{(2n)}$.
For a semisimple element $s\in G^*$, we write $-s:=z s=-I_{(2n)}\cdot s$.
For $\Gamma\in\cF$, let $\xi$ be a root of $\Gamma$.
We define $z.\Gamma$ to be the unique polynomial in $\cF$ such that $-\xi$ is a root of $z.\Gamma$.
For $\mu\in\Psi(s)$, we define $-\mu=(-\mu)_{z.\Gamma}$, with $(-\mu)_{z.\Gamma}=\mu_\Gamma$.
Let $\hat z\in\cE(\bG^F,z)$ be the character corresponding under Jordan decomposition to $1_{G}\in\cE(\bG^F,1)$.
Then $\hat z $ is the (unique) non-trivial linear character of $G$~(the definition of $\hat z$ also follows from \cite[Prop. 13.30]{DM91}).
Then by \cite[Prop.~1.3.1(ii)]{En08}, we have the following result.

\begin{prop}\label{z-times-char}
	$\hat z\chi_{s,\mu}=\chi_{-s,-\mu}$ in the sense that the pair $(s,\mu)$ in the subscript means in fact a $G^*$-conjugacy class.
\end{prop}

Since $\ell$ is odd and $z$ has order $2$, the character $\hat z$ in Proposition \ref{z-times-char} can be regarded as a linear $\ell$-Brauer character of $G$.

\begin{prop}\label{dec-Brauer-char}
	With the notation of Proposition \ref{action-ibr} and \ref{z-times-char},
	if $\ell$ is linear, then
	$\hat z\phi_{s,\mu}=\phi_{-s,-\mu}$ in the sense that the pair $(s,\mu)$ in the subscript means in fact a $G^*$-conjugacy class.
\end{prop}

\begin{proof}
Here, we use \cite[Lem.~2.4]{Fe18}. By its proof, $\hat z$ induces an automorphism of the associated group algebra.
Then it permutes the irreducible ordinary and $\ell$-Brauer characters in the way indicated.
Thus the assertion follows from  Remark \ref{brau-dec}, Proposition \ref {action-ibr} and \ref{z-times-char}.
\end{proof}

\subsection{Weights of $\SO_{2n+1}(q)$}

Now we let $V$ be an odd-dimensional orthogonal space,  $\tilde G=I(V)$ and $G=I_0(V)$. Then $\tilde G=Z(\tilde G)\times G$.
Define
$$i\cW_\ell(G)=\left\{~(s,\kappa,K)^{ G^*}~\middle|~
\begin{array}{c}
s~\textrm{is a semisimple $\ell'$-element of}~ G^*, \kappa\in \mathcal C(s),\\
K=K_\Gamma,~K_\Gamma:\bigcup_\delta\sC_{\Gamma,\delta}\to\{~\ell\textrm{-cores}~\}~\textrm{s.t.}~\\
\sum_{\delta,i,j}\ell^\delta |K_\Gamma(\psi_{\Gamma,\delta,i,j})|=w_\Gamma,
m_{\Gamma}(s)/\beta_\Gamma=|\kappa_\Gamma|+e_\Gamma w_\Gamma
\end{array}~\right\}.$$
Here, $(s,\kappa,K)^{ G^*}$ means a $G^*$-conjugacy class of $(s,\kappa,K)$.

A bijection between $\cW_\ell(G)$ and $i\cW_\ell(G)$ has been constructed implicitly in the proof of \cite[(4G)]{An94} and can be described as follows.
Let $(R,\varphi)$ be an $\ell$-weight of $G$.
Then $(R,\tilde\varphi)$ is an $\ell$-weight of $\tilde G$, where $\tilde \varphi=1_{Z(\tilde G)}\times \varphi$.
Set $\tilde C=C_{\tilde G}(\tilde R)$ and $\tilde N=N_{\tilde G}(\tilde R)$.
Thus there exists an $\ell$-block $\tilde b$ of $\tilde C\tilde R$ with $\tilde R$ a defect group such that $\tilde \varphi=\Ind_{\tilde N(\theta)}^{\tilde N}\tilde \psi$ where $\tilde \theta$ is the canonical character of $\tilde b$ and $\tilde \psi$ is a character of $\tilde N(\tilde \theta)$ lying over $\tilde \theta$ and of $\ell$-defect zero as a character of $\tilde N(\tilde \theta)/\tilde R$.

Let $V_0=C_V(R)$ and $V_+=[V,R]$.
Then $V=V_0\perp V_+$ and $V_+$ is an even-dimensional orthogonal space.
Suppose that $\mathrm{dim}(V_0)=2n_0+1$.
In addition, let $\tilde G_0=I(V_0)$, $G_0=I_0(V_0)$, $\tilde G_+=I(V_+)$ and $G_+=I_0(V_+)$.
Then $R=R_0\times R_+$, $\tilde b=\tilde b_0\times \tilde b_+$, $\tilde\theta=\tilde\theta_0\times\tilde{\theta}_+$,
where $R_0=\langle 1_{V_0} \rangle\le \tilde G_0$, $R_+\le \tilde G_+$,
$\tilde b_0$, $\tilde b_+$ are $\ell$-blocks of $\tilde G_0$, $C_{\tilde G_+}(R_+)$ respectively, and $\tilde{\theta}_0\in\Irr(\tilde b_0)$,  $\tilde{\theta}_+\in\Irr(\tilde b_+)$.

First, we let $\tilde C_0=\tilde N_0=\tilde G_0$, $C_+=C_{I(V_+)}(R_+)$, and
$\tilde N_+=N_{I(V_+)}(R_+)$.
Then $\tilde \varphi_0=\tilde \psi_0=\tilde \theta_0$ a character of $\tilde G_0$ of $\ell$-defect zero.
Let $\tilde \theta_0=1_{\langle -1_{V_0} \rangle}\times \theta_0$, where $\theta_0$ is a character ($\ell$-defect zero) of $G_0\cong \SO_{2n_0+1}(q)$.
So it is of the form $\chi_{s_0,\kappa}$ where $s_0$ is a  semisimple $\ell'$-element of $G_0^*\cong \Sp_{2n_0}(q)$ and $\kappa\in \Psi(s_0)$ such that $\kappa_\Gamma$ is an $e_\Gamma$-core  which affords the second component of the triple $(s,\kappa,K)$.

Secondly, assume we have the following decomposition $\tilde  \theta_+=\prod_{\Gamma,\delta,i}\tilde \theta_{\Gamma,\delta,i}^{t_{\Gamma,\delta,i}}$, $R_+=\prod_{\Gamma,\delta,i}R_{\Gamma,\delta,i}^{t_{\Gamma,\delta,i}}.$
$\theta_\Gamma$ determines a semisimple $\ell'$-element of $G_\Gamma$ with a unique elementary divisor $\Gamma$ of multiplicity  $\beta_\Gamma e_\Gamma$.
Let $s_+$ be a semisimple element of $G_+$ such that $s_+$ has divisors $\Gamma$ with multiplicity $\prod_{\delta,i} t_{\Gamma,\delta,i}\beta_\Gamma  e_\Gamma\ell^\delta$.
Then $s=s_0\ti s_+$ is the first component of the triple $(s,\kappa,K)$.
We can view the block $\tilde b$ as a block of $C_{\tilde G}(R)$.
Thus $(R,\varphi)$ belongs to an $\ell$-block of $G$ with label $(s,\kappa)$.
In particular, $\kappa\in \mathcal C(s)$.

Finally,
the correspondence $(R,\Ind_{\tilde N(\theta)}^{\tilde N}\tilde \psi))\mapsto (R_+,\Ind_{\tilde N_+(\theta_+)}^{\tilde N_+}\tilde \psi_+))$
is a bijection from
$\{(R,\Ind_{\tilde N(\theta)}^{\tilde N}\tilde \psi))\mid \tilde \psi\in\Irr^0(\tilde N(\theta)\mid\tilde\theta) \}$ to
$\{(R_+,\Ind_{\tilde N_+(\theta_+)}^{\tilde N_+}\tilde \psi_+))\mid \tilde \psi_+\in\Irr^0(\tilde N_+(\theta_+)\mid\tilde\theta_+) \}$.
Then the third component $K=\prod_\Gamma K_\Gamma$ of the triple $(s,\kappa,K)$ is given as in the statement preceding Lemma \ref{action-basic}.

Let $(R,\varphi)$ be the $\ell$-weight of $G$ with label $(s_1,\kappa_1,K)$.
Then by the proof of \cite[(4G)]{An94},
$(R,\varphi)$ is a $B(s_2,\kappa_2)$-weight if and only if $s_1$ and $s_2$ are $G^*$-conjugate and $\kappa_1=\kappa_2$.

\begin{prop}\label{act:weights}
	Let $(R,\varphi)$ be the $\ell$-weight of $G$ with label $(s,\kappa,K)^{G^*}$.
	Then $(R,\varphi)^\sigma$ is the $\ell$-weight of $G$ with label $({\sigma^*}^{-1}(s),{{\sigma^*}^{-1}\kappa},{{\sigma^*}^{-1} K})^{G^*}$.
\end{prop}

\begin{proof}
We want to find which triple corresponds to $(R,\varphi)^\sigma$.
	Assume it be $(s',\kappa',K')$.
	First, $R^\sigma=R$, $\tilde C_0^\sigma=\tilde N_0^\sigma=\tilde C_0=\tilde N_0$, $\tilde C_+^\sigma=\tilde C_+$ and $\tilde N_+^\sigma=\tilde N_+$ by Remark \ref{action-on-rad}..
	Now, $\tilde \varphi^\sigma=\tilde \varphi_0^\sigma\times\tilde \varphi_+^\sigma$.
	$\varphi_0=\Res^{\tilde C_0}_{C_0}(\tilde\varphi_0)$ is of the form $\chi_{s_0,\kappa}$ by construction.
	By Proposition \ref{act-on-irr}, $\chi_{s_0,\kappa}^\sigma=\chi_{{\sigma^*}^{-1}(s_0),{{\sigma^*}^{-1}\kappa}}$.
	Then we have $\kappa'={{\sigma^*}^{-1}\kappa}$.
	
	Secondly, $\sigma$ stabilizes every $\tilde C_{\Gamma,\delta,i}$.
Now $\tilde\theta_{\Gamma,\delta,i}^\sigma=\tilde\theta_{{}^{{\sigma^*}^{-1}}\Gamma,\delta,i}$ corresponds to a semisimple element with a unique elementary divisor $\Gamma$ of multiplicity 
$\beta_\Gamma e_\Gamma \ell^\delta$.
	Up to conjugacy, we have $s'=\sigma(s)$.
	
	Finally, by the argument above,
	we may denote $R_{\Gamma,\delta,i}^\sigma=R_{{^\sigma\Gamma},\delta,i}$, $\tilde N_{\Gamma,\delta,i}^\sigma=\tilde N_{{^\sigma\Gamma},\delta,i}$ and
	$\tilde C_{\Gamma,\delta,i}^\sigma=\tilde C_{{^\sigma\Gamma},\delta,i}$ although the corresponding terms indexed by $\Gamma$ and $^{{\sigma^*}^{-1}}\Gamma$ are actually the same.
	Then $K'={{\sigma^*}^{-1} K}$ follows by Lemma \ref{action-basic}.
\end{proof}

Let $i\cW'_\ell(G)$ be the set of $G^*$-conjugacy classes  of triples $(s,\kappa,\mu)$
such that $s$ is a semisimple $\ell'$-element of $G^*$, $\kappa\in\cC(s)$, $\mu\in \prod_\Gamma \mathcal P(\beta_\Gamma e_\Gamma,w_\Gamma)$.
Then by \cite[Prop.(3.7)]{Ol93}, $i\cW'_\ell(G)$ is also a labeling set for $\cW_\ell(G)$.
Now by Corollary \ref{action-basic-par} and Proposition \ref{act:weights}, we have

\begin{cor}\label{act-wei-b}
Let $(R,\varphi)$ be an $\ell$-weight of $G$ corresponding to $(s,\kappa,\mu)\in i\cW'_\ell(G)$, then
$(R,\varphi)^\sigma$  corresponds to $({\sigma^*}^{-1}(s),{\sigma^*}^{-1}\kappa,{\sigma^*}^{-1}\mu)$.
\end{cor}

\begin{thm}\label{equiv-bij-so}
	Let $G=\SO_{2n+1}(q)$, where $q=p^f$ is a power of an odd prime $p$,
	$\ell\ne p$ is an odd prime and $n\ge 2$. Assume that $\ell$ is linear.
		Then the blockwise bijection between $\IBr_\ell(G)$ and $\cW_\ell(G)$ given in \cite{An94} is equivariant under  the action of field automorphisms.
\end{thm}

\begin{proof}
By Theorem \ref{basicset}, the set $\cE(G,\ell')$ is a basic set of $\IBr_\ell(G)$ and
by \cite{GH97}, the corresponding $\ell$-decomposition matrix of $G$ is lower unitriangular  since $\ell$ is linear.	
Then there is a canonical bijection $\Xi$ from $i\IBr_\ell(G)$ to $i\cW'_\ell(G)$.
By the construction of $\Xi$ there,
$\Xi$ is $E$-equivariant.
Thus the assertion follows by Proposition	\ref{action-ibr} and
Corollary \ref{act-wei-b}.
\end{proof}

\subsection{Weights of $\Omega_{2n+1}(q)$}
Recall that $S=\Omega_{2n+1}(q)$ and
$z=-I_{2n}\in Z(\Sp_{2n}(q))$.
We may identify $\hat z$ with $1_{Z(\tilde G)}\times \hat z$ and regard $\hat z$ as a linear character of $\tilde G$.
We may assume $R_{\Gamma,\delta,i}=R_{{z.\Gamma},\delta,i}$, $N_{\Gamma,\delta,i}=N_{{z.\Gamma},\delta,i}$, and
$C_{\Gamma,\delta,i}=C_{{z.\Gamma},\delta,i}$.
We also may regard $\hat{z}$ as a non-trivial linear character of $C_{\Gamma,\delta,i}$ by Lemma \ref{detofC}.
Then by \cite[Prop.~12.6]{DM91},
$\hat{z} \tilde\theta_\Gamma=\pm\hat{z} R^{\tilde C_\Gamma}_{\tilde T_\Gamma}(\widehat {s_\Gamma})=
\pm R^{\tilde C_\Gamma}_{\tilde T_\Gamma}(\widehat {-s_\Gamma})$,
and then $\hat{z} \tilde\theta_\Gamma=\tilde \theta_{z.\Gamma}$.
Thus $\hat z\theta_{\Gamma,\delta,i}=\theta_{z.\Gamma,\delta,i}$.
So we may choose the labeling of $\sC_{\Gamma,\delta}$ and $\sC_{^{{\sigma^*}^{-1}}\Gamma,\delta}$ such that
\begin{equation}\label{weight:convention-eps}
\hat z\psi_{\Gamma,\delta,i,j}=\psi_{z.\Gamma,\delta,i,j}.
\end{equation}

We define $-K=\prod_{\Gamma} (-K)_{z.\Gamma}$ with $ (-K)_{z.\Gamma}=K_\Gamma$.
Since $N_G(R)/N_S(R)\cong G/S$, we may regard $\hat z$ as a linear character of $N_G(R)$ whose  kernel is $N_S(R)$.

\begin{prop}\label{dec-weig}
Let $(R,\varphi)$ be the $\ell$-weight of $G$ with label $(s,\kappa,K)^{G^*}$.
Then $(R,\hat z\varphi)$ is the $\ell$-weight of $G$ with label $(-s,-\kappa,-K)^{G^*}$.
\end{prop}

\begin{proof}
We want to find which triple corresponds to $(R,\hat z\varphi)$.
	Assume it be $(s',\kappa',K')$.	
	Now $\tilde\varphi=1_{Z(\tilde G)}\times\varphi$, so $\hat z\tilde\varphi=1_{Z(\tilde G)}\times\hat z\varphi$.	
	First, $\hat z\tilde \varphi=\hat z\tilde \varphi_0\times\hat z\tilde \varphi_+$.
	$\varphi_0=\Res^{\tilde G_0}_{G_0} \tilde \varphi_0$ is of the form $\chi_{s_0,\kappa}$ by construction.
	By Proposition \ref{z-times-char}, $\hat z\chi_{s_0,\kappa}=\chi_{-s_0,-\kappa}$.
	Then we have $\kappa'={-\kappa}$.
	
	Secondly,
	we have	$\hat z\tilde \theta_{\Gamma,\delta,i}=\tilde \theta_{z.\Gamma,\delta,i}$ as above.
	Note that
	$\hat z\tilde\theta_{\Gamma,\delta,i}$ corresponds to a semisimple element with a unique elementary divisor $\Gamma$ of multiplicity  
	$\beta_\Gamma e_\Gamma\ell^\delta$ and $\tilde\theta_{z.\Gamma,\delta,i}$ corresponds to 
	a semisimple element with a unique elementary divisor $z.\Gamma$ of multiplicity  
	$\beta_\Gamma e_\Gamma\ell^\delta$.
	Up to conjugacy, we have $s'=-s$.
	
	Finally, by the conventions above, we may assume $R_{\Gamma,\delta,i}=R_{{z.\Gamma},\delta,i}$, $\tilde N_{\Gamma,\delta,i}=\tilde N_{{z.\Gamma},\delta,i}$, and
	$\tilde C_{\Gamma,\delta,i}=\tilde C_{{z.\Gamma},\delta,i}$.
	To determine $K'$, we note that $\hat z\tilde \psi_+=\prod_{\Gamma,\delta,i}\hat z \tilde \psi_{\Gamma,\delta,i}$.
	By (\ref{weights:psi}), $\hat z\tilde \psi_{\Gamma,\delta,i}$ is
	\begin{align*}
	&\hat z\Ind
	_{\tilde N_{\Gamma,\delta,i}(\theta_{\Gamma,\delta,i})\wr
		\prod_j\fS(t_{\Gamma,\delta,i,j})}
	^{\tilde N_{\Gamma,\delta,i}(\theta_{\Gamma,\delta,i})\wr
		\fS(t_{\Gamma,\delta,i})}
	\left(\overline{\prod_j\tilde \psi_{\Gamma,\delta,i,j}^{t_{\Gamma,\delta,i,j}}}\right)\cdot
	\prod_j\phi_{\kappa_{\Gamma,\delta,i,j}}\\
	=&\Ind
	_{\tilde N_{z.\Gamma,\delta,i}(\theta_{z.\Gamma,\delta,i})\wr
		\prod_j\fS(t_{\Gamma,\delta,i,j})}
	^{\tilde N_{z.\Gamma,\delta,i}(\theta_{z.\Gamma,\delta,i})\wr
		\fS(t_{\Gamma,\delta,i})}\hat z
	\left(\overline{\prod_j\tilde \psi_{\Gamma,\delta,i,j}^{t_{\Gamma,\delta,i,j}}}\right)\cdot
	\prod_j\phi_{\kappa_{\Gamma,\delta,i,j}}.
	\end{align*}
	Since $\hat z\tilde \theta_{\Gamma,\delta,i}=\tilde \theta_{z.\Gamma,\delta,i}$, we have $\tilde N_{\Gamma,\delta,i}(\tilde \theta_{\Gamma,\delta,i})=
\tilde 	N_{{z.\Gamma},\delta,i}(\tilde \theta_{{z.\Gamma},\delta,i})$.
		We can fix the way to extend $\prod_j\tilde \psi_{\Gamma,\delta,i,j}^{t_{\Gamma,\delta,i,j}}$ as in \cite[Lem. 25.5]{Hu98}, then we have that
	$\hat z\left(\overline{\prod_j\tilde \psi_{\Gamma,\delta,i,j}^{t_{\Gamma,\delta,i,j}}}\right)
	=\overline{\prod_j\left(\hat z\tilde \psi_{\Gamma,\delta,i,j}\right)^{t_{\Gamma,\delta,i,j}}}.$
	Since $\hat z\tilde \psi_{\Gamma,\delta,i,j}=\tilde \psi_{z.\Gamma,\delta,i,j}$ by (\ref{weight:convention-eps}), $\hat z\tilde \psi_{\Gamma,\delta,i}$ would be
	$$\Ind
	_{\tilde N_{{z.\Gamma},\delta,i}(\tilde \theta_{{z.\Gamma},\delta,i})\wr
		\prod_j\fS(t_{\Gamma,\delta,i,j})}
	^{\tilde N_{{z.\Gamma},\tilde \delta,i}(\theta_{{z.\Gamma},\delta,i})\wr
		\fS(t_{\Gamma,\delta,i})}
	\overline{\prod_j\tilde \psi_{{z.\Gamma},\delta,i,j}^{t_{\Gamma,\delta,i,j}}}\cdot
	\prod_j\phi_{\kappa_{\Gamma,\delta,i,j}}.$$
	Then $K'_{z.\Gamma}=K_\Gamma$ which is just $K'={z. K}$.
	Thus we complete the proof.
\end{proof}

\begin{cor}\label{act-wei-b-omega}
	Let $(R,\varphi)$ be an $\ell$-weight of $G$ corresponding to $(s,\kappa,\mu)\in i\cW'_\ell(G)$, then
	$(R,\hat z\varphi)$  corresponds to $(-s,-\kappa,-\mu)$.
\end{cor}

	Let $\mathcal W_\ell(G)$ be a complete set of representatives of all $G$-conjugacy classes of $\ell$-weights of $G$. We may assume that for $(R_1,\varphi_1),(R_2,\varphi_2)\in \mathcal W_\ell(G)$, $R_1$ and $R_2$ are $G$-conjugate if and only if $R_1=R_2$.
	
	Let $S=\Omega_{2n+1}(q)$, then $\Rad_\ell(G)=\Rad_\ell(S)$.
	Now define a equivalence relation on $\mathcal W_\ell(G)$ such that for $(R_1,\varphi_1),(R_2,\varphi_2)\in \mathcal W_\ell(G)$, $(R_1,\varphi_1)\sim(R_2,\varphi_2)$ if and only if $R_1=R_2$ and $\varphi_1=\varphi_2 \eta$ for some $\eta\in\Irr(N_G(R_1)/N_S(R_1))$.
	Then by  \cite[Lem.~2.4]{Fe18} and Corollary \ref{relaofradicalconj},
	the set $\{(R,\psi)\}$, where $(R,\varphi)$ runs through a complete set of representatives of the equivalence classes of $\mathcal W_\ell(G)/\sim$ and $\psi$ runs through $\Irr(N_S(R)\ |\ \varphi)$, is a complete set of representatives of all $S$-conjugacy classes of $\ell$-weights of $S$.

	Let $(R,\varphi)$ be an $\ell$-weight of $G$, $(R,\psi)$ an $\ell$-weight of $S$ such that $\varphi\in\Irr(N_G(R)\ | \ \psi)$.
	Let $b=\bl_\ell(\varphi)$, $b_0=\bl_\ell(\psi)$ and $B=b ^G$ and $B_0=b_0^S$.
	By \cite[Lem. 2.3]{KS15}, if $b$ covers $b_0$, then $B$ covers $B_0$.
	
	Let $B_0$ be an $\ell$-block of $S$.
	Denote by $\mathcal B_0$ the union of the $\ell$-blocks of $S$ which are $G$-conjugate to $B_0$ and ${\mathcal B}$  the union of the $\ell$-blocks of $G$ which cover $B_0$.
	Then
	\begin{itemize}
		\item if $(R,\varphi)$ is an $\ell$-weight of $G$ belonging to ${\mathcal B}$, then for every $\psi\in\Irr(N_S(R)\ | \ \varphi)$, $(R,\psi)$ is an $\ell$-weight of $S$ belonging to $\mathcal B_0$, and
		\item if $(R,\psi)$ is an $\ell$-weight of $S$ belonging to $\mathcal B_0$, then there exists $\varphi\in\Irr(N_G(R)\ | \ \psi)$ such that $(R,\varphi)$ is an $\ell$-weight of $G$ belonging to ${\mathcal B}$.
	\end{itemize}

\begin{prop}\label{alp-Omega}
	Let $q$ be a power of an odd prime and $\ell$ an odd prime. Assume that $\ell$ is linear.
Then the Alperin weight conjecture \ref{weiconj} holds for every $\ell$-block of the
 group $S=\Omega_{2n+1}(q)$.
\end{prop}

\begin{proof}
	The proof here is analogous to \cite[Thm.~1.2]{Fe18}.
Let $\Theta$ be the canonical blockwise bijection between $\IBr_\ell(G)$ and $\cW_\ell(G)$.
For $\phi\in \IBr_\ell(G)$, let $(R,\varphi)=\Theta(\phi)$.
By Proposition \ref{dec-Brauer-char} and Corollary
\ref{act-wei-b-omega},
$\hat z \phi=\phi$ if and only if $\hat z \varphi=\varphi$.
Thus the assertion follows by the argument above.
\end{proof}

By \cite[\S2.5]{GLS98}, $\Aut(S)\cong G\rtimes E$, where $E=\langle F_p\rangle$.

\begin{thm}\label{bi-simple-b}
	Let $S=\Omega_{2n+1}(q)$, where $q=p^f$ is a power of an odd prime $p$,
$\ell\ne p$ is an odd prime and $n\ge 2$.
Assume further that $f$ is odd and $\ell$ is linear.
	Then there exists a blockwise $\Aut(S)$-equivariant bijection between $\IBr_\ell(S)$ and $\cW_\ell(S)$.
\end{thm}
\begin{proof}
It is analogous to the proof of \cite[Prop.~5.19]{Fe18}.
By Theorem \ref{equiv-bij-so} and the proof of Propsition \ref{alp-Omega}, it suffice to show that for any $\phi\in\IBr_\ell(G)$ and any $(R,\varphi)\in\mathcal W_\ell(G)$, 	
$E$ acts trivially on $\IBr_\ell(S\mid \phi)$ and $N_{G\rtimes E}(R)$ acts trivially on $\Irr(N_S(R)\mid \varphi)$.
Let $m_1$ be the length of an orbit of $E$ on $\IBr_\ell(S\mid \phi)$
and $m_2$ the length of an orbit of $N_{G\rtimes E}(R)$ on $\Irr(N_S(R)\mid \varphi)$.
Then $m_1,m_2\le 2$, $m_1\mid f$ and $m_2\mid f$.
Now $f$ is odd, so $m_1=m_2=1$.
This completes the proof.	
\end{proof}

Note that $X=\mathrm{Spin}_{2n+1}(q)$ is the universal $\ell'$-covering group of the simple group $S=\Omega_{2n+1}(q)$ (with $n\ge 3$ and $q$  odd) unless when $n=q=3$ by \cite[\S~6.1]{GLS98}.

\begin{proof}[Proof of Theorem \ref{ibawc-dominate}]
By assumption, $\Aut(S)/S$ is cyclic for $S=X/Z(X)= \Omega_{2n+1}(q)$.	
Thus by Corollary \ref{cyclic-outer-ind}, it suffices to show that there exists a blockwise $\Aut(X)_B$-equivariant bijection between $\IBr_\ell(B)$ and $\cW_\ell(B)$.

By Lemma \ref{ind-quot}, it suffices to show that there exists a blockwise $\Aut(S)$-equivariant bijection between $\IBr_\ell(S)$ and $\cW_\ell(S)$,
which follows by Theorem \ref{bi-simple-b}.	
\end{proof}

\subsection{The unipotent blocks}

\begin{prop}\label{unip-blo-B}
Let $G=\SO_{2n+1}(q)$ and $S=\Omega_{2n+1}(q)$ with $n\ge 2$ and $q$ odd, and $\ell\nmid q$ an odd prime.
Suppose $B$ is a unipotent $\ell$-block of $G$.
Then 
\begin{itemize}
	\item $B$ covers a unique $\ell$-block $b$ of $S$,
	\item $\Res_S^G: \IBr_\ell(B)\to \IBr_\ell(b)$ is bijective, and
	\item $\mathcal W_\ell( B)\to\mathcal W_\ell(  b) $, $( R, \varphi)\mapsto (   R,\Res^{N_G(  R)}_{N_S(  R)}  \varphi)$ is a bijection.
\end{itemize}
\end{prop}

\begin{proof}
For $\ell$-weights, this follows by Proposition  \ref{dec-weig}.
Now we prove that $\Res_S^G: \IBr_\ell(B)\to \IBr_\ell(b)$ is bijective.
First, $\hat z\otimes B$ is also an $\ell$-block of $G$ covering $b$~(for the definition of $\hat z\otimes B$, see Page \pageref{def-otimes-block}).
Also, $\hat z\otimes B\ne B$ since $\hat z\otimes B\subseteq \mathcal E_\ell(G,z)$.
Thus there are two $\ell$-blocks covering $b$.
So by Lemma \ref{ext-Br},
$\Res_S^G: \IBr_\ell(B)\to \IBr_\ell(b)$ is bijective.
\end{proof}

\begin{lem}\label{extension-B}
Let $G=\SO_{2n+1}(q)$ with $n\ge 2$ and $q$ odd, and $\ell\nmid q$ an odd prime.
Suppose $B$ is a unipotent $\ell$-block of $G$.
Let $\phi\in\IBr_\ell(B)$ and $(R,\varphi)\in\mathcal W_\ell(B)$.
Then 
\begin{itemize}
	\item  $\phi$ extends to $G\rtimes E$, and
	\item  $\varphi$ extend to  $(G\rtimes E)_{R}$.
\end{itemize}
\end{lem}

\begin{proof}
By Theorem \ref{basicset}, the set of unipotent characters of $B$ form a basic set of $B$.
By \cite[Thm.~2.5]{Ma08}, every unipotent character of $G$ is $E$-invariant.
So every $\ell$-Brauer character of $B$ is $E$-invariant.
On the other hand we have $\varphi$ is  $(G\rtimes E)_{R}$-invariant
by Corollary \ref{act-wei-b}.
Thus this assertion holds since $E$ is cyclic.
\end{proof}

\begin{proof}[Proof of Theorem \ref{ibawc-uni} for type $B_n$]
Let $b$ be a unipotent $\ell$-block of $X$.
Then $b$ dominates an $\ell$-block $\bar b$ of $S=\Omega_{2n+1}(q)$.
Thus there exists a unique unipotent $\ell$-block $\bar B$ of $G=\SO_{2n+1}(q)$ which covers $\bar b$.
By Proposition \ref{unip-blo-B},
$\Res_S^G: \IBr_\ell(\bar B)\to \IBr_\ell(\bar b)$ is bijective and
$\mathcal W_\ell(\bar B)\to\mathcal W_\ell(\bar b) $, $(\bar R,\bar \varphi)\mapsto (\bar  R,\Res^{N_G(\bar R)}_{N_S(\bar R)}\bar \varphi)$ is a bijection.
Thus by Theorem \ref{equiv-bij-so},
there exists a blockwise $\Aut(S)_{\bar b}$-equivariant bijection between $\IBr_\ell(\bar b)$ and $\cW_\ell(\bar b)$.
This gives a blockwise $\Aut(X)_b$-equivariant bijection between $\IBr_\ell(b)$ and $\cW_\ell(b)$ by Lemma \ref{ind-quot}.
Then by Lemma \ref{first-2-con}, it suffice to show the condition (iii) in Definition \ref{induc}.

For every $Q\in\Rad_\ell(X)$ and every $\phi\in\IBr_\ell(b\mid Q)$, we let $A:=A(\phi, Q)=G\rtimes E$.
Note that all irreducible character of $b$ have $Z(X)$ in their kernel.
By Theorem \ref{basicset}, $\Irr(b)\cap \mathcal E(X,\ell')$ is a basic set of $b$, so all irreducible $\ell$-Brauer characters of $b$ have $Z(X)$ in their kernel.
By Lemma \ref{wei-quotient}~(iii), all weight characters of $b$ also have $Z(X)$ in their kernel.
Thus conditions (1)-(3) of Definition \ref{induc} (iii) hold by Proposition \ref{unip-blo-B} and Lemma \ref{extension-B}.
For Definition \ref{induc}~(iii)(4), the proof of \cite[Lem.~7.2]{Fe18} applies here.
Then the (iBAW) condition holds for $b$, which completes the proof.
\end{proof}

\section{Type $C$}\label{typeC}

In this section, we denote by $\bG=\Sp_{2n}(\overline\F_q)$,
$\tilde\bG=\CSp_{2n}(\overline\F_q)$ with $q$ odd and $n\ge 2$.
Let $F_p : \tilde\bG \to \tilde\bG$
be the raising of matrix entries to the $p$-th power,
$F=F_p^f$, for  some $f$ such that $q=p^f$.
Let $G=\bG^F=\Sp_{2n}(q)$, $\tilde G=\tilde \bG^F=\CSp_{2n}(q)$.
Then $|\tG/GZ(\tG)|=2$.
We denote by $V$ the underlying space of $\tilde G$ and $G$.

\subsection{The characters of symplectic groups}

First note that ${\tilde\bG}^*$ is the corresponding special Clifford group~(then ${\tilde G}^*=({\tilde\bG}^*)^F$ is a special Clifford group over $\F_q$) and $\bG^*=\SO_{2n+1}(\overline\F_q)$.
For the definition of the Clifford groups of orthogonal spaces, see \cite[\S9]{Gr02} or \cite[\S~5.4]{KL90}.
Note that the special Clifford groups are called even Clifford groups in \cite{Gr02}.
Thus there is a natural epimorphism $\pi:{\tilde\bG^*}\to \bG^*$~(see also \cite[(2.3)]{FS89}).
Clearly, $\pi({\tilde G}^*)={G^*}={\bG^*}^F=\SO_{2n+1}(q)$.
For a semisimple element $s$ of $\tG^*$,
we write $\bar s=\pi(s)$.
Note that $m_{x-1}(\bar s)$ is odd and $m_{x-1}(\bar s)$ is even.
In particular, $m_{x-1}(\bar s)\ne 0$.
Let $i\Irr(\tilde G)$  be the set of $\tG^*$-conjugacy classes of pairs $(s,\mu)$, where $s$ is a semisimple $\ell'$-element of $G^*$ and $\mu\in\Psi(\bar s)$~(where $\Psi(\bar s)$ is defined as (\ref{def-par-sym})).
Here, $(s_1,\mu_1)$ and $(s_2,\mu_2)$ are conjugate if and only if $s_1$ and $s_2$ are $\tG^*$-conjugate and $\mu_1=\mu_2$.
With the parametrization of pairs involving semisimple elements and unipotent characters, the irreducible characters of $\tG$ were constructed by Lusztig \cite{Lu77}; by Jordan decomposition of characters, there is a bijection from $i\Irr(\tilde G)$ to $\Irr(\tilde G)$~(see also \cite[(4.5)]{FS89}).
We write $\tilde\chi_{s,\mu}$ for the character of $\tG$ corresponding to $(s,\mu)$.

In this section, we always assume that $\sigma=F_p$ is the field automorphism and $E=\langle F_p \rangle$ as above.
Then $\sigma^*$ is also a field automorphism.
Note that $\sigma^*$ commutes with $\pi$.

\begin{lem}\label{act-irr-scp}
	Let $(s,\mu)\in i\Irr(\tG)$. Then $\tilde\chi_{s,\mu}^{\sigma}=\tilde\chi_{{\sigma^*}^{-1}(s),{\sigma^*}^{-1}\mu}$.
\end{lem}

\begin{proof}
	Similar as in Proposition \ref{act-on-irr},
	this follows from \cite[Thm.~3.1]{CS13}  and the fact that every unipotent character of groups of type $A$, $^2A$, $B$, $D$ and $^2D$ is invariant under field automorphisms~(see \cite[Thm.~2.5]{Ma08}).
\end{proof}

We will make use of the following result.
\begin{thm}
	\label{action-irr-const}
	Let $\tilde\chi\in\Irr(\tG)$ and $\Delta=\Irr(G\mid \tilde\chi)$.
	Then $E_\Delta$ acts trivially on $\Delta$.
\end{thm}

\begin{proof}
	This is \cite[Thm.~3.1]{CS17}~(or \cite[Thm.~16.2]{Ta18}).
\end{proof}

If $\chi\in\Irr(G\mid \tilde\chi_{s,\mu})$, then we say $\chi$ corresponds to the pair $(\bar s,\mu)$.

\begin{lem}\label{res-sp}
	Let $(s,\mu)\in i\Irr(\tG)$, $\bar s=\pi(s)$
	and $\tilde\chi=\tilde\chi_{s,\mu}$.
	\begin{enumerate}[(i)]
		\item  If $-1$ is not an eigenvalue of $\bar s$,
		then $\Res^{\tG}_G\tilde\chi$ is irreducible.
		\item If  $-1$ is an eigenvalue of $\bar s$, then
		\begin{itemize}
			\item if $\mu_{x+1}$ is degenerate, then $\Res^{\tG}_G\tilde\chi$ is irreducible, and
			\item if $\mu_{x+1}$ is non-degenerate, then $\Res^{\tG}_G\tilde\chi$ is a sum of two irreducible constituents.
		\end{itemize}
	\end{enumerate}
\end{lem}

\begin{proof}
	First note that $\bG\hookrightarrow\tilde \bG$ is a  regular embedding.
	Let $\psi_\mu$ be the unipotent character of $C_{\tilde \bG}(s)^F$ corresponding to $\mu$.
	Then we may regard $\psi_\mu$ as a unipotent character of $C^\circ_{\bG}(\bar s)^F$ since $\cE(C_{\tilde \bG}(s)^F,1)=\cE(C^\circ_{\bG}(\bar s)^F,1)$.
	By Jordan decomposition (cf. \cite[Prop. 5.1]{Lu88}),
	$|\Irr(G\mid\tilde\chi)|=|\Irr(C_{\tilde\bG}(s)^F\mid \psi_\mu)|$.

	Let $V^*$ be the underlying space of $G^*$ and let $\bar s=\prod_\Gamma \bar s(\Gamma)$ and $V^*=\sum_\Gamma V^*_\Gamma(\bar s)$ be the primary decomposition.
	In particular, we abbreviate $V^*_1:=V^*_{x-1}(\bar s)$, $V^*_{-1}:=V^*_{x+1}(\bar s)$ and $V^*_\Gamma:=V^*_\Gamma(\bar s)$.
	Here, $V^*_1$ is of odd-dimension and $V^*_{-1}$ is of even-dimension.
Then $C_{\bG}(\bar s)^F=(\SO(V^*_1\perp V^*_{-1})\cap(\GO(V^*_1)\times \GO(V^*_{-1})))\times \prod_\Gamma C_{\GO(V^*_\Gamma)}(\bar s(\Gamma))$,
$C^\circ_{\bG}(\bar s)^F=\SO(V^*_1)\times \SO(V^*_{-1})\times \prod_\Gamma C_{\GO(V^*_\Gamma)}(\bar s(\Gamma))$
and
$C_{\GO(V^*_\Gamma)}(\bar s(\Gamma))\le \SO(V^*_\Gamma)$.
We also write $\psi_\mu=\prod_\Gamma\psi_{\mu_\Gamma}$ where
$\psi_{x-1}:=\psi_1$ and $\psi_{x+1}:=\psi_{-1}$ are the unipotent characters of $\SO(V^*_1)$ and $\SO(V^*_{-1})$
corresponding to $\mu_{x-1}$ and $\mu_{x+1}$
respectively and
 $\psi_{\mu_\Gamma}$ is the unipotent character of $C_{\SO(V^*_\Gamma)}(\bar s_\Gamma)$ corresponding to $\mu_\Gamma$ for $\Gamma\in\cF_1\cup\cF_2$.
If $-1$ is not an eigenvalue of $\bar s$, then $C_{\bG}(\bar s)$ is connected and then $\Res^{\tG}_G\tilde\chi$ is irreducible.

Now let $-1$ be an eigenvalue of $\bar s$.
Then $|C_{\bG}(\bar s)^F/C_{\bG}^\circ(\bar s)^F|=2$ and
$\Res^{\tG}_G\tilde\chi$ is irreducible if and only if $\psi_\mu$ is not $C_{\bG}(\bar s)^F/C_{\bG}^\circ(\bar s)^F$-invariant.
Let $g=g_{1}\times g_{-1}\times\prod_\Gamma g_\Gamma$ with $g_1\in\GO(V^*_1)$, $g_{-1}\in \GO(V^*_{-1})$ and $g_\Gamma\in C_{\GO(V^*_\Gamma)}(\bar s(\Gamma))$ such that $g_{1}$ and $g_{-1}$ are of determinant $-1$.
Then $\psi_\mu=\psi_1^{g_1}\times \psi_{-1}^{g_{-1}} \times\prod_\Gamma \psi_{\mu_\Gamma}^{g_\Gamma}=\psi_1\times \psi_{-1}^{g_{-1}} \times\prod_\Gamma \psi_{\mu_\Gamma}$ by \cite[Thm.~2.5]{Ma08} and then $\psi_\mu^g=\psi_\mu$ if and only if $\kappa_{x+1}$ is non-degenerate by \cite[Thm.~2.5]{Ma08} again.
Thus the assertion holds.	
\end{proof}

\begin{rmk}
\label{res-to-so-deg}
	In Lemma \ref{res-sp} (ii), if $\mu_{x+1}$ is degenerate,
	then by Jordan decomposition,
	$\Res^\tG_G\tilde \chi_{s,\mu}=\Res^\tG_G\tilde \chi_{s,\mu'}$, where $\mu'$ is defined as in \S\ref{notations-and-conventions}.
	Thus  $\Irr(\tG\mid \Res^\tG_G\tilde \chi_{s,\mu})\cap \cE(\tG,s)=\{\tilde \chi_{s,\mu},\tilde \chi_{s,\mu'}\}$.
\end{rmk}

\subsection{The blocks of symplectic groups}
\label{The-blocks-of-special-symplectic-groups}

Recall that we let $\ell$ be an odd prime with $\ell\nmid q$ and $e_0$ the multiplicative order of $q$ modulo $\ell$.
The labeling of $\ell$-blocks of $\tilde\bG^F$ and $\mathbf G^F$~(using $e_0$-Jordan-cuspidal pairs)~ described in  \cite{CE99} and \cite{KM15} can be stated as following.

\begin{thm}\label{ecuspidalofblock}
	Let $\mathbf H\in\{\tilde\bG, \bG\}$ and $e_0=e_0(q,\ell)$ as defined in Equation (\ref{definitionofe0}).
	\begin{enumerate}
		\item[(i)] For any $e_0$-Jordan-cuspidal pair $(\mathbf L,\zeta)$ of $\mathbf H$ such that
		$\zeta\in\mathcal E(\mathbf L^F,\ell')$, there exists a unique $\ell$-block $b_{\mathbf H^F}(\mathbf L,\zeta)$ of $\mathbf H^F$ such that all irreducible constituents of $R_{\mathbf L}^{\mathbf H}(\zeta)$ lie in $b_{\mathbf H^F}(\mathbf L,\zeta)$.
		\item[(ii)] Moreover, the map $\Xi: (\mathbf L,\zeta)\mapsto b_{\mathbf H^F}(\mathbf L,\zeta)$ is a bijection from the set of $\mathbf H^F$-conjugacy classes of $e_0$-Jordan-cuspidal pairs $(\mathbf L,\zeta)$ of $\mathbf H$ such that
		$\zeta\in\mathcal E(\mathbf L^F,\ell')$ to the $\ell$-blocks of $\mathbf H^F$.
	\end{enumerate}
\end{thm}

Now we give the relationship between the  $e_0$-cuspidal pairs of $\tilde\bG$  and the  $e_0$-cuspidal pairs of $\bG$.
\begin{prop}\label{restrofesplit}
	\begin{enumerate}
		\item[(i)] Let $(\tilde\bL,\tilde\zeta)$ be an $e_0$-cuspidal pair of $\tilde\bG$ and $B$ an $\ell$-block of $\bG^F$ covered by $\tilde B=b_{\tilde\bG^F}(\tilde\bL,\tilde\zeta)$, then $B=b_{\bG^F}(\bL,\zeta)$, where $\mathbf L=\tilde\bL\cap \bG$ and $\zeta$ is an irreducible constituent of $\Res^{\tilde\bL^F}_{\mathbf L^F}\tilde\zeta$.
		\item[(ii)] Let $(\bL,\zeta)$ be an $e_0$-cuspidal pair of $\bG$ and $\tilde B$ an $\ell$-block of $\tilde \bG^F$ which covers $B=b_{\bG^F}(\bL,\zeta)$, then $\tilde B=b_{\tilde\bG^F}(\tilde\bL,\tilde\zeta)$ for some $e_0$-cuspidal pair $(\tilde\bL,\tilde\zeta)$ satisfying that $\mathbf L=\tilde\bL\cap \bG$ and $\zeta$ is an irreducible constituent of $\Res^{\tilde\bL^F}_{\mathbf L^F}\tilde\zeta$.
	\end{enumerate}
\end{prop}
\begin{proof}
This  follows by \cite[Lem.~3.7 and 3.8]{KM15}~(see \cite[Prop.~4.5]{Fe18} for details).
\end{proof}

Note that we have $\tilde\bL=Z(\tilde \bG)\bL$ in Proposition \ref{restrofesplit}.
In fact, the $F$-stable Levi subgroups of $\tilde\bG$  and $\bG$ have been classified in \cite[(3A) and (3B)]{FS89}.

\begin{lem}\label{conofcharoflevi}
	Let $\tilde \bL$ be an $F$-stable Levi subgroup of $\tilde\bG$, $\tilde\zeta\in\Irr(\tilde\bL^F)$ and $\mathbf L=\tilde \bL\cap \mathbf G$.
	Let $\Delta:=\Irr(\mathbf L^F\mid \tilde\zeta)$, then  $N_{\mathbf G^F}(\mathbf L)_\Delta$ acts trivially on $\Delta$.
\end{lem}

\begin{proof}
	Let $L=\bL^F$ and $\tilde L=\tilde\bL^F$.
Follow \cite[(3A) and (3B)]{FS89},
 we may assume that there is an orthogonal decomposition
$V=V_0\perp V_+$ of $V$, where $V_+=\sum\limits_{i=1}^{t} V_i$ such that
\begin{itemize}
	\item  $L=L_0\times L_+$, where $L_0=\Sp(V_0)$, $L_+=\prod\limits_{i=1}^{t} L_i$ such that $L_i\le \Sp(V_i)$ isomorphic to some general linear or unitary group for $1\le i\le t$.
	\item $\tilde L=\langle \tau, L \rangle$, where $\tau$ satisfies $\tG=\langle G,\tau\rangle$ and $[\tau, L_+]=1$.
	Moreover, $\tau=\tau_0\times\tau_+$ such that $\tau_0\in\CSp(V_0)$ and $\tau_+\in\CSp(V_+)$.
\end{itemize}
Thus, $|\tilde L/LZ(\tilde L)|=2$ and
$N_G(\bL)=N_0\times N_+$, with $N_0=L_0$ and $N_+\le \Sp(V_+)$.
So $|\Delta|\le 2$.
If $|\Delta|=1$, then the assertion is obvious.
Now we  may assume that  $|\Delta|=2$.

Let $\Delta=\{\zeta, \zeta'\}$, then $\zeta$ and $\zeta'$ are $\tilde L$-conjugate.
We write $\zeta=\zeta_0\times \zeta_+$ and $\zeta'=\zeta'_0\times \zeta'_+$ with $\zeta_0,\zeta'_0\in\Irr(L_0)$ and $\zeta_+,\zeta'_+\in\Irr(L_+)$.
Since $[\tau, L_+]=1$, we know $\zeta_+=\zeta'_+$.
Hence $\zeta_0\ne \zeta'_0$.
For $n\in N_G(\bL)_\Delta$,
we let $n=n_1\times n_+$ where $n_0\in N_0$ and $n_+\in N_+$.
If $\zeta^n=\zeta'$, then $\zeta_0^{n_0}=\zeta'_0$ and this is impossible since $n_0\in L_0$.
So $\zeta^n\ne \zeta'$, which implies that
$N_G(\bL)_\Delta$ acts trivially on $\Delta$.
\end{proof}

\begin{rmk}\label{esplitofslsu}
	Let $\tilde\bL$ an $F$-stable Levi subgroup of $\tilde\bG$, and $\mathbf L=\tilde\bL\cap \mathbf G$.
    Then  by (4) of \cite[(3B)]{FS89} $\tilde\bL^F/\mathbf L^F\cong \tilde\bG^F/\mathbf G^F$ and then $\tilde \bG^F=\mathbf G^FN_{\tilde \bG^F}(\mathbf L)$.
	So the $\tilde\bG^F$-conjugacy classes of $e_0$-split Levi subgroups of $\bG$ are just the $\bG^F$-conjugacy classes of $e_0$-split Levi subgroups of $\bG$.
	
	We denote by $\widetilde{\mathcal L}$ a complete set of representatives of the $\tilde\bG^F$-conjugacy classes of $e_0$-Jordan-cuspidal pairs of $\tilde\bG$ such that
	$\tilde \zeta\in\mathcal E(\tilde\bL^F,\ell')$.
	We may assume that for $(\tilde\bL,\tilde \zeta)$, $(\tilde\bL',\tilde \zeta')\in \widetilde{\mathcal L}$, if  $\tilde\bL$ and $\tilde\bL'$ are $\tilde\bG^F$-conjugacy, then $\tilde\bL=\tilde\bL'$.	
	Now we define an equivalence relation on $\widetilde{\mathcal L}$ :  $(\tilde\bL,\tilde \zeta)\sim (\tilde\bL',\tilde \zeta')$ if and only if $\tilde\bL=\tilde\bL'$ and $\Res^{\tilde\bL^F}_{\bL^F}\tilde \zeta=\Res^{\tilde\bL^F}_{\bL^F}\tilde \zeta'$ where $\mathbf L=\tilde\bL\cap \mathbf G$.
	Then by Proposition \ref{restrofesplit}, Lemma \ref{conofcharoflevi}  and Clifford theory, $\{(\tilde\bL\cap \mathbf G,\zeta)\}$  is a complete set of representatives of $\mathbf G^F$-conjugacy classes of $e_0$-Jordan-cuspidal pairs of $\mathbf G$ such that
	$\zeta\in\mathcal E((\tilde\bL\cap\mathbf G )^F,\ell')$, where $(\tilde\bL,\tilde\zeta)$ runs through a complete set of representatives of the equivalence classes of $\widetilde{\mathcal L}/\sim$ and $\zeta$ runs through $\Irr( (\tilde\bL\cap \mathbf G)^F \ |\ \tilde\zeta)$.
\end{rmk}

Now we recall the classification of $\ell$-blocks of $\tilde\bG^F$ in \cite[\S11]{FS89}.
Let $i\Bl_\ell(\tG)$ be the set of $\tG^*$-conjugacy classes of pairs $(s,\kappa)$ where $s$ is a semisimple $\ell'$-element of $\tG^*$ and $\kappa\in\cC(\bar s)$~(where $\cC(\bar s)$ is defined as in (\ref{def-core})).
Here, $(s_1,\kappa_1)$ and $(s_2,\kappa_2)$ are $\tG^*$-conjugate if and only if $s_1$ and $s_2$ are $\tG^*$-conjugate and $\kappa_1=\kappa_2$.
Then  \cite[(11E)]{FS89} gives a bijection $(s,\kappa)\mapsto \tilde B(s,\kappa)$ from $i\Bl_\ell(\tG)$ to $\Bl_\ell(\tG)$.

For $(s,\kappa)\in i\Bl_\ell(\tG)$,
\cite[(13B)]{FS89} also gave a criterion for when an irreducible character of $\tG$ lies in the $\ell$-block $\tilde B=\tilde B(s,\kappa)$.
In particular, the irreducible characters of $\Irr(\tilde B)\cap\cE(\tG,\ell')$ are of the form $\tilde\chi_{s,\mu}$ with $\mu\in\Psi(\bar s,\kappa)$~(where $\Psi(\bar s,\kappa)$ is defined as in (\ref{def-par-sym-core})).
In addition, by (\ref{bijection-sym-skew-par}), $\Psi_\Gamma(\bar s,\kappa)$ is in bijection with
$\mathcal P(\beta_\Gamma e_\Gamma, w_\Gamma)$
if $\Gamma\ne x+1$ or $\Gamma= x+1$ and $\kappa_{x+1}$ is non-degenerate and in bijection with
$\mathcal P'(2 e, w_{x+1})$ if $\Gamma= x+1$ and $\kappa_{x+1}$ is degenerate.
Here, the sets $\mathcal P(\beta_\Gamma e_\Gamma, w_\Gamma)$ and $\mathcal P'(2 e, w_{x+1})$ are defined as in (\ref{def-set-lab}) and (\ref{def-p'}) respectively.

Fix $(s,\kappa)\in i\Bl_\ell(\tG)$.
Now we give an $e_0$-Jordan-cuspidal pair of $\tilde\bG^F$ corresponding to $\tilde B(s,\kappa)$.
First, we define an $e_0$-split Levi subgroup $\bL$ of $\bG$.
Let $(s,\mu)\in i\Irr(\tG)$ such that $\chi_{s,\mu}\in \tilde B(s,\kappa)$.
Recall that we have integers $w_\Gamma=e_\Gamma^{-1}(m_\Gamma(s)-|\kappa_\Gamma|)$ if $\Gamma\in\cF_1\cup\cF_2$ and
$w_\Gamma$ is the number of $e_\Gamma$-hooks (or $e_\Gamma$-cohooks) removed from $\mu_\Gamma$ to get $\kappa_\Gamma$ if $\Gamma\in\cF_0$~(see \S\ref{notations-and-conventions}).
Note that the $w_\Gamma$ do not depend on the choice of $\mu$ and are determined by $(s,\kappa)$.
Let $\tD$ be a defect group of the $\ell$-block $\tilde B(s,\kappa)$ and $D=\tD\cap G$.
Then by \cite[\S11]{FS89},
there exist corresponding orthogonal decompositions
$V=V_0\perp V_1\perp\cdots\perp V_t$ and $D=D_0\times D_1\times\cdots\times D_t$ such that $D_0=\langle 1_{V_0}\rangle$ and for $i>0$, $D_i=R_{m_i,\alpha_i,0,\beta_i}$ for integers $m_i$, $\alpha_i$, $\beta_i$.
Here, $R_{m_i,\alpha_i,0,\beta_i}=R_{m_i,\alpha_i,\gamma,\bc}$ with $\gamma=0$, $\bc=(\beta_i)$ defined as in \S\ref{section:radical}.
In addition,
we may write $D=D_0\prod_\Gamma (R_{m_\Gamma,\alpha_\Gamma,0,\beta})^{t_{\Gamma,\beta}}$, where $t_{\Gamma,\beta}$ are the coefficients occuring in the $\ell$-adic expansion $\sum_\beta t_{\Gamma,\beta}\ell^\beta$ of $w_\Gamma$ for $\Gamma\in\cF$.
Thus $V_0=C_V(D)$ and $V_+=[D,V]$, where $V_+=V_1\perp\cdots\perp V_t$.

Let $\bL$ be an $F$-stable Levi subgroup of $\bG$ (described as in \cite[(3A)]{FS89}) such that $L=\bL^F=L_0\times  L_+$  with
$L_0=\Sp(V_0)$, $L_+=\prod\limits_\Gamma\prod\limits_{i=1}^{w_\Gamma} L_{\Gamma,i}\le\Sp(V_+)$
and
$ L_{\Gamma,i}\cong \GL_{\frac{e_\Gamma \delta_\Gamma}{e}}(\eps q^e)$ for $1\le i\le w_\Gamma$,
where $\eps=1$ if $\ell$ is linear and $\eps=-1$ if $\ell$ is unitary.
Let $V=V_0+\sum\limits_\Gamma \sum\limits_{i=1}^{w_\Gamma} V_{\Gamma,i}$ be the corresponding orthogonal decomposition of $V$. Obviously, $V_+=\sum\limits_\Gamma \sum\limits_{i=1}^{w_\Gamma} V_{\Gamma,i}$.
From this we obtain an $e_0$-split Levi subgroup $\bL$ of $\bG$.
Then $\tilde \bL=\bL Z(\tilde\bG)$ is an $e_0$-split Levi subgroup of $\tilde\bG$ and the structure of $\bL^*$ and $\tilde\bL^*$ are described in \cite[(3A) and (3B)]{FS89}.
In fact, 
$\bL\hookrightarrow \tilde\bL$ is a regular embedding.
Clearly,  $\bar s\in L^*:={\bL^*}^F$ and $s\in \tilde L^*:=(\tilde\bL^*)^F$ up to conjugacy.
Also, $\tilde L=\langle L_0,\tau\rangle \cdot L_+$ is the central product of $\langle L_0,\tau\rangle$ and $L_+$, where $\tau$ is as in the proof of Lemma \ref{conofcharoflevi}.
Write $\bar s=\bar s_0\times \bar s_+$, with $\bar s_0\in L_0^*$ and $\bar s_+\in L_+^*$.

Let $V^*$ be the underlying space of  $G^*$ and
 $V^*=V^*_0+\sum\limits_\Gamma \sum\limits_{i=1}^{w_\Gamma} V^*_{\Gamma,i}$, $L^*=L^*_0\times L^*_+$ with $L^*_+=\prod\limits_\Gamma\prod\limits_{i=1}^{w_\Gamma} L^*_{\Gamma,i}$
 be the corresponding  orthogonal decompositions of $V^*$ and $L^*$ as in \cite[(3A)]{FS89}.
Denote $V^*_+=\sum\limits_\Gamma \sum\limits_{i=1}^{w_\Gamma} V^*_{\Gamma,i}$.
In addition, we have the primary decompositions
$V_0^*=\sum\limits_\Gamma V^*_\Gamma(\bar s_0)$ and $\bar s_0=\prod\limits_\Gamma \bar s_0(\Gamma)$
of $V_0^*$ and $\bar s_0$.
Thus $C^\circ_{\bL^*}(\bar s)^F=\SO(V^*_{x-1}(\bar s_0))\times \SO(V^*_{x+1}(\bar s_0))\times \prod\limits_{\Gamma\in\cF_1\cup\cF_2}\GL_{m_\Gamma(\bar s_0)}(\vare_\Gamma q^{\delta_\Gamma})\times C_{L^*_+}(s_+)$, where
$C_{L^*_+}(s_+)=\prod\limits_{\Gamma\in\cF}\prod\limits_{i=1}^{w_\Gamma} \GL_1((\eps q^e)^{e_\Gamma\delta_\Gamma/e})$.
Let $\phi_{\kappa}=\prod\limits_\Gamma \phi_\Gamma \times 1_{C_{L^*_+}(s_+)} $ be the unipotent character of $C^\circ_{\bL^*}(\bar s)^F$, where
$\phi_{\Gamma}$ is the unipotent character of $\SO(V^*_{\Gamma}(\bar s_0))$ corresponding to $\kappa_\Gamma$ if $\Gamma\in \cF_0$
and the unipotent character of $\GL_{m_\Gamma(\bar s_0)}(\vare_\Gamma q^{\delta_\Gamma})$ corresponding to $\kappa_\Gamma$ if $\Gamma\in \cF_1\cup\cF_2$.
Then $\phi_\kappa$ is an $e_0$-cuspidal unipotent character of $C^\circ_{\bL^*}(\bar s)^F$.
Now note that
$\mathcal E(C_{\tilde\bL^*}(s)^F,1)= \mathcal E(C^\circ_{\bL^*}(\bar s)^F,1)$ and then we may regard $\phi_\kappa$ as an $e_0$-cuspidal unipotent character of $C_{\tilde\bL^*}(s)^F$.
Let $\tilde\zeta$ be the character of $\tilde \bL^F$ corresponding under the Jordan decomposition to $\phi_\kappa\in\mathcal E(C_{\tilde\bL^*}(s)^F,1)$.
Then $(\tilde \bL,\tilde\zeta)$ is an $e_0$-Jordan-cuspidal pair of $\tilde\bG$.

\begin{lem}
With the notation above,  the $\ell$-block $b_{\tilde\bG^F}(\tilde \bL,\tilde\zeta)$ of $\tilde\bG^F$ corresponding to the $e_0$-Jordan-cuspidal pair  $(\tilde \bL,\tilde\zeta)$ is $\tilde B(s,\kappa)$.
\end{lem}

\begin{proof}
We prove that there is one irreducible constituent of $R_{\tilde\bL}^{\tilde\bG}(\tilde\zeta)$ lying in $\tilde B(s,\kappa)$.
In fact, this is essentially contained in \cite[\S 13]{FS89}.
Let $Q$ and  $\tilde Q$ be the $F$-fixed point of some $F$-stable Levi subgroups (say, $\mathbf Q$ and $\tilde{\mathbf Q}$) defined in \cite[p.~178]{FS89}, centralizer of a certain $\ell$-element in $Z(D)$ in $G$ and $\tilde G$ respectively.
Then $Q\cong Q_0\times Q_+$  and $\tilde Q=\langle Q_0,\tau\rangle Q_+$
with $Q_0=L_0$ and $L_+\le Q_+$.
Also, we let $\tilde{\mathbf b}$ be the $\ell$-block of $\tilde Q$ defined in \cite[p.~179]{FS89}.
Now $\tilde\bL\le\tilde {\mathbf Q}$, so
$R_{\tilde\bL}^{\tilde\bG}(\tilde\zeta)=R_{\tilde {\mathbf Q}}^{\tilde\bG}(R_{\tilde\bL}^{\tilde {\mathbf Q}}(\tilde\zeta))$.
In addition $C_{\tilde Q^*}(s)\le \tilde L^*$,
then
$R_{\tilde\bL}^{\tilde {\mathbf Q}}(\tilde\zeta)$ lies in $\cE(\tilde Q,s)$ and then by \cite[(13A)]{FS89}, lies in $\tilde{\mathbf b}$.
Thus we conclude from the proof of \cite[(13B)]{FS89} that there exists one irreducible constituent of $R_{\tilde\bL}^{\tilde\bG}(\tilde\zeta)$ lying in $\tilde B(s,\kappa)$.
\end{proof}

Now keep the hypotheses and setup above and we wish to investigate how many $\ell$-blocks of $G$ are covered by $\tilde B(s,\kappa)$.
This number is  equal to the cardinality of the set $\Irr(L\mid \tilde\zeta)$ by Remark \ref{esplitofslsu}.

Let ${\mathfrak e}$ be the identity element of the clifford algebra over $V^*$.
Then  $Z((\tilde\bG^*)^F)=\{ k{\mathfrak e}\mid k\in\F_q^\times \}$.
For $z\in Z((\tilde\bG^*)^F)$,
we denote by $\hat z$ the corresponding linear character (by \cite[Prop.~13.30]{DM91}) of ${\tG}$ as before.
Moreover, we may regard $\hat z$ as a linear character of $\tilde L/L$ since $\tilde L/L\cong \tG/G$.
From this, $(\tilde \bL,\hat z\tilde\zeta)$ is also an $e_0$-Jordan-cuspidal pair of $\tilde\bG$.
Also, $\hat z\tilde\zeta\in \mathcal E(\tilde \bL^F,\ell')$ if and only if $z$ is
of $\ell'$-order.
Conversely, if $\tilde B'$ is an $\ell$-block of $\tG$ such that $\tilde B'$ and $b_{\tilde\bG^F}(\tilde\bL,\tilde\zeta)$ cover the same $\ell$-blocks of $G$, then $\tilde B'=b_{\tilde\bG^F}(\tilde\bL,\hat z\tilde\zeta)$ for some $z\in\mathcal O_{\ell'}(Z(\tG^*))$.

The relations between conjugacy classes of $\tG^*$ and $G^*$ are given in \cite[\S6.4]{Lu77} (or \cite[(2D)]{FS89}).
Let $\bar C$ be the conjugacy class of $G^*$ containing $\bar s$ and $C=\pi^{-1}(\bar C)$.
If $-1$ is not an eigenvalue of $\bar s$,
then $C$ is the union of $|Z(\tG^*)|$ conjugacy classes of $\tG^*$ and each class contains a unique element of $\{zs\mid z\in Z(\tG^*) \}$.
If $-1$ is an eigenvalue of $\bar s$,
then $C$ is the union of $\frac{1}{2}|Z(\tG^*)|$ conjugacy classes of $\tG^*$ and each class contains exactly two elements $z_1$ and $z_2$ of
 $\{zs\mid z\in Z(\tG^*) \}$ such that $z_1=- z_2$.

If  $-1$ is not an eigenvalue of $\bar s$, then  $\tilde B(zs,\kappa)$ are distinct $\ell$-blocks of $\tG$ for $z\in\mathcal O_{\ell'}(Z(\tilde G^*))$ and they cover the same $\ell$-block of $G$.
In addition, $b_{\tilde\bG^F}(\tilde\bL,\hat z\tilde\zeta)=\tilde B(zs,\kappa)=\hat z \otimes \tilde B(s,\kappa)$.
Now suppose $-1$ is an eigenvalue of $\bar s$.
Then $\tilde B(zs,\kappa)$ and $\tilde B(zs,\kappa')$~(where $\kappa'$ is defined as in page \pageref{def-c-s}) are $\ell$-blocks of $\tG$,
where $z$ runs through a complete set of representatives of $\langle -{\mathfrak e} \rangle$-cosets in $\mathcal O_{\ell'}(Z(\tilde G^*))$, and they cover the same $\ell$-blocks of $G$.
If $w_{x+1}\ne 0$ or $\kappa_{x+1}$ is non-degenerate, then $\tilde B(zs,\kappa)=\tilde B(zs,\kappa')$ by
Remark \ref{res-to-so-deg}
and then
$b_{\tilde\bG^F}(\tilde\bL,\hat z\tilde\zeta)=\tilde B(zs,\kappa)=\hat z \otimes \tilde B(s,\kappa)$ for $z\in \mathcal O_{\ell'}(Z(\tilde G^*))$.
Let $w_{x+1}=0$ and $\kappa_{x+1}$ be degenerate.
Fix $z\in \mathcal O_{\ell'}(Z(\tilde G^*))$, then $\tilde B(zs,\kappa)$ and $\tilde B(zs,\kappa')$
are distinct $\ell$-blocks of $\tG$.
In addition, if $b_{\tilde\bG^F}(\tilde\bL,\hat z\tilde\zeta)=\tilde B(zs,\kappa)=\hat z \otimes \tilde B(s,\kappa)$ for $z\in \mathcal O_{\ell'}(Z(\tilde G^*))$, then
$b_{\tilde\bG^F}(\tilde\bL,\widehat{-{\mathfrak e} s}\tilde\zeta)=\tilde B(-{\mathfrak e} zs,\kappa)=\widehat{-{\mathfrak e} s} \otimes \tilde B(s,\kappa)$.

Now $\bL\hookrightarrow\tilde\bL$ is a regular embedding, so by Jordan decomposition, we have
$|\Irr(\bL^F\mid \tilde\zeta)|=|\Irr(C_{{\bL^*}}(\bar s)^F\mid \phi_\kappa)|$.
Thus if $-1$ is not an eigenvalue of $\bar s$,
then $|\Irr(\bL^F\mid \tilde\zeta)|=1$,
Now suppose that $-1$ is an eigenvalue of $\bar s$.
Then by Clifford theory,
$\Res^{\tilde{\bL}^F}_{\bL^F}\tilde\zeta$ is not irreducible if and only if $\phi_\kappa$ is $C_{{\bL^*}}(\bar s)^F$-invariant.
Note that
$C_{\bL^*}(\bar s)^F=\SO(V^*_{x-1}(\bar s_0)\perp V^*_{x+1}(\bar s_0))\cap (\GO(V^*_{x-1}(\bar s_0))\times \GO(V^*_{x+1}(\bar s_0)))
\times\prod\limits_{\Gamma\in\cF_1\cup\cF_2}\GL_{m_\Gamma(\bar s_0)}(\vare q^{\delta_\Gamma})\times C_{L^*_+}(s_+)$.
Here, $V^*_{x-1}$ has odd-dimension and $V^*_{x+1}$ has even-dimension.
In this way, by a similar proof as in Lemma \ref{res-sp}, $\Res^{\tilde{\bL}^F}_{\bL^F}\tilde\zeta$ is irreducible if and only if $\kappa_{x+1}$ is degenerate.

By the argument above, we have the following result.

\begin{thm}\label{blocks-sp}
Let $(s,\kappa)\in i\Bl_\ell(\tilde G)$, $\bar s=\pi(s)$,
$\tilde B=\tilde B(s,\kappa)$ and $B$ an $\ell$-block of $G$ covered by $\tilde B$.
\begin{enumerate}[(i)]
\item If $-1$ is not an eigenvalue of $\bar s$,
then  $B$ is the unique $\ell$-block of $G$ covered by $\tilde B$ and there are $|\mathcal O_{\ell'}(\F_q^\times)|$ $\ell$-blocks of $\tG$ covering $B$.
In addition, the $\ell$-blocks covering $B$ of $\tilde G$ are $\tilde B(zs,\kappa)$,
where $z$ runs through  $\mathcal O_{\ell'}(Z(\tilde G^*))$.
\item  If $-1$ is  an eigenvalue of $\bar s$, $w_{x+1}= 0$ and $\kappa_{x+1}$ is degenerate,
then  $B$ is the unique $\ell$-block of $G$ covered by $\tilde B$ and there are $|\mathcal O_{\ell'}(\F_q^\times)|$ $\ell$-blocks of $\tG$ covering $B$.
In addition, the $\ell$-blocks covering $B$ of $\tilde G$ are $\tilde B(zs,\kappa)$ and $\tilde B(zs,\kappa')$, where $z$ runs through a complete set of representatives of $\langle -{\mathfrak e} \rangle$-cosets in $\mathcal O_{\ell'}(Z(\tilde G^*))$.
\item If $w_{x+1}\ne 0$ or $\kappa_{x+1}$ is non-degenerate,
then there are $\frac{1}{2}|\mathcal O_{\ell'}(\F_q^\times)|$ $\ell$-blocks of $\tG$ covering $B$ and they are $\tilde B(zs,\kappa)$, where $z$ runs through a complete set of representatives of $\langle -{\mathfrak e} \rangle$-cosets in $\mathcal O_{\ell'}(Z(\tilde G^*))$.	
	Moreover,
\begin{itemize}
	\item if $w_{x+1}\ne 0$ and $\kappa_{x+1}$ is degenerate, then $B$ is the unique $\ell$-block of $G$ covered by  $\tilde B$ and
	\item if $\kappa_{x+1}$ is non-degenerate, then there are two $\ell$-blocks of $G$ covered by  $\tilde B$.
\end{itemize}
\end{enumerate}
\end{thm}

Now let $i\Bl^{(1)}_\ell(G)$ be the set of $G^*$-conjugacy classes of pairs $(s,\kappa)$, where $s\in G^*$ is a semisimple $\ell'$-element and $\kappa\in\cC(s)$ such that $-1$ is not an eigenvalue of $s$ or
 $\kappa_{x+1}$ is degenerate.
Here, we identify $(s,\kappa)$ with $(s,\kappa')$.
Let $i\Bl^{(2)}_\ell(G)$ be the set of $G^*$-conjugacy classes of pairs $(s,\kappa)$,
where $s\in G^*$ is a semisimple $\ell'$-element  and $\kappa\in\cC(s)$ is such that $w_{x+1}(s)\ne 0$ and $\kappa_{x+1}$ is non-degenerate.
Also $(s_1,\kappa_2)$ and $(s,\kappa_2)$ are $G^*$-conjugate means $s_1$ and $s_2$ are $G^*$-conjugate and $\kappa_1=\kappa_2$.
Then $i\Bl_\ell(G):=i\Bl^{(1)}_\ell(G)\cup i\Bl^{(2)}_\ell(G)$, where the elements of $i\Bl^{(2)}_\ell(G)$ are counted twice, is a labeling set for $\Bl_\ell(G)$ by Theorem \ref{blocks-sp}.
If $(s,\kappa)\in i\Bl^{(1)}_\ell(G)$, we denote by $B(s,\kappa)$ the $\ell$-blocks of $G$ corresponding to $(s,\kappa)$.
If $(s,\kappa)\in i\Bl^{(2)}_\ell(G)$,
then $B^{(1)}(s,\kappa)$ and $B^{(-1)}(s,\kappa)$ denote the two $\ell$-blocks of $G$ corresponding to $(s,\kappa)$.

\subsection{The action of $\Aut(G)$ on the Brauer characters and weights}

Now let $i\IBr^{(1)}_\ell(G)$ be the set of $G^*$-conjugacy classes of pairs $(s,\mu)$, where $s\in G^*$ is a semisimple $\ell'$-element and $\mu\in\Psi(s)$ such that either $-1$ is not an eigenvalue of $s$ or  $-1$ is an eigenvalue of $s$ and $\mu_{x+1}$ is degenerate.
Here, we identify $(s,\mu)$ with $(s,\mu)$ in $i\IBr^{(1)}_\ell(G)$, which means degenerate symbols are not counted twice in any case.
Let $i\IBr^{(2)}_\ell(G)$ be the set of $G^*$-conjugacy classes of pairs $(s,\mu)$,
where $s\in G^*$ is a semisimple $\ell'$-element  and $\mu\in\Psi(s)$ such that $-1$ is an eigenvalue of $s$ and $\mu_{x+1}$ is non-degenerate.
Then $i\IBr_\ell(G):=i\IBr^{(1)}_\ell(G)\cup i\IBr^{(2)}_\ell(G)$, where the elements of $i\IBr^{(2)}_\ell(G)$ are counted twice, is a labeling set for $\Irr(G)\cap \cE(G,\ell')$ by Lemma \ref{res-sp}.
If $(s,\mu)\in i\IBr^{(1)}_\ell(G)$, we denote by $\chi_{s,\mu}$ the character of $G$ corresponding to $(s,\mu)$.
If $(s,\mu)\in i\IBr^{(2)}_\ell(G)$,
then $\chi^{(1)}_{s,\mu}$ and $\chi^{(-1)}_{s,\mu}$ denote the two characters of $G$ corresponding to $(s,\mu)$.

Furthermore,
if $(s,\kappa)\in i\Bl^{(1)}_\ell(G)$ and $B=B(s,\kappa)$, then $\Irr(B)\cap\cE(G,s)=
\{\chi_{s,\mu}\mid \mu\in\Psi(s,\kappa),(s,\mu)\in i\IBr^{(1)}_\ell(G)   \} \cup \{\chi^{(\pm 1)}_{s,\mu}\mid \mu\in\Psi(s,\kappa),(s,\mu)\in i\IBr^{(2)}_\ell(G)  \}.$
If $(s,\kappa)\in i\Bl^{(2)}_\ell(G)$, $B^{(1)}=B^{(1)}(s,\kappa)$ and $B^{(-1)}=B^{(-1)}(s,\kappa)$, then
$\Irr(B^{(1)}\cup B^{(-1)})\cap\cE(G,s)=\{ \chi^{(\pm 1)}_{s,\mu}\mid \mu\in\Psi(s,\kappa) \}$.
We may assume that
$\Irr(B^{(i)})\cap\cE(G,s)=\{\chi^{(i)}_{s,\mu}\mid \mu\in\Psi(s,\kappa) \}$
for $i=\pm1$.
Note that, if $\kappa_{x+1}$ is non-degenerate, then $\mu_{x+1}$ is also non-degenerate,
and then we always have $(s,\mu)\in i\IBr^{(2)}_\ell(G)$.

\begin{rmk}\label{act-irr-sp}
Let $g\in \tilde G\setminus GZ(\tG)$.
Then $g$ induces the non-trivial diagonal automorphism on $G$.
By Lemma \ref{act-irr-scp} and \ref{res-sp}, we have
$\chi_{s,\mu}^\sigma=\chi_{{\sigma^*}^{-1}(s),{\sigma^*}^{-1}\mu}$ and
$\chi_{s,\mu}^g=\chi_{s,\mu}$
if $(s,\mu)\in i\IBr^{(1)}_\ell(G)$.
By Theorem \ref{action-irr-const}, we may assume that $(\chi^{(i)}_{s,\mu})^\sigma=\chi^{(i)}_{{\sigma^*}^{-1}(s),{\sigma^*}^{-1}\mu}$
and $(\chi^{(i)}_{s,\mu})^g=\chi^{(-i)}_{s,\mu}$
for $(s,\mu)\in i\IBr^{(2)}_\ell(G)$ and $i=\pm1$.	
\end{rmk}

For $(s,\kappa)\in i\Bl_\ell(G)$ and  an $\ell$-block $B$ of $G$ corresponding to $(s,\kappa)$,
we define $\mathcal P(B):=\prod\limits_\Gamma\mathcal P(\beta_\Gamma e_\Gamma,w_\Gamma)$, where the sets $\mathcal P(\beta_\Gamma e_\Gamma,w_\Gamma)$ are defined as in (\ref{def-set-lab}).

\begin{prop}\label{label-ibr-b}
With the preceding notation,
$\mathcal P(B)$ is a labeling set for $\Irr(B)\cap \cE(G,\ell')$.
\end{prop}

\begin{proof}	Let $\tilde B$ be an $\ell$-block of $\tG$ covering $B$.
If $(s,\kappa)\in i\Bl^{(2)}_\ell(G)$,
then every character of $\Irr(\tilde B)\cap \cE(\tG,\ell')$ is parametrized by an element of $i\IBr^{(2)}_\ell(G)$.
Thus the map $\Res^\tG_G:\Irr(\tilde B)\cap \cE(\tG,\ell')\to \Irr(B)\cap \cE(G,\ell')$ 	is bijective.
So $\Psi(s,\kappa)$ is a labeling set for $\Irr(B)\cap \cE(G,\ell')$.
From this
there is a canonical bijection between $\Psi(s,\kappa)$ and $\mathcal P(B)$ by  (\ref{bijection-sym-skew-par}).
So $\mathcal P(B)$ is a labeling set for $\Irr(B)\cap \cE(G,\ell')$.
	
Now we assume that $(s,\kappa)\in i\Bl^{(1)}_\ell(G)$.
If $-1$ is not an eigenvalue of $s$,
then every character of $\Irr(\tilde B)\cap \cE(\tG,\ell')$ is parametrized by an element of $i\IBr^{(1)}_\ell(G)$.
From this we obtain the map
$\Irr(\tilde B)\cap \cE(\tG,\ell')\to \Irr(B)\cap \cE(G,\ell')$,
which sends $\tilde\chi$ to the unique element of $\Irr(B)\cap \Irr(G\mid\tilde\chi)$.
So $\Psi(s,\kappa)$ is a labeling set for $\Irr(B)\cap \cE(G,\ell')$.
Thus $\mathcal P(B)$ is a labeling set for $\Irr(B)\cap \cE(G,\ell')$ as in the last paragraph.
	
Now we assume that $-1$ is an eigenvalue of $s$.
Let $\mu\in \Psi(s,\kappa)$.
If $(s,\mu)\in i\IBr_\ell^{(1)}(G)$,
then
$\mu_{x+1}$ is degenerate
 and if $(s,\mu)\in i\IBr_\ell^{(2)}(G)$,
 then
 $\mu_{x+1}$ is non-degenerate.
 By the proof of \cite[Prop.~15~(2)]{Ol86},
 if $\mu_{x+1}$ is degenerate, then
 it corresponds to an element of $\mathcal P_0(2e,w_{x+1})$~(defined as in
 (\ref{par-P0})) and if $\mu_{x+1}$ is non-degenerate, then $\mu_{x+1}$ and its copy
correspond to the two element of $\mathcal P(2e,w_{x+1})\setminus \mathcal P_0(2e,w_{x+1})$ which are equivalent in the sense of (\ref{equ-par}).
Thus we have a natural bijection between $\mathcal P(B)$ and $\Irr(B)\cap \cE(G,\ell')$. 	
\end{proof}

By the proof of Proposition \ref{label-ibr-b} and Remark \ref{act-irr-sp}, we have the following result immediately.

\begin{prop} \label{act-irr-ell'}
Let  $B$ be an $\ell$-block of $G$, $\mu\in \mathcal P(B)$ and $\chi$ be the irreducible character in $\Irr(B)\cap \cE(G,\ell')$ corresponding to $\mu$.
Let $g$ be an element of $\tG$ which induces a non-trivial diagonal automorphism and $\sigma$ be a field automorphism. Then
\begin{enumerate}[(i)]
	\item $\chi^\sigma$ is a character of $B^\sigma$ corresponding to ${\sigma^*}^{-1}\mu$,
	\item $\chi^g$ is a character of $B^g$ corresponding to $\mu^\dag$, where $\mu^\dag$ is defined as in Proposition \ref{action-par-diag-wei}.
\end{enumerate}
\end{prop}

\vspace{2ex}

Recall that $V$ is a $2n$-dimensional symplectic space over $\F_q$ with $n\ge 2$ and that $G=I(V)=I_0(V)=\Sp(V)$ throughout this section.
Let $B$ be   an $\ell$-block of $G$ covered by $\tilde B=\tilde B(s,\kappa)$
and define $i\cW_\ell(B)$ to be the set of
$K=\prod_\Gamma K_\Gamma$ where $K_\Gamma:\bigcup_\delta\sC_{\Gamma,\delta}\to \{\ell\textrm{-cores} \}$ such that
$\sum_{\delta,i,j}\ell^\delta |K_\Gamma(\psi_{\Gamma,\delta,i,j})|=w_\Gamma$.

A bijection between $\cW_\ell(B)$ and $i\cW_\ell(B)$ has been constructed implicitly in the proof of \cite[(4F)]{An94} and can be described as follows.
Let $D$ be a defect group of $B$,
$V_0=C_V(D)$ and $V_+=[V,D]$ so that $V=V_0\perp V_+$ as above.
Let $(D,\mathbf b)$ be a maximal Brauer pair of $G$ containing $(1,B)$, and $\vartheta$ be the canonical character of $\mathbf b$.
Then $D=D_0\times D_+$, $\mathbf b=\mathbf b_0\times \mathbf b_+$ and $\vartheta=\vartheta_0\times\vartheta_+$, where $D_0=\langle 1_{V_0} \rangle\le \Sp(V_0)$, $D_+\le\Sp(V_+)$, $\mathbf b_0$, $\mathbf b_+$ are $\ell$-blocks of $\Sp(V_0)$ and $C_{\Sp(V_+)}(D_+)$ respectively, and $\vartheta_0\in\Irr(\mathbf b_0)$, $\vartheta_+\in\Irr(\mathbf b_+)$.

Let $(R,\varphi)$ be a $B$-weight, $C=C_G(R)$ and $N=N_G(R)$.
Then there is an $\ell$-block $b$ of $CR$ with defect group $R$ and canonical character $\theta$ such that $b^G = B$ and $\varphi = \Ind^{N}_{N_{\theta}}\psi$ for some $\psi\in\Irr^0(N_{\theta}\mid \theta)$.
We may suppose $Z(D)\le Z(R)\le R\le D$.
Thus $V_0=C_V(R)$ and $V_+=[V,R]$,
so that $R=R_0\times R_+$, $C=C_0\times C_+$, $N=N_0\times N_+$, where $R_0=D_0$,
$R_+\le\Sp(V_+)$, $C_0=N_0=\Sp(V_0)$, $C_+=C_{\Sp(V_+)}(R_+)$ and $N_+=N_{\Sp(V_+)}(R_+)$.
Let $b=b_0\times b_+$ and $\theta=\theta_0\times\theta_+$ be the corresponding decompositions.
Then $\theta_0=\vartheta_0$ and $b_+^{\Sp(V_+)}=\mathbf b_+^{\Sp(V_+)}$.
Note that $N_\theta=N_0\times N_{\theta_+}$.
If $\psi\in\Irr^0(N_{\theta}\mid \theta)$, then $\psi=\vartheta_0\times\psi_+$, where
$\psi_+\in\Irr^0(N_{\theta_+}\mid \theta_+)$.
The map $(R,\Ind^{N}_{N_{\theta}}\psi)\mapsto(R_+,\Ind^{N_+}_{(N_+)_{\theta_+}}\psi_+)$
is a bijection from
$\{(R,\Ind^{N}_{N_{\theta}}\psi)\mid \psi\in   \Irr^0(N_{\theta}\mid \theta) \}$
to $\{ (R_+,\Ind^{N_+}_{(N_+)_{\theta_+}}\psi_+)\mid  \psi_+\in\Irr^0(N_{\theta_+}\mid \theta_+) \}$.
Then the bijection between $\cW_\ell(B)$ and $i\cW_\ell(B)$ can  be given as in \S \ref{act-wei-IV}.

Then there is a canonical bijection between
$i\cW_\ell(B)$ and
$\mathcal P(B):=\prod\limits_\Gamma\mathcal P(\beta_\Gamma e_\Gamma,w_\Gamma)$.
So $\mathcal P(B)$
is also a labeling set for $\cW_\ell(B)$ by \cite[(4F)]{An94}.

\begin{proof}[Proof of Theorem \ref{baw-sp}]
Let $B$ be an $\ell$-block of $G=\Sp_{2n}(q)$.
Then $\mathcal P(B)$ is a labeling set for both $\Irr(B)\cap\cE(G,\ell')$ and  $\cW_\ell(B)$ by the above argument.
So $|\Irr(B)\cap\cE(G,\ell')|=|\cW_\ell(B)|$.
Thus the assertion follows by Theorem \ref{basicset}.
\end{proof}

By  Corollary \ref{action-basic-par} and Proposition \ref{action-par-diag-wei}, we have the following result.

\begin{prop} \label{act-wei-c}
	Let $B$ be an $\ell$-block of $G$,
	$(R,\varphi)$ be a $B$-weight corresponding to $\mu\in \mathcal P(B)$
	Let $g$ be an element of $\tG$ which induces a non-trivial diagonal automorphism and $\sigma$ be a field automorphism of $G$.
	 Then
	\begin{enumerate}[(i)]
		\item $(R,\varphi)^\sigma$ is a $B^\sigma$-weight corresponding to ${\sigma^*}^{-1}\mu$,
		\item $(R,\varphi)^g$ is a  $B^g$-weight corresponding to $\mu^\dag$, which is defined as in Proposition \ref{action-par-diag-wei}.
	\end{enumerate}
\end{prop}

\vspace{2ex}

Note that $X=G=\mathrm{Sp}_{2n}(q)$ is the universal $\ell'$-covering group of the simple group $S=\mathrm{PSp}_{2n}(q)$ by \cite[\S~6.1]{GLS98}.
By \cite[\S~2.5]{GLS98}, the automorphisms induced by $\tilde G\rtimes E$ equal $\Aut(G)$.
Recall that $E=\langle F_p \rangle$.

By Proposition \ref{act-irr-ell'} and \ref{act-wei-c}, we have

\begin{thm}\label{bij-char-sp}
	Let $G=\Sp_{2n}(q)$,  $B$ an $\ell$-block of $G$
	where
	$q=p^f$ is a power of an odd prime $p$,
$n\ge 2$ and
$\ell\nmid q$ is an odd prime.
Then there is an $\Aut(G)_B$-equivariant bijection  between $\Irr(B)\cap \mathcal E(G,\ell')$ and $\cW_\ell(B)$.
\end{thm}

\begin{thm}\label{bij-sp}
Keep the hypothesis and setup of Theorem \ref{bij-char-sp}.
Assume further that $\ell$ is linear. Then there exists
an $\Aut(G)_B$-equivariant bijection  between $\IBr_\ell(B)$ and $\cW_\ell(B)$.
\end{thm}

\begin{proof}
	By Theorem \ref{basicset}, $\Irr(B)\cap\cE(G,\ell')$ is a basic set for $B$.
	Since $\ell$ is linear, the decomposition matrix is unitriangular by \cite{GH97}.
	Hence there is an $(\tG\rtimes E)_B$-equivariant bijection between $\Irr(B)\cap\cE(G,\ell')$ and $\IBr_\ell(B)$ by \cite[Lem.~7.5]{CS13}.
	Thus the assertion follows from Theorem  \ref{bij-char-sp}.
\end{proof}

Now we prove the main result of this paper for simple groups of type $C_n$.

\begin{proof}[Proof of Theorem \ref{equiv-bij-c}]
With the assumption that $f$ is odd,
we know $\Aut(S)/S$ is cyclic.
Then by Corollary \ref{cyclic-outer-ind}, it suffices to show that there exists an $\Aut(G)_B$-equivariant bijection between $\IBr_\ell(B)$ and $\cW_\ell(B)$ for every $\ell$-block $B$ of $G$.
Then the assertion follows by Theorem \ref{bij-sp}.	
\end{proof}

\subsection{The unipotent blocks}\label{uni-bloc-c}

We first summarize the description for  the unipotent $\ell$-blocks of symplectic groups above.
Let $G=\Sp_{2n}(q)$, with $n\ge 2$ and $q$ odd.
Then the unipotent $\ell$-blocks of $G$ are parametrised by $\cC(1)$.
Let $w=w_{x-1}$.
We also write $B(\kappa,w)$ for $B(1,\kappa)$ to emphasize $w$.
Then $\mathcal P(2e,w)$ is a labeling set for the unipotent characters of $B(\kappa,w)$.

Let $B=B(\kappa,w)$ be   a unipotent $\ell$-block of $G$.
Then by the above argument, there is a bijection between $\cW_\ell(B)$ and $\mathcal P(B)=\mathcal P(2e,w)$.
By Proposition \ref{act-wei-c}, we have

\begin{lem}\label{act-wei-typC}
Let $B$ be a unipotent $\ell$-block of $G=\Sp_{2n}(q)$ with $q$, $\ell$ as above.
Then every $B$-weight (in the sense of $G$-conjugacy class) is invariant under the action of $\tilde G\rtimes E$.
\end{lem}

\begin{rmk}\label{rad-bet-b}
	Now we give the relationship between $\ell$-weights of $\tilde G$ and $\ell$-weights of $G$ in the proof above.
	Since $|\tilde G/GZ(\tilde G)|=2$, by the same argument as in \cite[\S5.2 and \S5.3]{Fe18}, we have the following statements.
	\begin{itemize}
		\item $\tilde R\mapsto \tilde R\cap G$ gives a bijection from  $\Rad_\ell(\tilde G)$ to $\Rad_\ell(G)$ with inverse given by $R\mapsto R\times \mathcal O_\ell(Z(\tilde G))$.
		\item Let  $\tilde R=R \mathcal O_\ell(Z(\tilde G))$ for $R\in \Rad_\ell(G)$.
		If $(\tilde R,\tilde\varphi)$ is an $\ell$-weight of $\tilde G$, then $(R,\varphi)$ is an $\ell$-weight of $G$ for every $\varphi\in\Irr(N_G(R)\mid \tilde\varphi)$.
		Conversely, if $(R,\varphi)$ is an $\ell$-weight of $G$, then there exists $\tilde\varphi\in\Irr(N_{\tilde G}(\tilde R)\mid \varphi)$ such that
		$(\tilde R,\tilde\varphi)$ is an $\ell$-weight of $\tilde G$.
		\item Let $R\in \Rad_\ell(G)$, $\tilde R=R \mathcal O_\ell(Z(\tilde G))$, $(\tilde R,\tilde\varphi)$ an $\ell$-weight of $\tilde G$ and $\varphi\in\Irr(N_G(R)\mid \tilde\varphi)$.
		Suppose that $\tilde B$ is an $\ell$-block of $\tilde G$ and $B$ is an $\ell$-block of $G$.
		If $(\tilde R,\tilde\varphi)$ is an $\tilde B$-weight and $(R,\varphi)$ is an $B$-weight,
		then $\tilde B$ covers $B$.
	\end{itemize}
\end{rmk}

\begin{lem}\label{num-wei-csp}
Let $R\in \Rad_\ell(G)$, $\tilde R=R \mathcal O_\ell(Z(\tilde G))$ and $(R,\varphi)$ an $\ell$-weight of $G$.
Then  $|\{\tilde\varphi\in\Irr(N_{\tilde G}(\tilde R)\mid \varphi) \mid (\tilde R,\tilde \varphi)\ \text{is an}\ \ell  \text{-weight of}\ \tilde G \}|\le |\mathcal O_{\ell'}(\tilde G/G)|$.
\end{lem}

\begin{proof}
Let $\tilde\varphi_0\in\Irr(N_{\tilde G}(\tilde R)\mid \varphi)$ such that $(\tilde R,\tilde\varphi)$ is an $\ell$-weight of $\tilde G$.
Since $N_{\tilde G}(\tilde R)/N_{G}(R)\le \tilde G/G$ is cyclic, by Clifford theory we have that
$\Irr(N_{\tilde G}(\tilde R)\mid \varphi)=\{\tilde\varphi_0\tau\mid\tau\in\Irr(N_{\tilde G}(\tilde R)/N_{G}(R))\}$.
If $(\tilde R,\tilde\varphi_0\tau)$ is an $\ell$-weight of $\tilde G$, then $\tilde R$ is contained in the kernel of $\varphi_0\tau$, and then $\mathcal O_\ell(Z(\tilde G))$ is contained in the kernel of $\tau$.
Thus $\tau\in\Irr(N_{\tilde G}(\tilde R)/\mathcal O_\ell(Z(\tilde G))N_{G}(R))$.
Now $N_{\tilde G}(\tilde R)/\mathcal O_\ell(Z(\tilde G))N_{G}(R)$ is an $\ell'$-subgroup of $\tilde G/G$, thus the assertion holds.
\end{proof}

\begin{proof}[Proof of Theorem \ref{ibawc-uni} for type $C_n$]
Let $B=B(\kappa,w)$ be a unipotent $\ell$-block of $X=G=\Sp_{2n}(q)$.
Then by the argument above, $\mathcal P(2e,w)$ is a labeling set of $\cW_\ell(B)$.
By Lemma \ref{act-wei-typC}, every element of $\cW_\ell(B)$ is $\tilde G\rtimes E$-invariant.
On the other hand, $\Irr(B)\cap \mathcal E(G,1)$ is a basic set of $B$ by Theorem \ref{basicset}.
By \cite[Thm.~2.5]{Ma08}, every unipotent character of $G$ is $\tilde G\rtimes E$-invariant.
Thus $\mathcal P(2e,w)$ is a labeling set of $\IBr_\ell(B)$
and every element of $\IBr_\ell(B)$ is $\tilde G\rtimes E$-invariant.
Hence there exists an $\Aut(G)$-equivariant bijection between $\IBr_\ell(B)$ and $\cW_\ell(B)$.
Then it suffices to show condition (iii) of Definition \ref{induc}.
Note that the number of $\ell$-blocks of $\tilde G$ covering $B$ is $|\mathcal O_{\ell'}(\F_{q}^\times)|$ and there is a unique unipotent $\ell$-block $\tilde B$ of $\tG$ covering $B$ by Theorem \ref{blocks-sp}.
Then $\tilde B_z:=\hat z\otimes \tilde B$ for $z\in\mathcal O_{\ell'}(Z(\tG^*))$ are (all) the  $\ell$-blocks of $\tG$ covering $B$.

\emph{Claim \bf 1}:
Every $\phi\in \IBr_\ell(B)$ extends to $\tilde G\rtimes E$.

Let $\tilde \phi\in \IBr_\ell(\tilde B)$ be an  extension of $\phi$.
Note that every element of $\IBr_\ell(\tilde B)$ is $\tilde G\rtimes E$-invariant
since
$\Irr(\tilde B)\cap \mathcal E(\tilde G,1)$ is a basic set of $\tilde B$ by Theorem \ref{basicset} and every unipotent character of $\tilde G$ is $E$-invariant.
Thus $\tilde\phi$ extends to $\tilde G\rtimes E$ and so does $\phi$.

\emph{Claim \bf 2}: If $(R,\varphi)$ is a $B$-weight, then $\varphi$ extends to $(\tilde G\rtimes E)_{R}$.

By Lemma \ref{act-wei-typC}, $\varphi$ is $(\tilde G\rtimes E)_{R}$-invariant.
For $z\in \mathcal O_{\ell'}(Z(\tilde G^*))$, there exists $\tilde\varphi_z\in\Irr(N_{\tilde G}(\tilde R)\mid \varphi)$ such that $(\tilde R,\tilde\varphi_z)$ is a $\tilde B_z$-weight by Corollary \ref{lin-wei}.
Now the number of $\ell$-blocks of $\tilde G$ which cover $B$ is $|\mathcal O_{\ell'}(\F_{q}^\times)|$, by Lemma \ref{num-wei-csp}, $\tilde\varphi_z$ is unique.
Let $\tilde\varphi=\tilde\varphi_1$, that is, $\tilde\varphi\in\Irr(N_{\tilde G}(\tilde R)\mid \varphi)$ and $(\tilde R,\tilde\varphi)$ is a $\tilde B$-weight.
Now $\tilde B$ is $D$-invariant, so we have
$\tilde\varphi^x\in\Irr(N_{\tilde G}(\tilde R)\mid \varphi)$ and $(\tilde R,\tilde\varphi^x)$ is a $\tilde B$-weight for all $x\in (\tilde G\rtimes E)_{R}$.
Thus $\tilde\varphi^x=\tilde \varphi$ and then $\tilde\varphi$ is $(\tilde G\rtimes E)_{R}$-invariant.
Note that $E$ acts trivially on  $\Rad_\ell(\tilde G)/\sim_G$.
Thus there exists $g\in G$ such that $g^{-1}\sigma$ stabilizes $R$, and then $(\tilde G\rtimes E)_{R}=N_{\tilde G}(R)\langle g^{-1}\sigma \rangle$, which implies $(\tilde G\rtimes E)_{R}/N_{\tilde G}(\tilde R)$ is cyclic.
Hence $\tilde\varphi$ extends to $(\tilde G\rtimes E)_{R}$ and so does $\varphi$.

The remaining process is similar to the case of type $B$.
For a set $\IBr_\ell(B\mid Q)$ as in Lemma \ref{first-2-con}~(for definition, see the proof of \cite[Lemma 2.10]{Sc16}) and $\phi\in \IBr_\ell(B\mid Q)$, we let $A:=A(\phi,Q)=(\tilde G\rtimes E)/Z(\tG)$.
By a similar argument as in the proof of \cite[Prop.~7.1]{Fe18}, 
conditions (1)-(3) of Definition \ref{induc}~(iii) hold.
For Definition \ref{induc}~(iii)(4), the proof of \cite[Lem.~7.2]{Fe18} applies here.
\end{proof}

\section{Type $D$}\label{typeD}

Let $V$ be a $2n$-dimensional orthogonal space over $\F_q$ with $n\ge 4$ and let $\tilde G=I(V)=\GO(V)=\GO_{2n}^\eps(q)$ and $G=I_0(V)=\SO(V)=\SO_{2n}^\eps(q)$ with $\eps\in\{\pm\}$.
As before, $F_p$ denotes the field automorphism which sends $(a_{ij})$ to $(a_{ij}^p)$ and we write $E=\langle F_p\rangle$.

In this section, we write $J=\CSO_{2n}^\eps(q)$ for the special conformal orthogonal groups and $\tG=\GO_{2n}^\eps(q)$ for the general orthogonal groups for convention (of description of $\ell$-weight of special orthogonal groups in \cite{An94}), which is not the same as the notation in Appendix \ref{Appendix-blocks-so}.

The blocks of $\SO_{2n}^\eps(q)$ can be obtained from $\CSO_{2n}^\eps(q)$ as  we did for $\Sp_{2n}(q)$ in \S\ref{The-blocks-of-special-symplectic-groups};
see Appendix \ref{Appendix-blocks-so}.
But in this section, we only consider unipotent $\ell$-blocks, which are classified by Cabanes and Enguehard \cite{CE94}, which is easier to describe.
So we do not use the results of Appendix \ref{Appendix-blocks-so} in this section.
Let $\ell$ be an odd prime with $\ell\nmid q$ and $e_0$ and $e$ defined as before.
As in \S\ref{uni-bloc-c}, by \cite[Thm.]{CE94}, the unipotent $\ell$-blocks of $G=\SO^\eps_{2n}(q)$ are parametrised by the $G$-conjugacy classes of $e_0$-cuspidal pairs $(\mathbf L,\lambda)$.
Here $\mathbf L$ satisfies that $L=\mathbf L^F=\SO_{2(n-we)}^\delta(q)\times T_{e_0}^w$, with either $T_{e_0}=\GL_1(q^e)$ if $e_0$ is odd, or
$T_{e_0}=\GU_1(q^e)$ if $e_0$ is even,
and $\delta=\eps$ if $e_0$ is odd or $w$ is even, and $\delta=-\eps$ else,
and $\lambda$ is an $e_0$-cuspidal unipotent character of $L$~(for the structure of $e_0$-cuspidal pairs, see \cite{BMM93}).
Following the notation of \cite[\S5.3]{Ma17}, we write $B=B(L,\lambda)$ for the corresponding~(unipotent)~$\ell$-block.

Write $\lambda=\lambda_0\times 1_{T_{e_0}^w}$, where $\lambda_0$ is an $e_0$-cuspidal unipotent character of $\SO_{2(n-we)}^\delta(q)$.
Let $\kappa$ be the symbol corresponding to $\lambda_0$.
Then by \cite[\S3]{BMM93}, $\kappa$ is an $e$-core.
Moreover, the unipotent characters in the block $B(L,\lambda)$ are then the members of the $e_0$-Harish-Chandra series above $(L,\lambda)$, and then the ones parametrised by the symbols of rank $n$ and having $e$-core $\kappa$.
Thus we also write $B(\kappa,w)$ for $B(L,\lambda)$.

If $B=B(\kappa,w)$, then we let $\mathcal P(B):=\mathcal P(2e,w)$ if $\kappa$ is non-degenerate, and
 $\mathcal P(B):=\mathcal P'(2e,w)$ if $\kappa$ is degenerate.
Thus by (\ref{bijection-sym-skew-par}),
$\mathcal P(B)$ is a labeling set for $\Irr(B)\cap \cE(G,1)$ .
Also, $B$ is a defect zero $\ell$-block if and only if $w=0$.
So we always assume that $w>0$ from now on.

A defect group of $B(\kappa,w)$ is then obtained as a Sylow $\ell$-subgroup of $C_G([\mathbf L,\mathbf L])$, which is isomorphic to
a Sylow $\ell$-subgroup of $\GL_{ew}(q)$~(if $e_0$ is odd) or $\GU_{ew}(q)$~(if $e_0$ is even) by \cite[\S5.6]{Ma18}.

Denote $\tilde J:=\CO_{2n}^\eps(q)$ and $S=\mathrm{P\Omega}_{2n}^\eps(q)$.
Then $S$ is simple.
By \cite[\S2.5]{GLS98}, the automorphisms induced by $\tilde A:=\tilde J\rtimes E$ on $S$ equal  $\Aut(S)$ except when $n=4$ and $\eps=+$.
Recall that $E=\langle F_p \rangle$ is the group generated by the field automorphism $F_p$ which sends $(a_{ij})$ to $(a^p_{ij})$.
Here we start with a lemma following from \cite[Thm.~2.5]{Ma08}  immediately.

\begin{lem}\label{act-uni-D}
\begin{enumerate}[(i)]
	\item Every unipotent character of $\tilde G$ is $\tilde A$-invariant.
	\item Any element $g\in \tilde A$ fixes every unipotent character of $G$ except when $\eps=+$,
	the action of $g$ on $G$ can be induced by some element of $\tilde G\setminus G$ and the unipotent character is labelled by a degenerate symbol~(or an element of $\mathcal P_0(2e,w)$). Furthermore, such $g$ interchanges the two unipotent characters in all pairs labeled by the same degenerate symbol.
\end{enumerate}
\end{lem}

\begin{lem}\label{ext-ibr-d}
	Let $\phi\in\IBr_\ell(B)$ where $B$ is a unipotent $\ell$-block of $G$.
\begin{enumerate}[(i)]
	\item If $\eps=-$, then  $\phi$  extends to $\tilde A$.
	\item If $\eps=+$ and $\ell$ is linear, then $\phi$ extends to $\tilde A_\phi$.	
\end{enumerate}
\end{lem}

\begin{proof}
Let $B=B(\kappa,w)$.
If $w=0$, then $B$ is a defect zero $\ell$-block, and then the assertion follows from \cite[Thm.~2.4]{Ma08}.
Now we assume that $w>0$.
For every $z\in \mathcal O_{\ell'}(Z(J^*))$~(note that $J^*$ is the special Clifford group),
we let $\hat z$ be the corresponding linear character of $J$~(cf. \cite[Prop.~13.30]{DM91}) for $z\in \mathcal O_{\ell'}(Z(J^*))$.
Then by \cite[Prop.~13.30]{DM91}, $\hat z\otimes \cE(J,1)=\cE(J,z)$.
Thus there exists an $\ell$-block $\mathbf B_z$ of $J$ covering $B$ such that $\Irr(\mathbf B_z) \subseteq \mathcal E_\ell(J,z)$~(see also Theorem \ref{blocks-so}).
Moreover, $\mathbf B_z$ is unique.
In this way, there are $|\mathcal O_{\ell'}(\F_{q}^\times)|$ $\ell$-blocks of $J$ covering $B$.
We denote $\mathbf B=\mathbf B_1$.
Then $\mathbf B_z=\hat z \otimes \mathbf B$.
Since $z$  is of $\ell'$-order, we may regard $\hat z$ as a linear $\ell$-Brauer character of $J$.
Thus $\IBr_\ell(\mathbf B_z)=\{\hat z\phi\mid \phi\in\IBr_\ell(\mathbf B)\}$ by \cite[Lem.~2.4]{Fe18}.

If $\kappa$ is non-degenerate, then there are two unipotent $\ell$-blocks $\tilde B^{(1)}$, $\tilde B^{(2)}$ of $\tilde G$ covering $B$.
Let $\mathbf {\tilde B}^{(i)}$ be the unique  unipotent $\ell$-block of $\tilde J$ covering $\tilde B^{(i)}$ for $i=1,2$.
Since $J/G\cong \tilde  J/\tilde G$,
we may regard $\hat z$ as a linear character~(or linear $\ell$-Brauer character) of $\tilde J$  for $z\in \mathcal O_{\ell'}(Z(J^*))$.
Note that $\tilde G$ acts trivially on $J/G$, and then $\hat z$ is $\tilde G$-invariant.
Thus $\mathbf B_z$ is $\tilde J$-invariant for every $z\in \mathcal O_{\ell'}(Z(J^*))$.
Let $\mathbf {\tilde B}^{(i)}_z=\hat z \otimes\mathbf {\tilde B}^{(i)}$ for $z\in \mathcal O_{\ell'}(Z(J^*))$.
Then 	$\mathbf {\tilde B}^{(i)}_z$ covers $\mathbf B_z$.
From this, $\mathbf {\tilde B}^{(i)}_z$ for $i=1,2$ and $z\in \mathcal O_{\ell'}(Z(J^*))$ are distinct $\ell$-blocks of $\tilde J$.
In particular, there are $|\mathcal O_{\ell'}(\F_{q}^\times)|$ $\ell$-blocks of $\tilde J$ covering $B^{(i)}$ for $i=1,2$. 	
Now every character of $\Irr(B)\cap \mathcal E(G,1)$ is $\tilde A$-invariant by Lemma \ref{act-uni-D}.
By Theorem \ref{basicset}, $\Irr(B)\cap \mathcal E(G,1)$ is a basic set of $B$.
So every irreducible $\ell$-Brauer character of $B$ is $\tilde A$-invariant.
Then there is an extension $\tilde\phi\in\IBr_\ell(\tilde B^{(1)})$ of $\phi$ to $\tilde G$.
By Lemma \ref{act-uni-D} again, $\tilde B^{(1)}$ is $\tilde A$-invariant.
So $\tilde\phi$ is $\tilde A$-invariant by Lemma \ref{ext-Br}.
Then there exists an extension $\tilde \phi'\in\IBr_\ell(\mathbf {\tilde B}^{(1)})$ of $\tilde \phi$.
Note that the number of extensions of $\tilde\phi$ to $\tilde J$ is at most $|\mathcal O_{\ell'}(\F_{q}^\times)|$.
By Lemma \ref{ext-Br} again,  $\tilde \phi'$ is $\tilde A$-invariant and then extends to $\tilde A$ since $\tilde A/\tilde J$ is cyclic.

If $\eps=-$, then every $\kappa$ is non-degenerate, and then this assertion holds.
	
Now we let $\eps=+$. Then $\ell$ is linear.
By \cite{GH97}, with a suitable order, the decomposition matrix of $B$ with respect to the basic set $\Irr(B)\cap \mathcal E(G,1)$ is unitriangular.
By \cite[Thm.~2.4]{Ma08}, every $\chi\in \Irr(B)\cap \mathcal E(G,1)$ extends to $\tilde A_\chi$.
Then by Lemma \ref{ex-unitriibr}, every $\phi\in\IBr_\ell(B)$ extends to $\tilde A_\phi$.	
This completes the proof.
\end{proof}

Now let $B=B(\kappa,w)$ be a unipotent $\ell$-block of $G$ and $D$ a defect group of $B$. We state the results for $B$-weights which follows from the proof of \cite[(4H)]{An94}.
Let $V_0=C_V(D)$ and $V_+=[V,D]$ so that $V=V_0\perp V_+$.
Then $\mathrm{dim}(V_+)=2ew$.
Let $(D,\mathbf b)$ be a maximal Brauer pair of $G$ containing $(1,B)$, and $\vartheta$ be the canonical character of $\mathbf b$.
Let $\tilde G_0=\GO(V_0)$, $G_0=\SO(V_0)$, $\tilde G_+=\GO(V_+)$ and $G_+=\SO(V_+)$.
Then $D=D_0\times D_+$, $\mathbf b=\mathbf b_0\times \mathbf b_+$ and $\vartheta=\vartheta_0\times\vartheta_+$, where $D_0=\langle 1_{V_0} \rangle\le G_0$, $D_+\le G_+$, $\mathbf b_0$, $\mathbf b_+$ are $\ell$-blocks of $G_0$ and $C_{G_+}(D_+)$ respectively, and $\vartheta_0\in\Irr(\mathbf b_0)$, $\vartheta_+\in\Irr(\mathbf b_+)$.

Let $(R,\varphi)$ be a $B$-weight, $\tilde C=C_{\tilde G}(R)$, $C=C_G(R)$, $\tilde N=N_{\tilde G}(R)$ and $N=N_G(R)$.
Then there is an $\ell$-block $b$ of $CR$ with defect group $R$ and canonical character $\theta$ such that $b^G = B$ and $\varphi = \Ind^{N}_{N_{\theta}}\psi$ for some $\psi\in\Irr^0(N_{\theta}\mid \theta)$.
We may suppose $Z(D)\le Z(R)\le R\le D$.
Thus $V_0=C_V(R)$ and $V_+=[V,R]$,
so that $R=R_0\times R_+$, $C=G_0\times C_+$, $\tilde C=\tilde G_0\times C_+$, $N=\langle\tau, G_0\times N_+\rangle$ and  $\tilde N=\tilde G_0\times \tilde N_+$, where $R_0=D_0$,
$R_+\le G_+$, $C_+=C_{G_+}(R_+)$, $N_+=N_{G_+}(R_+)$,
$\tilde C_+=C_{\tilde G_+}(R_+)$, $\tilde N_+=N_{\tilde G_+}(R_+)$ and $\tau=\tau_0\times\tau_+$ with $\tau_0\in \tilde G_0$, $\tau_+\in\tilde G_+$ of determinant $-1$.
Then $\tilde N=\langle \tau_0,N \rangle$.
Let $b=b_0\times b_+$ and $\theta=\theta_0\times\theta_+$ be the corresponding decompositions.
Then $\theta_0=\vartheta_0$ and $b_+^{\tilde G_+}=\mathbf b_+^{\tilde G_+}$ and we suppose that $(R,b)\le(D,\mathbf b)$.

\emph{Case \bf1}.
$\vartheta_0^{\sigma_0}=\vartheta_0$ for some $\sigma_0\in\tilde G_0$ of determinant $-1$, \emph{i.e.}, $\kappa$ is non-degenerate.
There are two irreducible characters $\vartheta'_0$ and $\vartheta''_0$ of $\tilde G_0$ covering $\vartheta_0$.
Let $\vartheta'=\vartheta'_0\times \vartheta_+$,
$\vartheta''=\vartheta''_0\times \vartheta_+$,
and $\mathbf b'$, $\mathbf b''$ be the $\ell$-blocks of $C_{\tilde G}(D)$ containing $\vartheta'$,
$\vartheta''$ respectively.
Then $\mathbf b'^{\tilde G}$ and $\mathbf b''^{\tilde G}$
are two $\ell$-blocks of $\tilde G$.
Let $\tilde B=\mathbf b'^{\tilde G}$.

\emph{Case \bf2}.
$V_0=0$ or $\vartheta_0^{\sigma_0}\ne\vartheta_0$ for any $\sigma_0\in\tilde G_0$ of determinant $-1$, \emph{i.e.}, $\kappa$ is degenerate.
If $V_0=0$, then it is the case in Lemma \ref{(4C)}.
Now assume that $V_0\ne 0$ and $\vartheta_0^{\sigma_0}\ne\vartheta_0$ for any $\sigma_0\in\tilde G_0$ of determinant $-1$.
Also, we may assume that $\theta=\vartheta_0\times \theta_+$ for some character $\theta_+$ of $C_+$.
Then $\tilde N_{\theta}=G_0\times \tilde N_{\theta_+}$ and
$ N_{\theta}=G_0\times N_{\theta_+}$.
Thus each character $\tilde\psi\in\Irr(\tilde N_\theta\mid\theta)$ and each $\psi\in\Irr(N_\theta\mid\theta)$
decompose as $\tilde\psi=\vartheta_0\times \tilde\psi_+$ and
$\psi=\vartheta_0\times\psi_+$ for some
$\tilde\psi\in\Irr(\tilde N_{\theta_+}\mid\theta_+)$
and $\psi\in\Irr(N_{\theta_+}\mid\theta_+)$.
Then $\Res^{\tilde N_\theta}_{N_\theta} \tilde\psi$ is irreducible if and only if $\Res^{\tilde N_{\theta_+}}_{N_{\theta_+}} \tilde\psi_+$ is irreducible~(which is the case in Lemma \ref{(4C)}).
Let 
$\vartheta_0'=\vartheta_0+\vartheta_0^{\tau_0}$ and
${\mathbf b}'$ be the $\ell$-block of $C_{\tilde G}(D)D$ containing $\vartheta'=\vartheta_0'\times \vartheta_+$.
Then $\tilde B={\mathbf b'}^{\tilde G}$ is the unique $\ell$-block of $\tilde G$ covering $B$.

For both cases, we define $i\cW_\ell(\tilde B)$
to be the set of
$K:\bigcup_\delta\sC_{x-1,\delta}\to\{\ell\textrm{-cores}\}$ such that
$\sum_{\delta,i,j}\ell^\delta |K(\psi_{x-1,\delta,i,j})|=w.$

Note that $N_{\vartheta'}=\tilde G_0\times N_{\vartheta_+}$.
If $\psi\in\Irr^0(N_{{\vartheta'}}\mid {\vartheta'})$, then $\psi=\vartheta'_0\times\psi_+$, where
$\psi_+\in\Irr^0(N_{\vartheta_+}\mid \vartheta_+)$.
The map $(R,\Ind^{N}_{N_{\vartheta'}}\psi)\mapsto(R_+,\Ind^{N_+}_{(N_+)_{\vartheta_+}}\psi_+)$
is a bijection from
$\{(R,\Ind^{N}_{N_{\vartheta'}}\psi)\mid \psi\in   \Irr^0(N_{\vartheta'}\mid{\vartheta'}) \}$
to $\{ (R_+,\Ind^{N_+}_{(N_+)_{\vartheta_+}}\psi_+)\mid  \psi_+\in\Irr^0(N_{\vartheta_+}\mid \vartheta_+) \}$.
Then the bijection between $\cW_\ell(\tilde B)$ and $i\cW_\ell(\tilde B)$ can  be given as in \S \ref{act-wei-IV}.
As in the proof of \cite[(4E)]{An94}, there is a canonical bijection between $i\cW_\ell(\tilde B)$ and $\mathcal P(2e,w)$, so $\mathcal P(2e,w)$ is also a labeling set of $\cW_\ell(\tilde B)$~(see also Appendix \ref{(4E)}).

Similar with Lemma \ref{act-wei-typC}~(using Corollary \ref{action-basic-par} and Proposition \ref{action-par-diag-wei}), we have

\begin{lem}\label{act-wei-typD}
Every $\tilde B$-weight (in the sense of $\tilde G$-conjugacy class) is invariant under the action of $\tilde A=\CO_{2n}^\eps(q)\rtimes E$.
\end{lem}

Now we give $\ell$-weights of $G$ by the argument above and Lemma \ref{(4C)}.
If $\kappa$ is non-degenerate, then
$(R,\tilde \varphi)\mapsto (R,\Res^{N_{\tilde G}(R)}_{N_G(R)}\tilde \varphi)$
is a bijection from $\mathcal W_\ell(\tilde B)$ to $\mathcal W_\ell(B)$.
If $\kappa$ is degenerate and $w$ is odd,
then  $(R,\tilde \varphi)\mapsto (R,\Res^{N_{\tilde G}(R)}_{N_G(R)}\tilde \varphi)$
is also a bijection from $\mathcal W_\ell(\tilde B)$ to $\mathcal W_\ell(B)$.
Now let $\kappa$ be degenerate and let $w$ be even.
If $(R,\varphi)\in\mathcal W(\tilde B)$ corresponds to some element in $\mathcal P(2e,w)\setminus\mathcal P_0(2e,w)$, then $\Res^{N_{\tilde G}(R)}_{N_G(R)}\tilde\varphi$ is irreducible.
If $(R,\varphi)\in\mathcal W(\tilde B)$ corresponds to some element in $\mathcal P_0(2e,w)$, then $\Res^{N_{\tilde G}(R)}_{N_G(R)}\tilde\varphi$ is a sum of two irreducible constituents~(for the construction of these two irreducible constituents, see the proof of Lemma \ref{(4C)}).
In addition, by the argument above,
$\mathcal P(B)$ is a labeling set for $\cW_\ell(B)$ .

\begin{cor}\label{act-wei-SO}
Suppose that $g\in \tilde A$ and let $(R,\varphi)$ be a $B$-weight.
Let $\tilde\varphi\in\Irr(N_{\tilde G}(R)\mid\varphi)$ such that
$(R,\tilde\varphi)$ is a $\tilde B$-weight.
Then the $G$-conjugacy class of $(R,\varphi) $ is invariant under the action of $g$
except when $\eps=+$, the action of $g$ on $G$ can be induced by some element of $\tilde G\setminus G$ and $(R,\tilde\varphi)$ corresponds to an element of $\mathcal P_0(2e,w)$.
Furthermore, when $(R,\tilde\varphi)$ corresponds to some element of $\mathcal P_0(2e,w)$,
$g$ interchanges the two $G$-conjugacy classes of $(R,\varphi_1)$ and  $(R,\varphi_2)$,
where $\varphi_1$ and $\varphi_2$ are the irreducible constituents of $\Res^{N_{\tilde G}(R)}_{N_G(R)}\tilde\varphi$.
\end{cor}

\begin{proof}
By Lemma \ref{act-wei-typD}, it suffice to show that
the $G$-conjugacy class of $(R,\varphi) $ is invariant under the action of $\CSO^\eps_{2n}(q)\rtimes E$
if $\kappa$ is degenerate, $w$ is even and $(R,\tilde\varphi)$ corresponds to some element of $\mathcal P_0(2e,w)$.
It is similar to the proof of Lemma \ref{act-wei-typD}.
In fact, by the remark after \cite[(4A)]{An94}, for every $\delta$, the restriction of every weight character in $\mathscr C_{x-1,\delta}$ to the subgroup of $N_{x-1,\delta,i}$ with determinant $1$ is irreducible and then is invariant under the action of $\CSO^\eps_{2n}(q)\rtimes E$ by Lemma \ref{act-diag-wei}~(i).
So we conclude from Lemma \ref{(4C)} a similar result to Proposition \ref{action-basic} and \ref{action-par-diag-wei}~(i), hence the proof of Lemma \ref{act-wei-typC} also applies here, and finally that this assertion holds.
\end{proof}

\begin{lem}\label{ext-wei-d}
	Let $(R,\varphi)$ be a $B$-weight, where $B$ is a unipotent $\ell$-block of $G$.
Then $\varphi$ extends to $\tilde A_{R,\varphi}$.
\end{lem}

\begin{proof}
The proof here is similar to the proof of \emph{Claim \bf 2} in the proof of Theorem \ref{ibawc-uni} for type $C$, using the argument in the proof of Lemma \ref{ext-ibr-d}.
\end{proof}

\begin{lem}\label{wei-ibr-ome2n}
Let $T=[G,G]=\Omega^\eps_{2n}(q)$, $B$ a unipotent $\ell$-block of $G$ and $b$ an $\ell$-block of $T$ covered by $B$.
Then
\begin{enumerate}[(i)]
	\item $\Res^G_T:\IBr_\ell(B)\to \IBr_\ell(b)$ is a bijection, and
	\item $(R,\varphi)\mapsto (R,\Res^{N_G(R)}_{N_T(R)}\varphi)$ is a bijection from $\cW_\ell(B)$ to $\cW_\ell(b)$.
\end{enumerate}
\end{lem}
\begin{proof}
(ii) is similar with Proposition \ref{dec-weig}.
For (i), note that there exists another $\ell$-block of $G$ covering $b$.
Indeed, let $1\ne z\in Z(G^*)=Z(G)$ (we have $G^*=G$ here) and
$\hat z$ be the corresponding linear character (cf. \cite[Prop.~13.30]{DM91}),
then $\hat z \otimes B$ is an $\ell$-block of $G$ covering $b$ and is contained in $\mathcal E_\ell(G,z)$. So $\hat z\otimes B\ne B$. Thus (i) holds by Lemma \ref{ext-Br}.
\end{proof}

Note that $X=\mathrm{Spin}^\eps_{2n}(q)$ is the universal $\ell'$-covering group of the simple group $S=\mathrm{P\Omega}^\eps_{2n}(q)$ by \cite[\S~6.1]{GLS98}.

\begin{proof}[Proof of Theorem \ref{ibawc-uni}~(ii) and (iii)]
Let $b$ be a unipotent $\ell$-block of $X=\mathrm{Spin}^\eps_{2n}(q)$.
Then $b$ dominates an $\ell$-block $\mathfrak b$ of $T=\Omega^\eps_{2n}(q)$.
Thus there exists a unique unipotent $\ell$-block $B$ of $G=\SO^\eps_{2n}(q)$ covering $\mathfrak b$.
By Lemma \ref{act-uni-D} and Proposition \ref{act-wei-SO},
there is an $\tilde A$-equivariant bijection between $\mathcal W_\ell(B)$ and $\Irr(B)\cap\mathcal E_\ell(G,1)$.
If $\eps=-$, then every character of $\Irr(B)\cap\mathcal E_\ell(G,1)$ is $\tilde A$-invariant.
If $\eps=+$, then by assumption, $\ell$ is linear and
 thus by \cite[Lem~7.5]{CS13}, there is an $\tilde A$-equivariant bijection between $\Irr(B)\cap\mathcal E_\ell(G,1)$ and $\IBr_\ell(B)$~(since the corresponding decomposition matrix is unitriangular by \cite{GH97}).
In both cases, there exist $\tilde A$-equivariant bijections between $\mathcal W_\ell(B)$ and $\IBr_\ell(B)$.
From this, by Lemma \ref{wei-ibr-ome2n}, we get an $\tilde A$-equivariant bijection between $\mathcal W_\ell(\mathfrak b)$ and $\IBr_\ell(\mathfrak b)$.
Then by Lemma \ref{ind-quot}, there exists an
$\Aut(X)$-equivariant bijection between $\mathcal W_\ell(b)$ and $\IBr_\ell(b)$.
By Lemma \ref{first-2-con}, we only need to verify Definition \ref{induc}~(iii) now.

The remaining process is similar with the case of type $B$ since all irreducible character of $b$ have $Z(X)$ in their kernel.
Let $\bar b$ be the $\ell$-block of $S=\mathrm P \Omega_{2n}^{\eps}(q)$ dominated by $b$.
For a set $\IBr_\ell(b\mid Q)$ as in Lemma \ref{first-2-con}~(for definition, see the proof of \cite[Lemma 2.10]{Sc16}) and $\phi\in \IBr_\ell(b\mid Q)$, 
we denote by $\phi'\in\IBr_\ell(\mathfrak b)$ and $\bar\phi\in\IBr_\ell(\bar b)$ the $\ell$-Brauer characters associated with $\phi$.
Let $A:=A(\phi,Q)=\tilde A_{\phi'}/Z(G)$, where $\tilde A=\tilde J\rtimes E$ as above. 
Then by Lemma \ref{ext-ibr-d} and \ref{ext-wei-d}, conditions (1)-(3) of Definition \ref{induc}~(iii) hold.
For Definition \ref{induc}~(iii)(4), the proof of \cite[Lem.~7.2]{Fe18} applies here.
\end{proof}

\begin{appendices}

\section{Appendix A: Remarks on \cite{An94}}
\label{com-for-an94}

In line 35 of \cite[p.~33]{An94}, after ``Then $w\in X_0$ and so $\xi(w)=\xi^x(w)$", we can
only get $\xi_i(h)\xi_j(h)=\xi_i'(h)\xi_j'(h)$
but can not conclude $\xi_i(h)=\xi_i'(h)$,
nor does the claim $X_0\mathbf S_{\mathbf m}=K$ in line 29 of \cite[p.~33]{An94} follow.
For this reason, we give a new proof for \cite[(4C)]{An94}.

We will completely follow the notation in \cite{An94} and all references in this proof are to this paper  throughout Appendix \ref{com-for-an94}.

\subsection{For (4C)}
\subsubsection{}
Now we give some conventions for orthogonal cases first.
Let $\Gamma\in\cF_0$ and $$\mathscr C_{\Gamma,d}=\{\varphi_{\Gamma,d,i,j}\mid 1\le i\le 2e, 1\le j\le r^d\}$$ be the set as page 32.

Let $(R,\theta)$ be a pair of type $\Gamma$ and $R=R_{1,0,\gamma,\mathbf c}$~($\mathbf c=(c_1,\ldots,c_l)$) a basic subgroup such that $\gamma+c_1+\cdots+c_l=d$ .
Let $V$ be the underlying (orthogonal) space of $R$.
Then $|\Irr^0(N(\theta)\mid\theta)|=2e(r-1)^l$.
Let $N_0(\theta)=N(\theta)\cap \SO(V)$.
Then $|N(\theta):N_0(\theta)|=2$ and the restriction of each character of $\Irr^0(N(\theta)\mid\theta)$ to $N_0(\theta)$ is irreducible by the remark of (4A).
For $\varphi_1,\varphi_2\in N(\theta)$, we write $\varphi_1\sim\varphi_2$ if $\varphi_1|_{N_0(\theta)}=\varphi_2|_{N_0(\theta)}$.
Then if $\varphi_1\sim\varphi_2$ and $\varphi_1\ne\varphi_2$, then
\begin{equation}\label{tow-exts}
\varphi_1(n)=-\varphi_2(n) \ \text{for any}\ n\in N(\theta)\ \text{with determinant}\ -1.
\end{equation}

We may assume that
\begin{equation}
\varphi_{\Gamma,d,i,j}\sim \varphi_{\Gamma,d,e+i,j} \ \text{for all}\ 1\le i\le e, 1\le j\le r^d.
\end{equation}

Now we keep the notation and assumption preceding (4B).
Let $B$ be the block in (4B).
Then the proof of (4B) gives a bijection between $B$-weights and the assignments
\begin{equation}\label{bijection}
\coprod\limits_{d\ge0}\mathscr C_{\Gamma,d}\to \{r\text{-cores} \},\ \ \varphi_{\Gamma,d,i,j}\mapsto\kappa_{\Gamma,d,i,j}
\end{equation}
such that $\sum\limits_{d\ge0}r^d\sum\limits^{\beta_\Gamma e_\Gamma}_{i=1}\sum\limits^{r^d}_{j=1}|\kappa_{\Gamma,d,i,j}|=w_\Gamma$.

(4C) should be as follows.
\begin{lem}\label{(4C)}
	With the notation and hypothesis preceding (4C), let $G=O(V)$ be an orthogonal group, $G_0=\SO(V)$,
	and $R$ a radical subgroup of $G$ such that $[V, R] = V$.
	Let $(R, b)$ a Brauer pair of $G_0$ labeled by $(R, s, -)$ and $\theta$ the canonical character of $b$.
	Then $|N(\theta):N_0(\theta)|=\beta_\Gamma$.
	Moreover, if we write $N_0(\theta)=N(\theta)\cap G_0$, then the restriction $\psi|_{N_0(\theta)}$ of each $\psi\in\Irr^0(N(\theta),\theta)$ to $N_0(\theta)$ is irreducible unless when $\Gamma\in \cF_0$,
	$w_\Gamma$ is even
	and the associated assignment (\ref{bijection}) of $\psi$ satisfies that $\kappa_{\Gamma,d,i,j}=\kappa_{\Gamma,d,e+i,j}$ for all
	$1\le i\le e, 1\le j\le r^d$.
\end{lem}

\subsubsection{Proof}

\begin{lem}\label{num-in-res}
	Let $M$ be arbitrary finite group, $M_1, M_2\unlhd M$ with $M_1M_2=M$ and $L=M_1\cap M_2$.
	Let $\xi\in\Irr(M_1)$ such that $\varphi=\xi_{L}$ is irreducible.
	Let $\psi\in\Irr(M(\xi)\mid\xi)$ and $\eta=\psi|_{M(\xi)\cap M_2}$.
	Then $\eta$ is irreducible and the following statements hold.
	\begin{enumerate}[(i)]
		\item If $M_2(\eta)=M(\xi)\cap M_2$, then the restriction of $\Ind^M_{M(\xi)}(\psi)$ to $M_2$ is irreducible.
		\item If $|M_2(\eta):M(\xi)\cap M_2|=2$, then the restriction of $\Ind^M_{M(\xi)}(\psi)$ to $M_2$ is a sum of two irreducible characters.
		\item If $M_2(\varphi)=M(\xi)\cap M_2$, then the restriction defines a bijection from $\Irr(M\mid\xi)$ onto $\Irr(M_2\mid\varphi)$.
	\end{enumerate}
\end{lem}

\begin{proof}
	By \cite[Cor. (4.2)]{Is84}, the restriction defines a bijection from $\Irr(M(\xi)\mid\xi)$ onto $\Irr(M(\xi)\cap M_2\mid\varphi)$.
	So $\eta$ is irreducible.
	
	(i)
	By Mackey formula, $(\Ind_{M(\xi)}^{M}(\psi))|_{M_2}=\Ind_{M(\xi)\cap M_2}^{M_2}(\psi|_{M(\xi)\cap M_2})=\Ind_{M_2(\eta)}^{M_2}(\eta)$.
	Then (i) follows by Clifford theory.
	
	(ii) By Gallagher's theorem, $\Ind^{M_2(\eta)}_{M(\xi)\cap M_2}(\eta)=\tilde \eta+\tilde \eta'$, where $\tilde \eta$ and $\tilde \eta'$ are two extensions of $\eta$ to $M_2(\eta)$~(and $\tilde \eta\ne\tilde \eta'$).
	Thus $\Ind^{M_2}_{M(\xi)\cap M_2}(\eta)=\Ind^{M_2}_{M_2(\eta)}(\tilde \eta)+\Ind^{M_2}_{M_2(\eta)}(\tilde \eta')$ and both $\Ind^{M_2}_{M_2(\eta)}(\tilde \eta)$ and $\Ind^{M_2}_{M_2(\eta)}(\tilde \eta')$ are irreducible ~(and $\Ind^{M_2}_{M_2(\eta)}(\tilde \eta)\ne\Ind^{M_2}_{M_2(\eta)}(\tilde \eta')$) by Clifford theory.
	Then (ii) follows by Mackey formula.
	
	(iii) follows from (i) immediately.
\end{proof}

\begin{proof}[Proof of Lemma \ref{(4C)}]
	Note that the first and second paragraph of (4C) also apply here.
	Keep the notation in (4C) (and in the first and second paragraph of its proof).
	First, we have $|N(\theta):N_0(\theta)|=\beta_\Gamma$ and then we may assume that $\Gamma\in \cF_0$.
	Also, we suppose that $u=1$ and $d=d_1$.
	
	For $\psi\in\Irr^0(N(\theta),\theta)$, we write $\psi=\Ind^{H}_{X\mathbf S_{\mathbf m}}(\tilde\xi\chi)$ (as in the second paragraph of the proof of (4C)), where $\chi$ is some character of $X\mathbf S_{\mathbf m}$ trivial on $X$.
	Let $\xi_0=\xi|_{X_0}$ and let $\xi=\boxtimes_{k=1}^{d}\xi_k$, where $\xi_k\in\Irr(T)$.
	By the remark of (4A), $\xi_k|_{T_0}$ is irreducible for every $1\le k \le d$ and hence $\xi|_{T_0^d}$ is irreducible. So $\xi_0$ is irreducible.
	Let $K$ be the stabilizer of $\xi_0$ in $H_0$ and let $\tilde\xi_0=\tilde\xi|_{X_0\mathbf S_{\mathbf m}}$.
	Then $\tilde\xi_0$ is an extension of $\xi_0$ to $X_0\mathbf S_{\mathbf m}$.
	By Clifford theory, each irreducible character of $X_0\mathbf S_{\mathbf m}$ covering $\xi_0$ has the form $\tilde\xi_0\chi$, where $\chi$ is an irreducible character of $X_0\mathbf S_{\mathbf m}$ trivial on $X_0$.
	Then by Lemma \ref{num-in-res} (iii), $\psi|_{N_0(\theta)}$ is irreducible
	if $K= X_0\mathbf S_{\mathbf m}$.
	
	Now we assume that $K\ne X_0\mathbf S_{\mathbf m}$.
	Let $x\in K$ such that $x\notin X\mathbf S_{\mathbf m}=H(\xi)$ first.
	Then $d>1$. Also, we may assume that $x\in \mathbf S(d)$.
	Write $\xi':=\xi^x=\boxtimes_{k=1}^{d}{\xi'}_k$, where $\xi'_k\in\Irr(T)$.
	Then $\xi'\ne \xi$ and $\xi'|_{X_0}= \xi|_{X_0}=\xi_0$.
	Thus both $\xi'_k$ and $\xi_k$ are extensions of the irreducible character $\xi_k|_{T_0}$ to $T$ for every $1\le k \le d$~(\emph{i.e.}, $\xi'_k\sim\xi_k$).
	Let $1\le i,j\le n$ with $i\ne j$ and $h\in T$ with determinant $-1$ and $w=\mathrm{diag}\{w_1,\ldots,w_d\}$ such that $w_i=h=w_j$ and $w_k=1$ if $k\ne i,j$.
	Then $w\in X_0$ and so $\xi(w)=\xi'(w)$ and then $\xi_i(h)\xi_j(h)=\xi'_i(h)\xi'_j(h)$.
	Thus if $\xi'_i\ne \xi_i$, then $\xi'_j\ne \xi_j$ by (\ref{tow-exts}).
	Hence either $\xi'_k= \xi_k$ for all $1\le k \le d$
	or $\xi'_k\sim \xi_k$ and $\xi'_k\ne \xi_k$ for all $1\le k \le d$ holds.
	On the other hand, if this holds, it is easy to check that $K\ne X_0\mathbf S_{\mathbf m}$.
	Also, by the argument above, $|K:X_0\mathbf S_{\mathbf m}|\le 2$.
	
	Now we write $\xi=\boxtimes_{d,i,j}\varphi_{\Gamma,d,i,j}^{t_{\Gamma,d,i,j}}$.
	By the argument above, $K\ne X_0\mathbf S_{\mathbf m}$, if and only if
	\begin{equation}\label{con-sta}
	t_{\Gamma,d,i,j}=t_{\Gamma,d,e+i,j}\  \text{for all}\ 1\le i\le e, 1\le j\le r^d.
	\end{equation}
	This occurs only when $w_\Gamma$ is even.
	
	Now we assume that $K\ne X_0\mathbf S_{\mathbf m}$ and thus $|K:X_0\mathbf S_{\mathbf m}|=2$ and (\ref{con-sta}) holds.
	Let $\zeta=\tilde\xi\chi$ and $\zeta_0=\tilde\xi_0\chi$, then $\zeta_0=\zeta|_{X_0\mathbf S_{\mathbf m}}$.
	Then $X_0\mathbf S_{\mathbf m}\le H_0(\zeta_0)\le K$.
	Then by (i) and (ii) of Lemma \ref{num-in-res}, $\psi|_{N_0(\theta)}$ is irreducible if $H_0(\zeta_0)=X_0\mathbf S_{\mathbf m}$,  and $\psi|_{N_0(\theta)}$ is a sum of two irreducible characters if $H_0(\zeta_0)=K$.
	
	Now let $x\in K\setminus X_0\mathbf S_{\mathbf m}$ and without loss of generality we assume that $x\in \mathbf S(d)$.
	Now we may write $\xi=\prod_{i=1}^{s}(\xi_i^{t_i}\times {\xi'_i}^{t_i})$ with $\xi_i\sim\xi'_i$ and $\xi_i\ne\xi'_i$ for all $1\le i\le s$.
	Then $\xi^x=\prod_{i=1}^{s}({\xi'_i}^{t_i}\times {\xi_i}^{t_i})$.
	Note that the values of the extension of $\xi_i^{t_i}$ to $T\wr \mathbf S(t_i)$ on $\mathbf S(t_i)$ only depend on $\xi_i(1)$~(see, for example, \cite[Lem.~25.5]{Hu98}), \emph{i.e.}, there exists extension $\eta_i$ (resp. $\eta'_i$) of ${\xi}_i^{t_i}$ (resp. ${\xi'}_i^{t_i}$) to $T^{t_i}\rtimes \mathbf S(t_i)$ such that $\eta_i|_{\mathbf S(t_i)}=\eta'_i|_{\mathbf S(t_i)}$.
	So we may assume that $\tilde\xi^x|_{\mathbf S_{\mathbf m}}=\tilde\xi|_{\mathbf S_{\mathbf m}}$.
	On the other hand, $\tilde\xi|_{X_0}=\xi_0$,
	then
	$\tilde\xi^x|_{X_0}=\tilde\xi|_{X_0}$, and
	then $\tilde\xi^x|_{X_0\mathbf S_{\mathbf m}}=\tilde\xi|_{X_0\mathbf S_{\mathbf m}}$.
	Thus $\tilde\xi_0^x=\tilde\xi_0$.
	Hence
	$$H_0(\zeta_0)=K  \Leftrightarrow \tilde\xi_0\chi=(\tilde\xi_0\chi)^x=\tilde\xi_0\chi^x
	\Leftrightarrow \chi=\chi^x.$$
	So $H_0(\zeta_0)=K$ if and only if
	$\kappa_{\Gamma,d,i,j}=\kappa_{\Gamma,d,e+i,j}$ for all
	$1\le i\le e, 1\le j\le r^d$.
	This completes the proof.		
\end{proof}

\subsection{(4E) and (4H)}\label{app.4E4H}

The remark of (4E) and (4H) used (4C). Thus we give some remarks here.

\subsubsection{The remark of (4E)}\label{(4E)}
By Lemma \ref{(4C)}, the remark of (4E) should be stated as follows.

With the assumption of (4E), we have a bijection between $B$-weights and $\prod_\Gamma \mathscr T_\Gamma$, where $\mathscr T_\Gamma$ is the set of $\beta_\Gamma e_\Gamma$-tuples $(\kappa_1,\kappa_2,\ldots,\kappa_{\beta_\Gamma e_\Gamma})$ of partitions $k_i$ such that $\sum\limits_{i=1}^{\beta_\Gamma e_\Gamma}|\kappa_i|=w_\Gamma$ (by the proof of (4E).

Let $G = O(V)$, $G_0 = \SO(V)$, $(R,\varphi)$ a $B$-weight of $G$, and $\theta$ an irreducible character of $C = C_G(R)$ covered by $\varphi$.
Then $|N(\theta):N_0(\theta)|=1$ or $2$ according as $m_{X\pm 1}=0$ or $m_{X\pm 1}\ne0$.
Moreover, for each $\psi\in\Irr^0(N(\theta)\mid \theta)$, the restriction $\psi|_{N(\theta)\cap G_0}$ is irreducible unless when
$m_{X\pm 1}\ne0$, $w_{X\pm 1}$'s are even and the element $\kappa=\prod_\Gamma \kappa_\Gamma$~(with $\kappa_\Gamma=(\kappa^\Gamma_1,\kappa^\Gamma_2,\ldots,\kappa^\Gamma_{\beta_\Gamma e_\Gamma})$)~ in $\prod_\Gamma \mathscr T_\Gamma$  corresponding to $\psi$ satisfies that
$\kappa^\Gamma_i=\kappa^\Gamma_{e+i}$ for every $\Gamma\in\cF_0$ and every $1\le i\le e$.

\subsubsection{(4H)}\label{(4H)}

For $\Gamma\in\cF$, we recall that the integer $f_\Gamma$ is defined to be the number of $\beta_\Gamma e_\Gamma$-tuples $(\kappa_1,\ldots,\kappa_{\beta_\Gamma e_\Gamma})$ of partitions such that $\sum\limits_{i=1}^{\beta_\Gamma e_\Gamma}|\kappa_i|=w_\Gamma$.

If $\Gamma\in \cF_0$ and $w_\Gamma$ is even, we define $f_\Gamma'$ to be the number of $e$-tuples $(\kappa_1,\ldots,\kappa_{e})$ of partitions such that $\sum\limits_{i=1}^{e}|\kappa_i|=\frac{1}{2}w_\Gamma$.

The conclusion of (4E) should be as follows.
\begin{enumerate}[(1)]
	\item The number of $B$-weights is $\prod\limits_\Gamma f_\Gamma$ if one of the following statements holds:
	\begin{enumerate}[(a)]
		\item $m_{X\pm 1}(s_+)=0$,
		\item $\vartheta_0^{\sigma_0}=\vartheta_0$ for some $\sigma_0\in O(V_0)$ of determinant $-1$,
	\end{enumerate}
	\item Suppose that $m_{X\pm 1}(s_+)\ne0$ and either $V_0=0$ or $\vartheta_0^{\sigma_0}\ne\vartheta_0$ for any $\sigma_0\in O(V_0)$  of determinant $-1$.
	\begin{enumerate}[(a)]
		\item If either $w_{X-1}$ or $w_{X+1}$ is odd,
		then the number of $B$-weights is $\frac{1}{2} \prod_\Gamma f_\Gamma$.
		\item If both $w_{X-1}$ and $w_{X+1}$ are even, then the number of $B$-weights is $\frac{f_{X-1}f_{X+1}+3f'_{X-1}f'_{X+1}}{2} \prod\limits_{\Gamma\notin \cF_0} f_\Gamma.$
	\end{enumerate}	
\end{enumerate}

\section{Appendix B: The blocks of special orthogonal groups in even-dimension}\label{Appendix-blocks-so}

Let $G=\SO_{2n}^\eps (q)$ with $\eps\in\{\pm\}$, $q$ odd and $n\ge 4$.
Now we give a classification for $\ell$-blocks of $G$ for an odd prime $\ell\nmid q$, which is completely analogous with the case of $\Sp_{2n}(q)$ in \S\ref{The-blocks-of-special-symplectic-groups}.
Let $V$ be the underlying space of $G$ and $\tG=\CSO(V)$.
Then $\tG^*$ is the special Clifford group over $V^*$ and $G^*=G$.
Let $\pi:\tG^*\to G^*$ be the natural epimorphism.
As usual, we let $\bG=\SO_{2n}^\eps (\overline\F_{q})$ and $\tilde\bG=\CSO_{2n}^\eps (\overline\F_{q})$ for the corresponding algebraic groups and $F$ the Frobenius endomorphism.

Note that the notation of special conformal orthogonal groups and general orthogonal groups are not the same with those in \S\ref{typeD}, since the relations of $\ell$-blocks of $\SO_{2n}^\eps(q)$ and $\CSO_{2n}^\eps(q)$ is similar with the relations of $\ell$-blocks of $\Sp_{2n}(q)$ and $\CSp_{2n}(q)$.
So we use the notation which is analogous to the one in \S\ref{The-blocks-of-special-symplectic-groups}.

The irreducible characters of $\tG$ have been classified by Lusztig \cite{Lu77}.
For a semisimple element $s$ of $\tG^*$,
we write $\bar s=\pi(s)$.
Note that both $m_{x-1}(\bar s)$ and $m_{x+1}(\bar s)$ are even.
Let $i\Irr(\tilde G)$  be the set of $\tG^*$-conjugacy classes of pairs $(s,\mu)$, where $s$ is a semisimple $\ell'$-element of $\tG^*$ and $\mu\in\Psi(\bar s)$~(where $\Psi(\bar s)$ is defined as (\ref{def-par-sym})).
Here, $(s,\mu)$ and $(s',\mu')$ are conjugate if and only if $s$ and $s'$ are $\tG^*$-conjugate and $\mu=\mu'$.
By Jordan decomposition of characters, there is a bijection from $i\Irr(\tilde G)$ onto $\Irr(\tilde G)$~(see also \cite[(4.5)]{FS89}).
We write $\tilde\chi_{s,\mu}$ for the character of $\tG$ corresponding to $(s,\mu)$.

If $\chi\in\Irr(G\mid \tilde\chi_{s,\mu})$, then we say $\chi$ corresponds to the pair $(\bar s,\mu)$.
Furthermore, we have the following result about the characters of $G$, which is similar to Lemma \ref{res-sp}.

\begin{lem}\label{res-so}
	Let $(s,\mu)\in i\Irr(\tG)$, $\bar s=\pi(s)$
	and $\tilde\chi=\tilde\chi_{s,\mu}$.
	\begin{enumerate}[(i)]
		\item  If $1$ or $-1$ is not an eigenvalue of $\bar s$,
		then $\Res^{\tG}_G\tilde\chi$ is irreducible.
		\item If both $1$ and $-1$ are eigenvalues of $\bar s$, then
		\begin{itemize}
			\item if $\mu_{x-1}$ or $\mu_{x+1}$ is degenerate, then $\Res^{\tG}_G\tilde\chi$ is irreducible, and
			\item if both $\mu_{x-1}$ and $\mu_{x+1}$ are non-degenerate, then $\Res^{\tG}_G\tilde\chi$ is a sum of two irreducible constituents.
		\end{itemize}
	\end{enumerate}
\end{lem}

Let $I=\GO_{2n}^\eps(q)$.
We recall the action of $I$ on $\Irr(G)$ which was given in \cite[(4D)]{FS89}.
Let $(s,\mu)\in i\Irr(\tG)$, $\tilde\chi=\tilde\chi_{s,\mu}$ and $g\in I$ of determinant $-1$.
Then $\tilde\chi^g$ corresponds to the pair $(g^* s{g^*}^{-1},\mu^*)$~(the operator $^*$ is defined as in \cite[\S2,~p.~132]{FS89}). More precisely, we have \begin{enumerate}[(a)]
	\item $\tilde\chi^g=\tilde\chi$ if
	\begin{itemize}
		\item $1$ is an eigenvalue of $\bar s$, $-1$ is not an eigenvalue of $\bar s$, and $\mu_{x-1}$ is non-degenerate, or
		\item both $1$ and $-1$ are eigenvalues of $\bar s$ and $\mu_{x-1}$ is non-degenerate, and
	\end{itemize}
	\item $\tilde\chi^g\ne\tilde\chi$ and $\Res^\tG_G(\tilde\chi^g)=\Res^\tG_G(\tilde\chi)$ if
	\begin{itemize}
		\item $1$ is not an eigenvalue of $\bar s$, $-1$ is an eigenvalue of $\bar s$, and $\mu_{x+1}$ is non-degenerate, or
		\item both $1$ and $-1$ are eigenvalues of $\bar s$ and $\mu_{x-1}$ is degenerate and $\mu_{x+1}$ is non-degenerate, and
	\end{itemize}
	\item $\tilde\chi^g\ne\tilde\chi$ and $\Res^\tG_G(\tilde\chi^g)\ne\Res^\tG_G(\tilde\chi)$ if
	\begin{itemize}
		\item both $1$ and $-1$ are not eigenvalues of $\bar s$, or
        \item $1$ is an eigenvalue of $\bar s$, $-1$ is not an eigenvalue of $\bar s$, and $\mu_{x-1}$ is degenerate, or
        \item $1$ is not an eigenvalue of $\bar s$, $-1$ is an eigenvalue of $\bar s$, and $\mu_{x+1}$ is degenerate, or
		\item both $1$ and $-1$ are eigenvalues of $\bar s$ and both $\mu_{x-1}$ and $\mu_{x+1}$ are degenerate.
	\end{itemize}
\end{enumerate}
Now let $\chi\in\Irr(G\mid\tilde\chi)$.
By Lemma \ref{res-so}, if we are in case (b) or (c),
then $\chi=\Res^\tG_G\tilde\chi$.
Moreover, $\chi$ is $I$-invariant in case (b) and $\chi$ is not $I$-invariant in case (c).
If $1$ is an eigenvalue of $\bar s$ and $-1$ is not an eigenvalue of $\bar s$ or both $1$ and $-1$ are eigenvalues of $\bar s$ and $\mu_{x-1}$ is non-degenerate and $\mu_{x+1}$ is degenerate, then we also have  $\chi=\Res^\tG_G\tilde\chi$ and thus $\chi$ is $I$-invariant.
If both $1$ and $-1$ are eigenvalues of $\bar s$ and both $\mu_{x-1}$ and $\mu_{x+1}$ are non-degenerate,
then  $\Res^\tG_G\tilde\chi$ is a sum of two irreducible constituents, and then
$\tilde\chi$ is $I$-invariant but we do not know whether $\chi$ is $I$-invariant or not in this case now.

\vspace{2ex}

Now we recall the classification of $\ell$-blocks of $\tilde\bG^F$ given in \cite[\S11]{FS89}.
Let $i\Bl_\ell(\tG)$ be the set of $\tG^*$-conjugacy classes of pairs $(s,\kappa)$ where $s$ is a semisimple $\ell'$-element of $\tG^*$ and $\kappa\in\cC(\bar s)$, where $\bar s=\pi(s)$ and $\cC(\bar s)$ is defined as (\ref{def-core}).
Here, $(s,\kappa)$ and $(s',\kappa')$ are $\tG^*$-conjugate if and only if $s$ and $s'$ are $\tG^*$-conjugate and $\kappa=\kappa'$.
Also note that both $m_{x-1}(\bar s)$ and $m_{x+1}(\bar s)$ are even and $|\kappa|=1,2$, or $4$.
By \cite[(11E)]{FS89},
there is a bijection $(s,\kappa)\mapsto \tilde B(s,\kappa)$ from $i\Bl_\ell(\tG)$ to $\Bl_\ell(\tG)$.

For $(s,\kappa)\in i\Bl_\ell(\tG)$,
\cite[(13B)]{FS89} also gave a criterion that when an irreducible character of $\tG$ lies in the $\ell$-block $\tilde B=\tilde B(\bar s,\kappa)$.
In particular, the irreducible characters of $\Irr(\tilde B)\cap\cE(\tG,\ell')$ are of form $\tilde\chi_{s,\mu}$ with $\mu\in\Psi(\bar s,\kappa)$~(where $\Psi(\bar s,\kappa)$ is defined as in (\ref{def-par-sym-core})).
In addition, by  (\ref{bijection-sym-skew-par}), $\Psi_\Gamma(\bar s,\kappa)$ is in bijection with $\mathcal P(\beta_\Gamma e_\Gamma, w_\Gamma)$
if $\Gamma\in\cF_1\cup\cF_2$ or $\Gamma\in \cF_0$ and $\kappa_\Gamma$ is non-degenerate and in bijection with
$\mathcal P'(2 e, w_\Gamma)$ if $\Gamma\in\cF_0$ and $\kappa_\Gamma$ is degenerate.
Here, the sets $\mathcal P(\beta_\Gamma e_\Gamma, w_\Gamma)$ and $\mathcal P'(2 e, w_\Gamma)$ are defined as in (\ref{def-set-lab}) and (\ref{def-p'}) respectively.

Let ${\mathfrak e}$ be the identity element of the Clifford algebra over $V^*$.
With the similar argument with \S\ref{The-blocks-of-special-symplectic-groups},
we may construct the $e_0$-Jordan-cuspidal pair for an $\ell$-block $\tilde B(s,\kappa)$ of $\tG=\CSO_{2n}^\eps (q)$,
which is completely analogous with the case of $\tilde G=\CSp_{2n}(q)$ and then
we have the following result which is completely analogous with Theorem \ref{blocks-sp}.

\begin{thm}\label{blocks-so}
	Let $(s,\kappa)\in i\Bl_\ell(\tilde G)$, $\bar s=\pi(s)$,
	$\tilde B=\tilde B(s,\kappa)$ and $B$ an $\ell$-block of $G$ covered by $\tilde B$.
	\begin{enumerate}[(i)]
		\item  If $1$ or $-1$ is not eigenvalue of $\bar s$,
		then  $B$ is the unique $\ell$-block of $G$ covered by $\tilde B$ and there are $|\mathcal O_{\ell'}(\F_q^\times)|$ $\ell$-blocks of $\tG$ covering $B$.
		In addition, the $\ell$-blocks covering $B$ of $\tilde G$ are $\tilde B(zs,\kappa)$, where $z$ runs through $\mathcal O_{\ell'}(Z(\tilde G^*))$.
		\item  If both $1$ and $-1$ are eigenvalues of $\bar s$,  and there exists $\Gamma\in\cF_0$ such that $w_\Gamma=0$ and $\kappa_\Gamma$ is degenerate,
		then  $B$ is the unique $\ell$-block of $G$ covered by $\tilde B$ and there are $|\mathcal O_{\ell'}(\F_q^\times)|$ $\ell$-blocks of $\tG$ covering $B$.
		In addition, the $\ell$-blocks covering $B$ of $\tilde G$ are $\tilde B(zs,\kappa)$ and $\tilde B(zs,\kappa')$, where $z$ runs through a complete set of representatives of $\langle -{\mathfrak e} \rangle$-cosets in $\mathcal O_{\ell'}(Z(\tilde G^*))$.	
		\item Suppose that both $1$ and $-1$ are eigenvalues of $\bar s$ and $w_\Gamma\ne 0$ if $\Gamma\in \cF_0$ and $\kappa_\Gamma$ is degenerate.		
		Then there are $\frac{1}{2}|\mathcal O_{\ell'}(\F_q^\times)|$ $\ell$-blocks of $\tG$ covering $B$ and they are $\tilde B(zs,\kappa)$, where $z$ runs through a complete set of representatives of $\langle -{\mathfrak e} \rangle$-cosets in $\mathcal O_{\ell'}(Z(\tilde G^*))$.	
		Moreover,
		\begin{itemize}
			\item if $\kappa_{x-1}$ or $\kappa_{x+1}$ is degenerate, then $B$ is the unique $\ell$-block of $G$ covered by  $\tilde B$, and
			\item if  both $\kappa_{x-1}$ and $\kappa_{x+1}$ are non-degenerate, then there are two $\ell$-blocks of $G$ covered by  $\tilde B$.
			
		\end{itemize}
	\end{enumerate}
\end{thm}

Now let $i\Bl^{(1)}_\ell(G)$ be the set of $G$-conjugacy classes of pairs $(s,\kappa)$, where $s\in G^*$ is a semisimple $\ell'$-element and $\kappa\in\cC(s)$ such that  either
(1) $1$ or $-1$ is not eigenvalue of $s$ or
(2) $\kappa_{x-1}$ or $\kappa_{x+1}$ is degenerate.
Here, we identify $(s,\kappa)$ with $(s,\kappa')$.
Let $i\Bl^{(2)}_\ell(G)$ be the set of $G^*$-conjugacy classes of pairs $(s,\kappa)$,
where $s\in G^*$ is a semisimple $\ell'$-element  and $\kappa\in\cC(s)$ such that both $1$ and $-1$ are eigenvalues of $s$ and both $\kappa_{x-1}$ and $\kappa_{x+1}$ are non-degenerate.
Then $i\Bl_\ell(G):=i\Bl^{(1)}_\ell(G)\cup i\Bl^{(2)}_\ell(G)$, where the elements of $i\Bl^{(2)}_\ell(G)$ counting twice, is a labeling set for $\Bl_\ell(G)$ by Theorem \ref{blocks-so}.

Let $i\IBr^{(1)}_\ell(G)$ be the set of $G^*$-conjugacy classes of pairs $(s,\mu)$, where $s\in G^*$ is a semisimple $\ell'$-element and $\mu\in\Psi(s)$ such that either $1$ or $-1$ is not an eigenvalue of $s$ or  both $1$ and $-1$ are eigenvalues of $s$ and $\mu_{x-1}$ or $\mu_{x+1}$ is degenerate.
Let $i\IBr^{(2)}_\ell(G)$ be the set of $G^*$-conjugacy classes of pairs $(s,\mu)$,
where $s\in G^*$ is a semisimple $\ell'$-element  and $\mu\in\Psi(s)$ such that both $1$ and $-1$ are eigenvalues of $s$ and both $\mu_{x-1}$ and $\mu_{x+1}$ are non-degenerate.
Then $i\IBr_\ell(G):=i\IBr^{(1)}_\ell(G)\cup i\IBr^{(2)}_\ell(G)$, where the elements of $i\IBr^{(2)}_\ell(G)$ counting twice, is a labeling set for $\Irr(G)\cap \cE(G,\ell')$ by Lemma \ref{res-so}.
If $(s,\mu)\in i\IBr^{(1)}_\ell(G)$, we denote by $\chi_{s,\mu}$ the character of $G$ corresponding to $(s,\mu)$.
If $(s,\mu)\in i\IBr^{(2)}_\ell(G)$,
then $\chi^{(1)}_{s,\mu}$ and $\chi^{(-1)}_{s,\mu}$ denote the two characters of $G$ corresponding to $(s,\mu)$.

Furthermore,
if $(s,\kappa)\in i\Bl^{(1)}_\ell(G)$ and $B=B(s,\kappa)$, then $\Irr(B)\cap\cE(G,s)=
\{\chi_{s,\mu}\mid \mu\in\Psi(s,\kappa),(s,\mu)\in i\IBr^{(1)}_\ell(G)   \} \cup \{\chi^{(\pm 1)}_{s,\mu}\mid \mu\in\Psi(s,\kappa),(s,\mu)\in i\IBr^{(2)}_\ell(G)  \}.$
If $(s,\kappa)\in i\Bl^{(2)}_\ell(G)$, $B^{(1)}=B^{(1)}(s,\kappa)$ and $B^{(-1)}=B^{(-1)}(s,\kappa)$, then
$\Irr(B^{(1)}\cup B^{(-1)})\cap\cE(G,s)=\{ \chi^{(\pm 1)}_{s,\mu}\mid \mu\in\Psi(s,\kappa) \}$.
We may assume that
$\Irr(B^{(i)})\cap\cE(G,s)=\{\chi^{(i)}_{s,\mu}\mid \mu\in\Psi(s,\kappa) \}$
for $i=\pm1$.
Note that, if both $1$ and $-1$ are eigenvalue of $s$, and
$\kappa_{x-1}$ and $\kappa_{x+1}$ are non-degenerate, then both $\mu_{x-1}$ and $\mu_{x+1}$ are also non-degenerate,
and then we always have $(s,\mu)\in i\IBr^{(2)}_\ell(G)$.

Now we give a labeling set for $\Irr(B)\cap \cE(G,\ell')$.
First, we define a set $\mathcal P(2e,w_1,w_2):=\mathcal P(2e,w_1)\times \mathcal P(2e,w_2)$ for integers $e\ge 1$, $w_1,w_2\ge 0$,
where $\mathcal P(2e,w_1)$ and $\mathcal P(2e,w_2)$ are defined as in (\ref{def-set-lab}).
Now we define
$\mathcal P_0(2e,w_{1},w_{2}):=\mathcal P_0(2e,w_{1})\times \mathcal P_0(2e,w_{2})$,
where $\mathcal P_0(2e,w_{1})$ and $\mathcal P_0(2e,w_{2})$ are defined as in
(\ref{par-P0}).
First we define an equivalent relation on the set $\mathcal P(2e,w_{1},w_{2})$.
For $\mu^{(k)}=\mu^{(k,1)}\times \mu^{(k,2)}$, where
$\mu^{(k,1)}=(\mu^{(k,1)}_1,\ldots,\mu^{(k,1)}_{2e})\in \mathcal P(2e,w_{1})$,
$\mu^{(k,2)}=(\mu^{(k,2)}_1,\ldots,\mu^{(k,2)}_{2e})\in \mathcal P(2e,w_{2})$ and $k=1,2$,
we let $\mu^{(1)}\sim \mu^{(2)}$ if $\mu^{(1,1)}\sim \mu^{(2,1)}$ and $\mu^{(1,2)}\sim \mu^{(2,2)}$ in the sense of (\ref{equ-par}).
Then we define $\mathcal P_1'(2e,w_{1},w_{2})=(\mathcal P(2e,w_{1},w_{2})\setminus \mathcal P_0(2e,w_{1},w_{2})))/\sim$.
Let $\mathcal P'(2e,w_{1},w_{2}):=\mathcal P_1'(2e,w_{1},w_{2})\cup \mathcal P_0(2e,w_{1},w_{2})$, where the elements of $\mathcal P_0(2e,w_{1},w_{2})$ are counted twice.

Define $\mathcal P(B):=$
\begin{enumerate}[(a)]
\item $\prod\limits_\Gamma \mathcal P(\beta_\Gamma e_\Gamma,w_\Gamma)$ if one of the following holds,
\begin{itemize}
	\item $w_{x-1}=w_{x+1}=0$, or
	\item  $\kappa_{x-1}$ or $\kappa_{x+1}$ is non-degenerate,
\end{itemize}
\item $\mathcal P'(2e,w_{\Gamma_0}) \times\prod\limits_{\Gamma\ne\Gamma_0} \mathcal P(\beta_\Gamma e_\Gamma,w_\Gamma)$
if both $\kappa_{x-1}$ and $\kappa_{x+1}$ are degenerate and there exists a unique $\Gamma_0\in\cF_0$ such that $w_{\Gamma_0}$ is odd, where $\mathcal P'(2e,w_{\Gamma_0})$ is defined as in (\ref{def-p'}),
\item $\mathcal P'(2e,w_{x-1},w_{x+1})\times \prod\limits_{\Gamma\notin\cF_0}\mathcal P(e_\Gamma,w_\Gamma)$, if one of the following holds,
\begin{itemize}
	\item both $\kappa_{x-1}$ and $\kappa_{x+1}$ are degenerate and both $w_{x-1}$ and $w_{x+1}$ are  odd, or
	\item  $w_{x-1}$ or $w_{x+1}$ is non-zero, both $\kappa_{x-1}$ and $\kappa_{x+1}$ are degenerate and both $w_{x-1}$ and $w_{x+1}$ are even.
\end{itemize}
\end{enumerate}

Similar with Proposition \ref{label-ibr-b}, we have

\begin{prop}\label{label-ibr-d}
	With the preceding notation,
	$\mathcal P(B)$ is a labeling set for $\Irr(B)\cap \cE(G,\ell')$.
\end{prop}

Let $f_\Gamma$ and $f'_\Gamma$ be defined as in Appendix \ref{(4H)}.
Then $f_\Gamma=|\mathcal P(\beta_\Gamma e_\Gamma,w_\Gamma)|$ and
$f'_\Gamma=|\mathcal P_0(\beta_\Gamma e_\Gamma,w_\Gamma)|$.
We end the appendix by giving the number of irreducible $\ell$-Brauer characters in an $\ell$-block of $G=\SO_{2n}^\pm(q)$, which follows
by Proposition \ref{label-ibr-d} immediately.

\begin{thm}\label{num-ibr}
Let $B$ be an $\ell$-block corresponding to $(s,\kappa)\in i\Bl_\ell(G)$ and $l(B)=|\IBr_\ell(B)|$.
\begin{enumerate}[(i)]
\item $l(B)=\prod_\Gamma f_\Gamma$ if one of the following statements holds.
\begin{itemize}
	\item $w_{x-1}=w_{x+1}=0$.
	\item  $\kappa_{x-1}$ or $\kappa_{x+1}$ is non-degenerate.
\end{itemize}
\item Suppose that $w_{x-1}$ or $w_{x+1}$ is non-zero and
both $\kappa_{x-1}$ and $\kappa_{x+1}$ are degenerate.
\begin{itemize}
\item  If either $w_{x-1}$ or $w_{x+1}$ is odd, then $l(B)=\frac{1}{2}\prod_\Gamma f_\Gamma$.
\item If both $w_{x-1}$ and $w_{x+1}$ are even, then $l(B)=\frac{f_{x-1}f_{x+1}+3f'_{x-1}f'_{x+1}}{2} \prod\limits_{\Gamma\notin \cF_0} f_\Gamma$.
\end{itemize}
\end{enumerate}
\end{thm}

Now we consider the Alperin weight conjecture for $G=\SO_{2n}^\pm(q)$.

\begin{prop}\label{alp-conj-case}
Let $B$ be an $\ell$-block corresponding to $(s,\kappa)\in i\Bl_\ell(G)$.
If
\begin{itemize}
	\item $1$ or $-1$ is not an eigenvalue of $s$, or
	\item $\kappa_{x-1}$ or $\kappa_{x+1}$ is degenerate,
\end{itemize}
then the Alperin weight conjecture \ref{weiconj} holds for $B$, \it{i.e.}, $l(B)=|\mathcal W_\ell(B)|$.
\end{prop}

\begin{proof}
Let $V_0$, $\vartheta_0$ and $s_0$ be defined as in \cite[(4H)]{An94}.	
Then by the proof of \cite[(4H)]{An94},
$\vartheta_0\in G_0$ and $\tilde\vartheta_0\in\Irr(\tilde G_0)$ with
$G_0=\SO(V_0)$ and
$\tilde G_0=\CSO(V_0)$ 	
such that $\tilde\vartheta_0=\tilde \chi_{t_0,\kappa}$,
where $t_0\in \tilde G_0^*$ satisfies that $s_0=\pi(t_0)$.
Then the assertion follows by  \S\ref{(4H)} and Theorem \ref{num-ibr} and the criterion for when $\vartheta_0$ is $\GO(V_0)$-invariant given in the statements after Lemma \ref{res-so}.
\end{proof}

Now we consider  the following properties about the action of $I=\GO_{2n}^\eps(q)$ on the characters of $G$.

\begin{itemize}
	\item[($\dagger$)] Let $\tilde\chi\in\Irr(\tG)$ and $\Delta=\Irr(G\mid\tilde\chi)$.
	Then $I_{\Delta}$ acts trivially on $\Delta$.
	\item[($\ddagger$)] Let $s$ be a semisimple element of $G$, $\mu\in\Psi(s)$
	and $\chi$ be a character of $G$ corresponding to $(s,\mu)$.
	Suppose that both $1$ and $-1$ are eigenvalues of $s$ and
	 both $\mu_{x-1}$ and $\mu_{x+1}$ are non-degenerate.
	Then $\chi$ is $I$-invariant.
\end{itemize}

Then by the proof of Proposition \ref{alp-conj-case} and the statements after Lemma \ref{res-so}, we have
\begin{thm}\label{alp-conj-SO2n}
\begin{enumerate}[(i)]
\item The Alperin weight conjecture \ref{weiconj} holds for every $\ell$-block of the special orthogonal group $G=\SO_{2n}^\eps(q)$ with every $n\ge 4$, odd $q$ and $\eps=\pm$ if ($\dagger$) is true for the special orthogonal group $G=\SO_{2n}^\eps(q)$ with every $n\ge 4$, odd $q$ and $\eps=\pm$.
\item ($\dagger$) holds if and only if ($\ddagger$) holds.
\end{enumerate}
\end{thm}

\end{appendices}

\section*{Acknowledgements}

The first author wishes to thank Gunter Malle for fruitful conversations, vital information and useful indications.
The authors thank him for commenting on an earlier version of this paper.


\end{document}